\documentclass[a4paper,oneside,10pt]{article}

\usepackage[utf8]{inputenc}  
\usepackage{nicefrac}
\usepackage{amsmath,amssymb}
\usepackage{dsfont}
\usepackage{array}
\usepackage{booktabs}
\usepackage{graphicx}
\usepackage{bbm}   
\usepackage{comment}
\usepackage{url}
\usepackage{nameref}
\usepackage[colorlinks, linkcolor = black, citecolor = black, filecolor = black, urlcolor = blue]{hyperref} 
\usepackage{ifthen}
\usepackage{natbib}
\usepackage{color}
\usepackage{epstopdf}
\usepackage{caption}
\usepackage{subcaption}
\usepackage{arydshln}
\usepackage{longtable}
\usepackage{marvosym}
\usepackage{cases}
\usepackage{enumitem}
\usepackage{amsthm}

\usepackage[a4paper, total={6in, 8in}]{geometry}
\usepackage{pgfplots}
\usetikzlibrary{arrows.meta}
\usepackage{tikz}

\theoremstyle{plain}
\newtheorem{theorem}{Theorem}
\newtheorem{lemma}{Lemma}

\newtheorem{corollary}[theorem]{Corollary}
\newtheorem{assumption}{Assumption}
\theoremstyle{remark}
\newtheorem{remark}[lemma]{Remark}

\newtheorem{example}[lemma]{Example}
\newtheorem*{remark*}{Remark}

\newcommand{\wjk}[3]{w_{j,k}({#1},{#2};{#3})} 

\newcommand{\wijk}[3]{w_{i;j,k}({#1},{#2};{#3})} 

\newcommand{\Zi}[2]{Z_{i;{#1},{#2}}^{\otimes2}}
\newcommand{\PclassCov}{\mc P(\gamma, C_Z, \beta, L)}
\newcommand{\Zclass}[0]{\mc P(\gamma)}

\newcommand{\covest}[3]{\hat \Gamma_{n}({#1}, {#2};{#3})}

\newcommand{\Kh}[2]{K_{h}\left(\begin{array}{c}{#1}\\{#2}\end{array}\right)}
\newcommand{\Uh}[2]{U_{m,h}\left(\begin{array}{c}{#1}\\{#2}\end{array}\right)}
\newcommand{\Cmax}{C_1}
\newcommand{\Clip}{C_2}
\newcommand{\Ccard}{C_3}
\newcommand{\Csum}{C_4}

\newcommand{\mc}{\mathcal}
\newcommand{\dx}{\mathrm{d}}

\newcommand{\bs}[1]{\boldsymbol{#1}} 
\newcommand{\Quad}{\qquad \quad} 
\newcommand{\QQuad}{\qquad \qquad}

\newcommand{\expec}{{\mathbb{E}}}
\newcommand{\prob}{{\mathbb{P}}}
\newcommand{\var}{{\operatorname{\mathbb{V}\mathrm{ar}}}}
\newcommand{\cov}{{\operatorname{\mathbb{C}\mathrm{ov}}}}

\newcommand{\ind}{\text{\bf{1}}} 
\newcommand{\one}{\mathbbm{1}}

\newcommand{\supp}{{\operatorname{supp}}}
\newcommand{\floor}[1]{\left\lfloor#1\right\rfloor}
\newcommand{\ceil}[1]{\left\lceil#1\right\rceil}
\newcommand{\abs}[1]{\lvert #1 \rvert}
\newcommand{\KL}{\operatorname{KL}}
\newcommand{\diam}{\operatorname{diam}}
\newcommand{\card}{\operatorname{card}}
\newcommand{\tr}{{\operatorname{trace}}}
\DeclareMathOperator*{\argmin}{argmin}

\newcommand{\subG}{{\operatorname{subG}}} 

\newcommand{\Oop}[1]{{\operatorname{\mathcal{O}}\left(#1\right)}}
\newcommand{\norm}[1]{{\lVert#1\rVert}} 
\newcommand{\normb}[1]{{\big\lVert#1\big\rVert}}

\newcommand{\normm}[1]{{\bigg\lVert#1\bigg\rVert}}

\newcommand{\skpb}[2]{{\big\langle#1,#2\big\rangle}}  
\newcommand{\absb}[1]{\big|#1\big|}
\newcommand{\Abs}[1]{\Big|#1\Big|}
\newcommand{\abss}[1]{\bigg|#1\bigg|}

\newcommand{\absA}[1]{\left|#1\right|}
\newcommand{\vecTwo}[2]{\left(\begin{array}{c}#1\\#2\end{array}\right)}

\newcommand{\R}{{\mathbb{R}}}
\newcommand{\N}{{\mathbb{N}}}

\newcommand*{\defeq}{\mathrel{\vcenter{\baselineskip0.5ex \lineskiplimit0pt
			\hbox{\scriptsize.}\hbox{\scriptsize.}}}%
	=}
\newcommand*{\defeql}{ = \mathrel{\vcenter{\baselineskip0.5ex   	\lineskiplimit0pt
			\hbox{\scriptsize.}\hbox{\scriptsize.}}}%
}

\newcommand{\lessneqsim}{\raisebox{-0.15cm}{~\shortstack{$\ll$ \\[-0.05cm]
			$\sim$}}~}

\newcommand{\wjks}[4]{w_{j,k}^{({\bs{#4}})}({#1},{#2};{#3})}

\newcommand{\cO}{\ensuremath{\mathcal{O}}} %

\title{Optimal rates for estimating the covariance kernel from synchronously sampled functional data}

\author{Max Berger and Hajo Holzmann\footnote{Corresponding author. Prof.~Dr.~Hajo Holzmann, Department of Mathematics and Computer Science, Philipps-Universität Marburg, Hans-Meerweinstr., 35043 Marburg, Germany} \\
\small{Department of Mathematics and Computer Science}  \\
\small{Philipps-Universität Marburg} \\
\small{\{mberger, holzmann\}@mathematik.uni-marburg.de}}

\date{} 

\begin{document}

\maketitle

\begin{abstract}
    We obtain minimax-optimal convergence rates in the supremum norm, including information-theoretic lower bounds, for estimating the covariance kernel of a stochastic process which is repeatedly observed at discrete, synchronous design points. We focus on the supremum norm instead of the simpler $L_2$ norm, since it corresponds to the visualization of the estimation error and forms the basis for the construction of uniform confidence bands. For dense design, {\color{black} assuming Hölder-smooth sample paths} we obtain the $\sqrt n$-rate of convergence in the supremum norm without additional logarithmic factors which typically occur in the results in the literature. 
    Surprisingly, in the transition from dense to sparse design the rates do not reflect the two-dimensional nature of the covariance kernel but correspond to those for univariate mean function estimation.  
    Our estimation method can make use of higher-order smoothness of the covariance kernel away from the diagonal, and does not require the same smoothness on the diagonal itself. Hence, our results 
    cover covariance kernels of processes with rough, {\color{black} non-differentiable} sample paths. Moreover, the estimator does not use mean function estimation to form residuals, and no smoothness assumptions on the mean have to be imposed. In the dense case we also obtain a central limit theorem in the supremum norm, which can be used as the basis for the construction of uniform confidence sets. {\color{black} Extensions to estimating partial derivatives as well as to asynchronous designs are also discussed.}  
    Simulations and real-data applications illustrate the practical usefulness of the methods. 
\end{abstract}

\vspace{3mm}

\noindent \textit{Keywords.} Covariance kernel, functional data, optimal rates of convergence, supremum norm, synchronously sample data
		


\section{Introduction}

Mean function and covariance kernel are the two most-important parameters of a stochastic process having finite second moments. Estimates of the covariance kernel of a repeatedly observed stochastic process allow to assess its variability, in particular through the associated principle component functions \citep{cai2010nonparametric, hall2009theory, ramsay1998functional}. Further, the smoothness of the paths of a Gaussian process is closely related to the smoothness of the covariance kernel on the diagonal \citep{azmoodeh2014necessary}. Therefore, estimates of the covariance kernel allow to draw conclusions on the path properties of the observed process. 

In the setting of functional data analysis, stochastic processes are repeatedly observed at discrete locations and potentially with additional observation errors \citep{wang2016functional, cai2011optimal}. A deterministic, synchronous design refers to fixed, non-random observation points which are equal across functions. It typically arises for machine recorded data, such as weather data at regular time intervals at weather stations. In contrast, in random, asynchronous designs observation points are realizations of independent random variables. 

Estimating the covariance kernel has been intensely investigated for random asynchronous design. Here, \citet{li2010uniform, zhang2016sparse, hall2006properties, xiao2020asymptotic, mohammadi2024functional} obtain rates of convergence in $L_2$ and also in the supremum norm under Hölder smoothness assumptions on the covariance kernel. \citet{cai2010nonparametric} derive optimal rates in the random design setting under the assumptions that the sample paths of the process are contained in a particular reproducing kernel Hilbert space. {\color{black} Estimates of the covariance kernel from functional snippets under a semiparametric specification are obtained and analysed in \citet{lin2022mean}.} 

For fixed synchronous design there are fewer contributions. \citet{zbMATH06254554} give a consistency result in the supremum norm, and \citet{xiao2020asymptotic} presents rates in $L_2$ and in the sup-norm for spline estimators over Hölder smoothness classes. 

In the present paper we comprehensively analyse estimates of the covariance kernel for fixed synchronous design. We obtain optimal rates of convergence in a minimax sense over Hölder smoothness classes, including information-theoretic lower bounds. In our analysis we focus on the supremum norm instead of the simpler $L_2$ norm, since it corresponds to the visualization of the estimation error and forms the basis for the construction of uniform confidence bands. 
In particular, for dense design we obtain the $\sqrt n$ rate of convergence in the supremum norm and also the associated central limit theorem without the additional logarithmic factors, which are prevalent in the literature \citep{xiao2020asymptotic, li2010uniform, zhang2016sparse, mohammadi2024functional}.  
Our method can make use of higher-order smoothness of the covariance kernel away from the diagonal, and does not require the same amount of smoothness on the diagonal itself. Thus, as in \citet{mohammadi2024functional} our results apply to covariance kernels of processes with relatively rough, in particular non-differentiable sample paths. Notably, our estimator does not require mean estimation and forming residuals. Hence virtually no smoothness assumptions on the mean function are required, in contrast to e.g.~\citet{xiao2020asymptotic}. 
In the transition from dense to sparse design the optimal rates that we obtain do not reflect the two-dimensional nature of the covariance kernel but correspond to those for univariate mean function estimation as recently obtained in \citet{berger2023dense}.  A similar phenomenon has been observed in \citet{cai2010nonparametric, hall2006properties} but for somewhat different settings. The matching lower bounds, in particular in the sparse-to-dense transition regime, require technically involved arguments.


The paper is organized as follows. In Section \ref{sec:model:estimators} we introduce the model as well as our estimator, a modification of the local polynomial estimator restricted to observation pairs above the diagonal. Section \ref{sec:rate_cov_estimation} contains our main results: upper and matching lower bounds for the rate of convergence in the supremum norm. These are complemented by a central limit theorem in the space of continuous functions in  Section \ref{sec:asympnorm}. 
{\color{black} In Section \ref{sec:extensions} we discuss extensions of our method to estimating partial derivatives of the covariance kernel, as well as to asynchronous designs, in which a preliminary estimate of the mean is required. }

%
Section \ref{sec:simapp} contains simulations and a real-data application. 
In the simulations in Section \ref{sec:sims} we investigate the effect of the choice of the bandwidth and propose a cross-validation scheme for bandwidth selection. Further we  simulate the contribution to the sup-norm error of various components in the error decomposition, thus illustrating our theoretical analysis. Moreover we demonstrate numerically the need to leave out empirical variances of the data points and to restrict smoothing to the upper triangle for covariance kernels which are not globally smooth.  
%
In Section \ref{sec:realdata} we provide an application to a data set of daily temperature series, and discuss how estimates of the standard deviation curves and the correlation functions vary over the year. 
Section \ref{sec:conclude} concludes. Proofs of the main results are gathered in Section \ref{sec:proofs}. The supplementary appendix contains further technical material. An R package \verb+biLocPol+ can be found on \href{https://github.com/mbrgr/biLocPol}{Github}, together with the \href{https://github.com/mbrgr/Optimal-Rates-Covariance-Kernel-Estimation-in-FDA.git}{\texttt{R}-Code} used for the data in this paper.

\smallskip

We conclude the introduction by introducing some relevant notation. 
%
For sequences $(p_n), (q_n)$ tending to infinity, we write $p_n \lesssim q_n$, if $p_n =\Oop{q_n}$, $p_n \simeq q_n$ if $p_n \lesssim q_n$ and $q_n\lesssim p_n$, and $p_n \lessneqsim q_n$, if $p_n = o(q_n)$.  $\mc O_{\prob,  h \in (h_1, h_0]}$ denotes stochastic convergence uniformly for $h \in (h_1 , h_0]$. Finally, $\norm{\,\cdot\,}_\infty$ denotes the supremum norm, where the domain becomes clear from the context. 

\section{The model, smoothness classes and linear estimators} \label{sec:model:estimators}


Let the observed data $(Y_{i,  j}, x_{ j})$ be distributed according to the model \citep{berger2023dense, cai2011optimal}
\begin{align}
	Y_{i,  j}= \mu( x_{ j}) + Z_i( x_{ j}) + \epsilon_{i,j} \,, \quad  i=1,\dotsc,n\,, \ j = 1,\ldots,p  \,, \label{eq:model}
\end{align}
where $Y_{i, j}$ are real-valued response variables and the {\color{black}$x_j \in [0,1]$} are known non-random design points with $x_1 < \ldots < x_{p}$.
The processes $Z_1,\dotsc,Z_n$ are i.i.d.~copies of a mean-zero, square integrable random process $Z$ with an unknown covariance kernel 
$$\Gamma(x,y) = \expec[Z(x)\,Z(y)], \qquad x,y \in [0,1],$$ the estimation of which we shall focus on in this paper. The errors $\epsilon_{i,j}$ are independent with mean zero and are also independent of the $Z_i$, and the mean function $\mu$ is unknown. 
The number of design points $p \defeq p_n$ as well as the design points $ x_{j}= x_j(p,n)$ themselves depend on the number $n$ of functions which are observed.

Following the ideas in \citet{mohammadi2024functional}, by symmetry $\Gamma(x,y) = \Gamma(y,x)$ it suffices to estimate $\Gamma$ on the upper triangle 
\begin{align}
	T \defeq \{(x,y) \in  [0,1]^2 \mid x \leq y\}. \label{eq:upper_triangle}
\end{align}


As estimator for $\Gamma$ at $(x,y) \in T$ we consider 
%
\begin{equation}
\color{black}
    \covest xyh \defeq  \sum_{j<k}^p \wjk xyh \, z_{j,k;n} , \quad \text{where } \quad    z_{j,k;n} \defeq  \frac1{n-1}\sum_{i=1}^n\big(Y_{i,j} Y_{i,k} - \bar Y_{n, j} \bar Y_{n, k}\big)\label{eqn:estimatorCovariance}
\end{equation}
with   $\bar Y_{n,j } = n^{-1}\sum_{i=1}^n Y_{i,j}$,  $j,k \in \{1, \ldots, p\}$. The $\textbf{}$ are the entries of the empirical covariance matrix; $h>0$ is a bandwidth parameter and  $\wjk xyh = w_{j,k;p}(x,y;h;x_{1},\dotsc,x_{p})$ are weights which satisfy the assumptions listed in Section \ref{sec:rate_cov_estimation}. 
Note that we leave out the diagonal ($j = k$) in order to avoid bias induced by the squared errors $\epsilon_{i,j}^2$, and furthermore that following \citet{mohammadi2024functional} we build the estimator with observation pairs above the diagonal $j < k$ to avoid the potential lower smoothness of the covariance kernel on the diagonal. Let us also stress that the estimator does not require an estimator $\hat \mu_n$ of the mean function $\mu$ to form residuals $Y_{i,j} - \hat \mu_n(x_j)$. Indeed, {\color{black} the entries $z_{j,k;n}$ of the empirical covariance matrix} and hence $\covest xyh$ are independent of $\mu$. Therefore  no assumptions on $\mu$ are required, see also the discussion in Remark \ref{rem:meanest}. For $x^\prime, y^\prime \in [0,1]$ with $x^\prime>y^\prime$ we simply set
$\hat \Gamma_n(x^\prime,y^\prime;h) : = \hat \Gamma_n(y^\prime,x^\prime;h)$. Then 
$$ \normb{\hat \Gamma_{n}(\cdot; h) - \Gamma}_\infty = \sup_{x,y \in [0,1]} \big|\hat \Gamma_n(x,y;h) - \Gamma(x,y)\big| = \sup_{(x,y) \in T} \big|\hat \Gamma_n(x,y;h) - \Gamma(x,y)\big|,$$
so that we may focus on the analysis of the estimator on $T$. 
\begin{example}[Local polynomial estimator]\label{ex:localpolynomial}
   We shall show that restricted bivariate local polynomial estimators of order $m \in \N_0$, that is the first coordinate $(\hat \vartheta(x,y))_1$ of the vector 
    \begin{equation}
        \hat \vartheta(x,y) = \argmin_{\vartheta} \sum_{j<k}^p \bigg( {\color{black} z_{j,k;n}} - \bs \vartheta^\top U_m\left(
\begin{array}{c}
(x_j -x)/h\\
(x_k-y)/h\\
\end{array}
\right)\bigg)^2 K\left(
\begin{array}{c}
(x_j -x)/h\\
(x_k-y)/h\\
\end{array}
\right), \quad x \leq y,\label{eqn:minimisationProblemLocPol}
    \end{equation}
are linear estimators under mild assumptions and satisfy our requirements in Assumption \ref{ass:weights} in the next section. Here, 
$K$ is a bivariate non-negative kernel, $h>0$ is a bandwidth and $U_m\colon \R^2 \to \R^{N_m}$ with $N_m \defeq \frac{(m+1)(m+2)}2$ is a vector containing the monomials up to order $m$, with the constant as first entry, that is
\begin{align*}
	U_m(u_1,u_2)\defeq \big(1, P_{1}(u_1, u_2), \ldots, P_{m}(u_1,u_2)\big)^\top,\quad u_1,u_2\in[0,1]\,,
\end{align*}
where 
\begin{equation*}
	P_l(u_1, u_2)\defeq\bigg(\frac{u_1^l}{l!}, \frac{u_1^{l-1}u_2}{(l-1)!}, \frac{u_1^{l-2}u_2^2}{(l-2)!2!},\ldots, \frac{u_2^l}{l!}\bigg), \quad u_1,u_2 \in [0,1]\,.
\end{equation*}

\end{example}


Now let us turn to the assumptions that we impose on the process $Z$ and its covariance function. 
  
A function $f\colon T \to \R$ is H\"older-smooth with order $\gamma>0$ 
if for all indices $ \bs s=(s_1,s_2) \in \N_0^2$ with $\abs{\bs s} = s_1 + s_2 \leq \lfloor \gamma \rfloor =\max\{  k\in\N_0 \mid k<\gamma \} \defeql k$, 
 the partial derivatives $D^{\bs s} f(w) = \partial_1^{s_1}\, \partial_2^{s_2} f(v)$, $v \in T$ exist and if the H\"older-norm given by
$$\norm{f}_{\mc H, \gamma;T}\defeq\max_{\abs{\bs s}\leq k} \sup_{v \in T} \abs{D^{\bs s} f(v)}+ \max_{\abs{\bs s}=k}\sup_{v, w \in T,\, v\neq w} \frac{|D^{\bs s} f(v)-D^{\bs s} f(w)|}{\norm{v-w}_\infty^{\gamma-k}}$$
is finite.
	%
%
Define the H\"older class with parameters $\gamma>0$ and $L>0$ on $T$  by
\begin{equation}\label{def:hoelder:class}
    \mc H_{T}(\gamma, L) = \big\{f\colon T \to \R \mid \norm{f}_{\mc H, \gamma;T} \leq L \big\}.
\end{equation}

{\color{black} For a function $f$ on $[0,1]^2$, $\norm{f}_{\mc H, \gamma;[0,1]^2}$ and $\mc H_{[0,1]^2}(\gamma, L)$ are defined analogously.} 
Note that a symmetric function $f$ on $[0,1]^2$, {\color{black} if restricted to $T$} can be contained in $\mc H_{T}(\gamma, L)$ for $\gamma >1$ even if $f$ is not partially differentiable on the diagonal of $[0,1]^2$ and hence not Hölder smooth of order greater than $1$ on $[0,1]^2$.  {\color{black} Indeed, while $f \in \mc H_{T}(\gamma, L)$ for $\gamma \leq 1$ implies $f \in \mc H_{[0,1]^2}(\gamma, 2\,L)$, for $\gamma > 1$ it only implies $f \in \mc H_{[0,1]^2}(1, 2\,L)$.}
For covariance functions of second order stochastic processes, roughly speaking  
{\color{black} existence of the partial derivatives up to order $2\,k$ of the covariance kernel of a centered Gaussian process} 
in a neighborhood of the diagonal of $[0,1]^2$ implies that the paths are
a.s.~differentiable in mean square of order $k$, which together with some small additional assumptions implies that the paths are a.s.~
$k$-times continuously differentiable 
\citep[Section 1.4]{azais2009level}. Thus, as stressed in \citet{mohammadi2024functional} processes with relatively rough sample paths such as the Brownian motion do not have covariance kernels which are smooth on $[0,1]^2$. However, as is the case for Brownian motion, these kernels can still be smooth on $T$, and our estimator will be able to make use of higher order smoothness of $\Gamma$ restricted to $T$.





\vspace{2mm}

For the process $Z$  we further assume that $\expec[Z(0)^4] < \infty$ and that the paths are Hölder continuous of some potentially low order: there exists $0 < \beta \leq 1$ and a random variable $M = M_Z>0$ with $\expec [M^4] < \infty $ such that 
\begin{equation}\label{eq:hoeldercontpathsZ}
	\big|Z( x)-Z( y)\big| \leq M \, | x- y|^\beta, \qquad x,  y \in [0,1] \quad \text{almost surely.}
\end{equation}

%
%
Given $C_Z>0$ and $0 < \beta_0 \leq 1$, we consider the class of processes 
\begin{align}
	\Zclass \defeq \mc P (\gamma;L,\beta_0, C_Z) & = \big\{ Z: [0,1] \to \R \ \text{centered random process} \mid \exists \ \beta \in [\beta_0,1] \text{ and } M\nonumber\\& \qquad  \text{ s.th. }
	 \expec [M^4] + \expec [Z( 0)^4]\leq C_Z \text{, \eqref{eq:hoeldercontpathsZ} holds} \label{eq:classprocesses} \text{ and  } \Gamma_{\mid T} \in \mc H_T(\gamma,L)\big\}. 
\end{align}	

\begin{example}[Gaussian processes]\label{ex:Gaussian}
If $Z$ has covariance function $\Gamma$ which satisfies $\Gamma_{\mid T} \in \mc H_T(\gamma,L)$, then  
{\color{black} 
\begin{align*}
\dx_Z(x,y) \defeq \expec\big[(Z(x) - Z(y))^2\big] & \leq \big|\Gamma(x,y) - \Gamma(x,x) \big| + \big|\Gamma(x,y) - \Gamma(y,y) \big|\\
& \leq 2\, L \, |x-y|^{\min(\gamma,1)}    
\end{align*}
by using the mean value theorem in case $\gamma > 1$. }
Hence Theorem 1 in \citet{azmoodeh2014necessary} implies that for a (centered) Gaussian process $Z$ with $\Gamma_{\mid T} \in \mc H_T(\gamma,L)$ the sample paths satisfy \eqref{eq:hoeldercontpathsZ} for each $\beta < \min(\gamma,1)/2$, and the fourth moment of the Hölder constant $M$ can be bounded in terms of $\beta, \gamma$ and $L$. Since the rate of convergence for estimating $\Gamma$ will not depend on $\beta_0$, our results therefore also apply to the class
\begin{equation}\label{eq:Gaussianclass}
 \mc P_{\mc G}(\gamma) \defeq \mc P_{\mathcal G} (\gamma;L) = \big\{ Z: [0,1] \to \R \ \text{centered Gaussian process} \mid \Gamma_{\mid T} \in \mc H_T(\gamma,L)\big\}.
\end{equation} 
\end{example}



\medskip


\section{Optimal rates of convergence for covariance kernel estimation and asymptotic normality}\label{sec:rate_cov_estimation}

\subsection{Upper bounds}

To derive the upper bounds consider the following assumptions on the weights of the linear estimator in \eqref{eqn:estimatorCovariance}, on the distribution of the errors and on the design. 

\begin{assumption}\label{ass:weights} 
	There is a $c>0$ and a $h_0>0$ such that for  sufficiently large $p$, for the weights $\wjk xyh$ the following holds for all $h \in (c/p_, h_0]$ for constants $\Cmax, \Clip>0$ which are independent of $n, p,h$ and $(x, y) \in T$.
	\begin{enumerate}[label=\normalfont{(W\arabic*)},leftmargin=9.9mm]
		\item The weights reproduce polynomials of a degree $\zeta \geq 0$, that is for $(x,y)  \in T$,
		\begin{align*} 
			\sum_{ j<k}^{ p}  \wjk xyh =1\,, \quad \sum_{j<k}^{ p}( x_{ j}- x)^{r_1}(x_k - y)^{r_2} \,  \wjk xyh = 0\,, 
		\end{align*} \label{ass:weights:polynom}
        for $r_1,r_2 \in \N_0$ s.t. {\color{black} $1 \leq $} $r_1 + r_2 \leq \zeta$.
		\item We have $\wjk xyh = 0$ if $\max(\abs{x_j-x}, \abs{x_k - y})> h$ with $(x,y) \in T$. \label{ass:weights:vanish}
		%
		%
		\item For the absolute values of the weights it holds $  \max_{j<k} \big| \wjk xyh\big|  \leq \Cmax (p\,h)^{-2},\,  x \leq y.$  \label{ass:weights:sup}
		\item For a Lipschitz constant $\Clip > 0$ it holds that
		\begin{align*}
			\absb{  \wjk xyh -  \wjk{x^\prime}{y^\prime}h} \leq \frac{\Clip}{( p\, h)^2} \bigg(\frac{\max(\abs{x-x^\prime}, \abs{x-y^\prime})}h \wedge 1\bigg), \ \;x\leq y,\, x^\prime\leq y^\prime \,.
		\end{align*} \label{ass:weights:lipschitz}
        %

	\end{enumerate}		
\end{assumption}

{\color{black} These are checked in Section \ref{sec:proof:lemma:locpol:weights} of the supplementary appendix for the restricted form of the local polynomial weights of Example \ref{ex:localpolynomial}  for designs with a continuous, non-zero design density, see Assumption \ref{ass:designdensity} in the supplement.} 

\begin{assumption}[Sub-Gaussian errors] \label{ass:distribution}
		The random variables $\{\epsilon_{i, j} \mid 1 \leq i \leq n,\,  1 \le  j \le  p\}$ are independent and independent of the processes $Z_1,\dotsc,Z_n$. Further we assume that the distribution of $\epsilon_{i, j}$ is sub-Gaussian, and setting $ \sigma_{ij}^2 \defeq \expec[\epsilon_{i, j}^2]$ we have that $\sigma^2 \defeq \sup_n \max_{ i,j} \sigma_{ij}^2 < \infty$ and that there exists {\color{black}$\kappa\geq 1$} such that {\color{black}$\kappa^2\sigma_{i, j}^2$} is an upper bound for the squared sub-Gaussian norm of $\epsilon_{i, j}$.
\end{assumption}

\begin{assumption}[Design Assumption] \label{ass:design:localization}
	There is a constant $\Ccard > 0$ such that for each $x \in T$ and $h>0$ we have that
	\begin{align*}
		\card \big\{ j \in \{1, \ldots, p \}\mid x_j \in [x-h, x+h]\big\} & \leq \Ccard \,p\,h\,.
	\end{align*}
\end{assumption}

{\color{black} Assumption \ref{ass:design:localization}, which is from \citet[Assumptions (LP), p. 37]{tsybakov2008introduction}, implies that an $h$-neighborhood of any $x \in [0,1]$ can only contain up to order $p\,h$ design points. It is satisfied for designs with a Lipschitz-continuous design density \citep[Lemma 7]{berger2023dense}.  When additionally imposing high-level assumptions on the weights in \eqref{eqn:estimatorCovariance} which we collect in Assumption \ref{ass:weights}, Assumption \ref{ass:design:localization} can be used to bound the variance of the estimator. The construction of appropriate weights, see Section \ref{sec:proof:lemma:locpol:weights} of the supplementary appendix, requires stronger design assumptions. There we use designs with a continuous, non-zero design density, see Assumption \ref{ass:designdensity} in the supplement, which also implies a lower bound on the number of design points in $h$-neighborhoods.  }    




\begin{theorem}\label{thm:rates_cov_estimation}
	Consider model \eqref{eq:model} under Assumptions \ref{ass:design:localization} and \ref{ass:distribution}. Suppose that for given $\gamma>0$ the weights in the linear estimator $\hat \Gamma_{n}(\cdot; h)$ for the covariance kernel $\Gamma$ in \eqref{eqn:estimatorCovariance}  satisfy Assumption \ref{ass:weights} with $\zeta = \floor \gamma$. 
 Then for $0 < \beta_0 \leq 1$ and $L, C_Z>0$ we have that 
 \begin{align*}
		\sup_{h \in (c/p, h_0]}\, \sup_{Z \in \mc P (\gamma;L,\beta_0, C_Z)} a_{n,p,h}^{-1} \expec\Big[\normb{\hat \Gamma_{n}(\cdot; h)  - \Gamma}_\infty\Big] = \mc O(1)\,,
	\end{align*}
	where
	\begin{align}\label{eq:upperboundwithbandwitdh}
		a_{n,p,h} & = \max \Big( h^\gamma, \Big( \frac{\log(h^{-1})}{n\,p\,h}\Big)^{1/2}, n^{-1/2}\Big)\,.
	\end{align}
	Hence by setting
	$	h^\star \sim \max \Big( c/p, \Big(\frac{\log(n\,p)}{n\,p}\Big)^{\frac1{2\gamma + 1}} \Big)\,$
	we obtain 
    {\color{black} 
	\begin{align}\label{eq:rateofconvcovkernel}
		\sup_{Z \in \mc P (\gamma;L,\beta_0, C_Z)} \expec \Big[\normb{\hat \Gamma_{n}(\cdot; h^\star) - \Gamma}_\infty\Big] & = \mc O \Big(\max \Big(p^{-\gamma},  \Big(\frac{\log(n\,p)}{n\,p}\Big)^{\frac \gamma{2\gamma + 1}},  n^{-1/2}\Big) \Big)\,.
	\end{align}	
    }
 Furthermore, in both upper bounds the class $\Zclass = \mc P (\gamma;L,\beta_0, C_Z)$ can be replaced by $\mc P_{\mc G}(\gamma) = \mc P_{\mathcal G} (\gamma;L)$ in \eqref{eq:Gaussianclass}.
\end{theorem}

\smallskip

The proof is given in Section \ref{sec:proofupperbound}. 

\begin{remark}[Comments on the assumptions and the rate of convergence]
The rate in \eqref{eq:rateofconvcovkernel} consists of a discretization bias $p^{-\gamma}$, {\color{black} which dominates the rate in the sparse regime for relatively small $p$;} the $1/\sqrt n$ rate arising from the contribution of the processes $Z_i$ {\color{black} dominating  in the dense regime with large $p$}, as well as the intermediate term involving the errors which is specific to the use of the supremum norm. {\color{black} To achieve the parametric $1/\sqrt n$-rate without additional logarithmic terms, in contrast to e.g.~\citet{li2010uniform, zhang2016sparse} we additionally require some Hölder smoothness for the sample paths in \eqref{eq:hoeldercontpathsZ}}. 
Overall \eqref{eq:rateofconvcovkernel} is analogous to the rate obtained for the mean function in one dimension $d=1$ in \citet{berger2023dense}. Somewhat surprisingly, the fact that $\hat \Gamma_{n}(x,y; h)$ is a bivariate function neither influences the rate arising from the discretization bias $p^{-\gamma}$ nor that from the observation errors, $(\log(n\,p)/(n\,p))^{ \gamma/(2\gamma + 1)}$, where a factor $2\gamma + 2$ would be expected in the denominator of the exponent. This seems to be an improvement of the rates of covariance kernel estimation obtained in \citet{li2010uniform, zhang2016sparse}, and is reminiscent of the one-dimensional rates obtained in \citet{cai2010nonparametric, hall2006properties} for the principal component functions. {\color{black} Let us mention that this improvement comes at the cost of assuming sub-Gaussian errors in Assumption \ref{ass:distribution}, which is more restrictive than the moment assumptions used e.g.~in \citet{li2010uniform, zhang2016sparse}.}
  Overall, the discussions regarding the regimes in the rate \eqref{eq:rateofconvcovkernel} as well as a choice of $h$ independently of the smoothness $\gamma$ from \citet[Remarks 4 and 5]{berger2023dense} apply to this setting as well. 
\end{remark}

\begin{remark}[Error decomposition and proof techniques]\label{rem:errordecomp}

 In Lemma \ref{lem:cov_error_decomp} in the appendix we derive an error decomposition of the form
  { \color{black}
	\begin{align}
		\covest xyh  -\Gamma(x,y)& = \sum_{j<k}^{p} \wjk xyh \big( \Gamma({x_j},{x_k}) - \Gamma(x,y) \big)  + \sum_{j<k}^{p} \wjk xyh \frac1n \sum_{i=1}^n \epsilon_{i,j}\epsilon_{i,k} \label{eq:decomp:sketch} \\
		 & + \sum_{j<k}^{p} \wjk xyh \frac1n \sum_{i=1}^n\big(Z_i({x_j})Z_i({x_k}) - \Gamma({x_j},{x_k})\big) \nonumber\\
         & \, + \sum_{j<k}^{p} \wjk xyh \frac1n\sum_{i=1}^n \big(Z_i({x_j})\epsilon_{i,k} + Z_i({x_k})\epsilon_{i,j}\big) + \text{ higher order terms} \nonumber \\
         &\defeql A_1 + A_2 + A_3  + A_4 + \text{higher order terms}.\nonumber
	\end{align} 
 }
\end{remark}

The term {\color{black}$A_1$} in \eqref{eq:decomp:sketch} is bounded by $h^{\gamma}$ by using the property \ref{ass:weights:polynom} of the weights and standard estimates.  The term {\color{black}$A_3$} in \eqref{eq:decomp:sketch} induces the $1/\sqrt n$ rate by using \\$\expec\big[\norm{n^{-1} \sum_{i= 1}^n Z_i(\cdot)Z_i(\cdot) - \Gamma}_\infty\big] = \mc O(n^{-1/2})$ as well as the boundedness of sums of absolute values of the weights. The term {\color{black}$A_2$} involving products of errors is analyzed using the Hanson-Wright inequality and can be bounded by $( \log(h^{-1})/(n\,(p\,h)^2))^{1/2}$, and since $p^{-1} \lesssim h$ is negligible. {\color{black} Intuitively, apart from the logarithmic factor this arises from weighted averaging of order $(p\, h)^2$ terms $n^{-1}\,\sum_{i=1}^n \epsilon_{i,j}\epsilon_{i,k}$,  which themselves are of order $1{\sqrt n}$ and approximately independent.} Finally, the intermediate term {\color{black}$A_4$ which captures the interplay between processes and errors gives the rate $( \log(h^{-1})/(n\,p\,h))^{1/2}$, resulting in the overall bound \eqref{eq:upperboundwithbandwitdh}.  Intuitively, averaging the $Z_i$ gives a $n^{-1/2}$, and the sum over the weights effecting the errors is an average of order $(p\,h)$ terms, resulting in $p\, h)^{-1/2}$. These effect the rate multiplicatively by independence of $Z_i$ and the errors $\epsilon_{i,k}$. } 
A detailed proof is contained in Section \ref{sec:proofs}.

\begin{remark}[Effect of mean function estimation]\label{rem:meanest}
    The mean function $\mu$ cancels out when forming the empirical covariance matrix in \eqref{eqn:estimatorCovariance}. Therefore, formally no assumptions are required on $\mu$, and we could of course include a $\sup_\mu$ in the statements of the upper bounds. This is in contrast to most approaches in the literature which rely on an initial estimate of the mean and forming residuals \citep{zhang2016sparse, li2010uniform, xiao2020asymptotic}. However, this feature seems to be particular to the synchronous design that we consider. If design points are asynchronous, it appears that also estimating the mean cannot be avoided. {\color{black} In Section \ref{sec:asyncdesign} and in Section \ref{sec:asynchronousdesign} of the supplementary appendix we give an extension of our method to this situation. 
    However, the precise effect of mean function estimation on covariance kernel estimation in a minimax sense, similar to nonparametric variance function estimation in nonparametric regression \citep{Wang}, has to the best of our knowledge not been investigated in the literature.  Also note that the observed sample paths $\mu + Z_i$ include the mean, and Hölder-continuity of these paths requires Hölder-continuity of the mean $\mu$.}    
\end{remark}

    

\subsection{Lower Bounds}

Now let us turn to corresponding lower bounds. Intuitively since covariance kernel estimation should be at least as hard as mean function estimation, and since the rate in \eqref{eq:rateofconvcovkernel} corresponds to the optimal rate for mean function estimation in one dimensions $d=1$ \citep{berger2023dense}, optimality of the rates is not surprising. However, the proofs in particular for the intermediate term in \eqref{eq:rateofconvcovkernel} are much more involved than for the mean function.


\begin{theorem}\label{theorem:optimality}
	{\color{black} Assume that in model \eqref{eq:model} the errors $\epsilon_{i,j}$ are i.i.d.~$\mathcal N(0, \sigma_0^2)$ - distributed, $\sigma_0^2 >0$.} Then setting
	$$a_{n,p} = \color{black} \max \Big(p^{-\gamma},  \Big(\frac{\log(n\,p)}{n\,p}\Big)^{\frac \gamma{2\gamma + 1}},  n^{-1/2}\Big)$$
	we have that
	\begin{align*}
		\liminf\limits_{n, p \to \infty} &\inf\limits_{\hat G_{n,p}} \, \sup_{Z \in \mc P_{\mathcal G}(\gamma; L)} a_{n,p}^{-1} \,\expec\, \Big[\normb{\hat G_{n,p} - \Gamma}_\infty\Big] >0\,,
	\end{align*}
	
	where the infimum is taken over all estimators {\color{black} $\hat G_{n,p}$ of $\Gamma$}. 
\end{theorem}

The proof is provided in Section \ref{ssec:proof:optimality}. 



\subsection{Asymptotic normality}\label{sec:asympnorm}


To derive the asymptotic normality of the estimator \eqref{eqn:estimatorCovariance} we need the following smoothness assumption on the forth moment function. 
%

\begin{assumption}[Forth moment function]\label{ass:asymp_norm}
        The forth moment function 
    \begin{equation}\label{eq:forthmomentoperator}
        R(x,y,s,t) \defeq \expec[Z(x)Z(y)Z(s)Z(t)]-\Gamma(x, y)\Gamma(s, t), \quad x,y,s,t \in [0,1]\,,
    \end{equation}    
        of the process $Z$ is Hölder smooth of some positive order: $R \in \mc H_{[0,1]^4}(\zeta, \tilde L)$ for some $ 0<\zeta \leq 1$ and $\tilde L > 0$. 
%
\end{assumption}

\begin{theorem}\label{thm:asymp_norm}
    In model \eqref{eq:model} under Assumptions \ref{ass:distribution} and \ref{ass:asymp_norm}, consider the linear estimator in (\ref{eqn:estimatorCovariance}) with weights satisfying Assumption \ref{ass:weights} with $\zeta = \floor{\gamma}$. Further suppose that for some $\delta > 0$ we have that $p \gtrsim n^{1/(2\gamma)} \log(n)^{1+2\delta}$. Then for all sequences of smoothing parameters $h_n$ for which 
    \begin{align*}
       h_n \in H_n & \defeq \big[ \log(n)^{1+\delta}/p\,,\,n^{-1/(2\gamma)}\log(n)^{-\delta} \big], \qquad n \in \N,
    \end{align*}
    it holds that
    \begin{equation}
        \sqrt{n}\, \big( \,\hat \Gamma_{n}(\cdot; h_n) - \Gamma\,) \;\stackrel{D}\rightarrow \;\mc G(\,0, R), 
    \end{equation}
    where $\mc G$ is a real-valued Gaussian process on $[0,1]^2$ with covariance operator $R$ given in \eqref{eq:forthmomentoperator}.
\end{theorem}



The proof is deferred to Section \ref{ssec:proof_asymp_norm} in the supplementary appendix. 

\begin{remark}
    The asymptotic covariance operator $R$ is as in the case with continuous and error-free observations, see e.g.~\citet[Theorem 2.1]{dette2020functional}.
\end{remark}





{\color{black}
\section{Extensions to derivative estimation and asynchronous designs}\label{sec:extensions}
\subsection{Derivative estimation}

Local polynomial estimators also yield estimates of derivatives. Derivative estimation of the mean function in model \eqref{eq:model} is discussed in \citet{berger2025smooth}.

Suppose that $\Gamma \in  \mc H_{T}(\gamma, L)$, $ \bs s=(s_1,s_2) \in \N_0^2$ with $\abs{\bs s} \leq \lfloor \gamma \rfloor $, and we aim to estimate the partial derivative $D^{\bs s}$ of order $\bs s$ of $\Gamma$ when restricted to the upper triangle $T$. 
To this end, linear estimators can be set up analogously to \eqref{eqn:estimatorCovariance} as
\begin{equation}
	\hat \Gamma_{n,p}^{(\bs s)}(x,y;h) \defeq \sum_{j<k}^p
	\wjks xyhs \, z_{j,k;n}, \quad (x,y) \in T\,,\label{eqn:estimatorCovariancederi}
\end{equation}
where the weights now reproduce the partial derivative of bivariate polynomials of order $\bs s$. A particular estimator arises from the restricted bivariate local polynomials in Example \ref{ex:localpolynomial} of order at least $\abs{\bs s}$, where the coordinate corresponding to ${\bs s}$ in $\hat \vartheta(x,y)$ is chosen and multiplied by $h^{- \abs{\bs s}}$ to yield the estimator. Precise assumptions and upper bounds for the rates of convergence are presented in the supplementary appendix, Section \ref{sec:derivativeestimation}, while proofs can be found in \citet{berger2025diss}.  
 
As applications of derivative estimation of $\Gamma$, if the process $Z$ has continuously differentiable sample paths, then under some additional regularity (e.g.~Gaussianity) the derivative process $Z^\prime$ has covariance kernel $\partial_1 \partial_2 \Gamma$. See also \citet{dai2018} for a discussion relating to the principle component basis functions of the  covariance kernel $\partial_1 \partial_2 \Gamma$ of $Z^\prime$. In particular, the variance function of $Z^\prime$ is given by $\var(Z^\prime(x)) = \partial_1 \partial_2 \Gamma(x,x)$, which occurs e.g.~in the construction of uniform confidence bands based on the Kac-Rice formula \citep{liebl2019fast}. 

Moreover, estimates of the partial derivatives of $\Gamma$ can be used to assess differentiability of the sample paths of $Z$ as follows. Suppose that $Z$ is a-priori only assumed to be Hölder-continuous as in \eqref{eq:hoeldercontpathsZ}, and that $\Gamma \in \mc H_{T}(\gamma, L)$ with $\gamma >1$. Then the partial derivatives $\partial_1 \Gamma$ and $\partial_2 \Gamma$ of $\Gamma$ restricted to $T$ can be estimated using \eqref{eqn:estimatorCovariancederi}. Now if $\Gamma$ is actually smooth also on the diagonal, so that $\Gamma \in \mc H_{[0,1]^2}(\gamma, L)$, then using the symmetry of $\Gamma $ it can be checked that $\partial_1 \Gamma(x,x) = \partial_2 \Gamma(x,x)$.  Therefore, comparing the restricted estimates $\hat \Gamma_{n,p}^{(1,0)}(x,x;h)$ and $\hat \Gamma_{n,p}^{(0,1)}(x,x;h)$ allows to check whether a covariance kernel $\Gamma \in \mc H_{T}(\gamma, L)$ is globally differentiable. But differentiability of the sample paths of $Z$ implies (under some regularity) global differentiability of the covariance kernel, so this also gives a method to assess the former property. For further discussion and an application see \citet[Sections 2.3, 3.2]{berger2025smooth}.

\subsection{Asynchronous designs}\label{sec:asyncdesign}
We briefly discuss how our estimation approach may be extended to an asynchronous design, in which the non-random design points 
$x_{i,1} < \ldots < x_{i,p}$ may differ for each row $i$. For notational simplicity we assume the same number $p$ in each row, but extensions to $p_i$ depending on $i$ would be possible.  
As an estimator for $\Gamma$ we consider an average of row-wise estimates centered at the estimated mean,
\begin{align}
    \hat \Gamma^{\mathrm{a}}_{n,p} (x,y;h) & = \frac{1}{n} \sum_{i = 1}^n \sum_{j <k}^p w_{i;j,k}(x,y;h) \hat z_{j,k;i}, \qquad \hat z_{j,k;i} = \big( Y_{i,j}- \hat \mu(x_{i,j})\big) \big( Y_{i,k}- \hat \mu(x_{i,k})\big)\,.\label{eq:covestasy}
\end{align}
Here $\hat \mu$ is some preliminary estimate of the mean, as e.g.~discussed in \citet[Section 7]{berger2023dense}, and the weights
$w_{i;j,k}(x,y;h) = w_{i;j,k}(x,y;x_{i,1} , \ldots , x_{i,p}, h)$ are assumed to satisfy the conditions of Assumption \ref{ass:weights} with constants that are uniform in $i$. These may be obtained by applying the restricted local polynomial approach of Example \ref{ex:localpolynomial} separately for each $i$, with $z_{j,k;n}$ replaced by $\hat z_{j,k;i}$.
In the supplementary appendix, Section \ref{sec:asynchronousdesign}, we analyze this estimator, and in particular show that if $\hat \mu$ is formed independently of the data $Y_{i,j}$ used in  \eqref{eq:covestasy} (say by sample splitting, which does not effect the rates but of course the constants), then we retain the upper bound $a_{n,p,h}$ in \eqref{eq:upperboundwithbandwitdh} for the sup-norm distance with an additional term $\|\hat \mu - \mu\|_\infty^2$. Thus the estimation effect of the mean only effects the upper bound quadratically, and therefore is negligible except in the relatively sparse case with small $p$ compared to $n$, and a mean function which is much rougher than the covariance kernel.

}
\section{Simulations and real-data illustration}\label{sec:simapp}

In this section we present simulation results for our methods and give a real-data application. First in Section \ref{sec:sims} we illustrate the finite sample effect of the choice of the bandwidth and propose and investigate a cross validation procedure to select a bandwidth. Further we simulate  the size in the sup-norm of the terms in the error decomposition in \eqref{eq:decomp:sketch} which the rate in Theorem \ref{thm:rates_cov_estimation} relies on. Finally we  compare the proposed estimator which uses only the empirical covariances above the diagonal to a more conventional bivariate local polynomial estimator which still leaves out the diagonal terms to reduce variability resulting from the errors but otherwise smooths over the diagonal.  
In Section \ref{sec:realdata} we give an illustration to  daily temperature curves in Nuremberg, in which we show how standard deviation and correlation functions resulting from our estimate of the covariance vary over the year. 
The \texttt{R}-code regarding the simulations and the real data examplae can be found in the Github repository \href{https://github.com/mbrgr/Optimal-Rates-Covariance-Kernel-Estimation-in-FDA.git}{\texttt{mbrgr/Optimal-Rates-Covariance-Estimation-in-FDA}}. The implementation of the calculation of the weights of the bivariate local polynomial estimator and the estimator itself can be found in the \texttt{biLocPol} package, which is also available on Github in the repository \href{https://github.com/mbrgr/biLocPol.git}{mbrgr/biLocPol}.

\subsection{Simulations}\label{sec:sims}

\begin{figure}[b!]
    \begin{subfigure}{0.45\linewidth}
        \includegraphics[width=\linewidth]{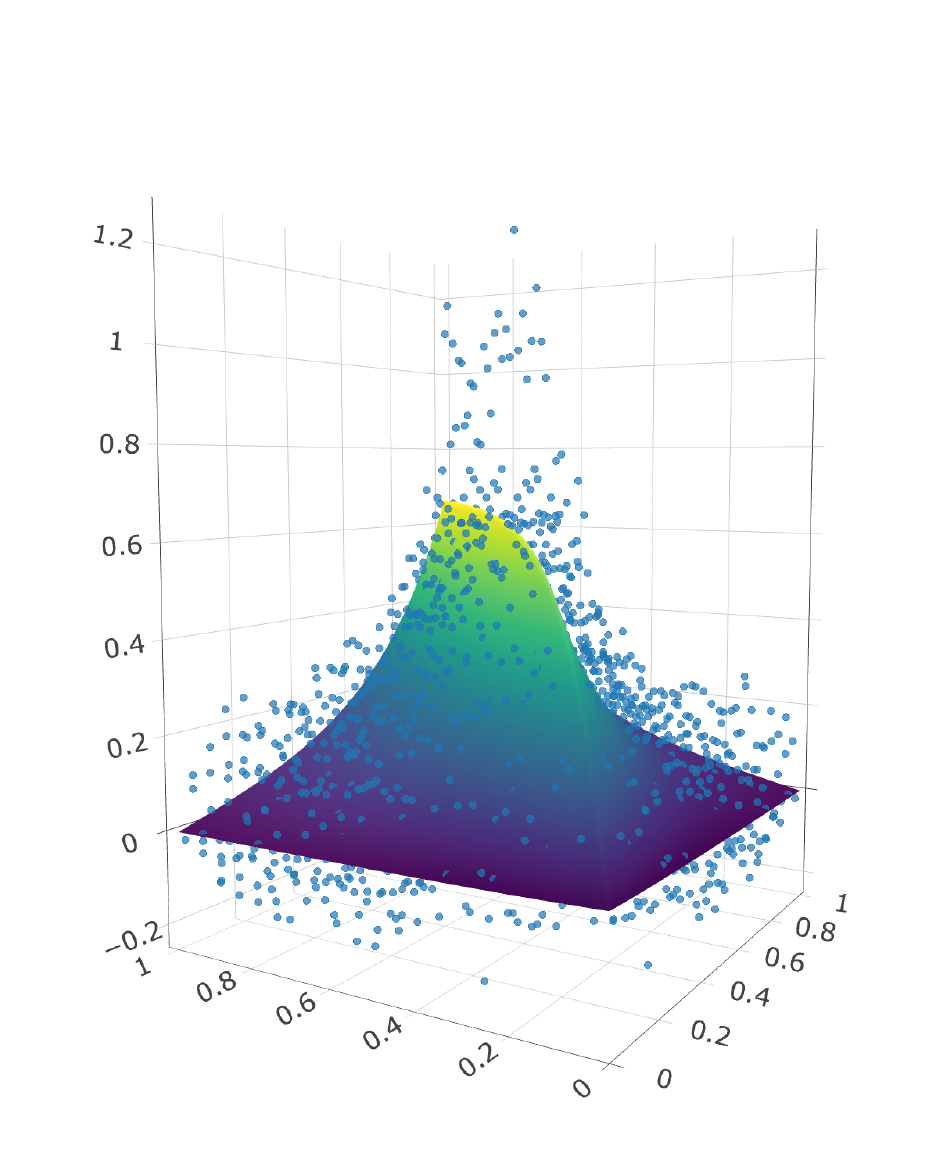}
        \caption{\small \href{https://mbrgr.github.io/from-dense-to-sparse-fda/figure31.html}{Front view.} }
		\label{fig:OU_observations_a}
    \end{subfigure}
\hfill
    \begin{subfigure}{0.45\linewidth}
        \includegraphics[width=\linewidth]{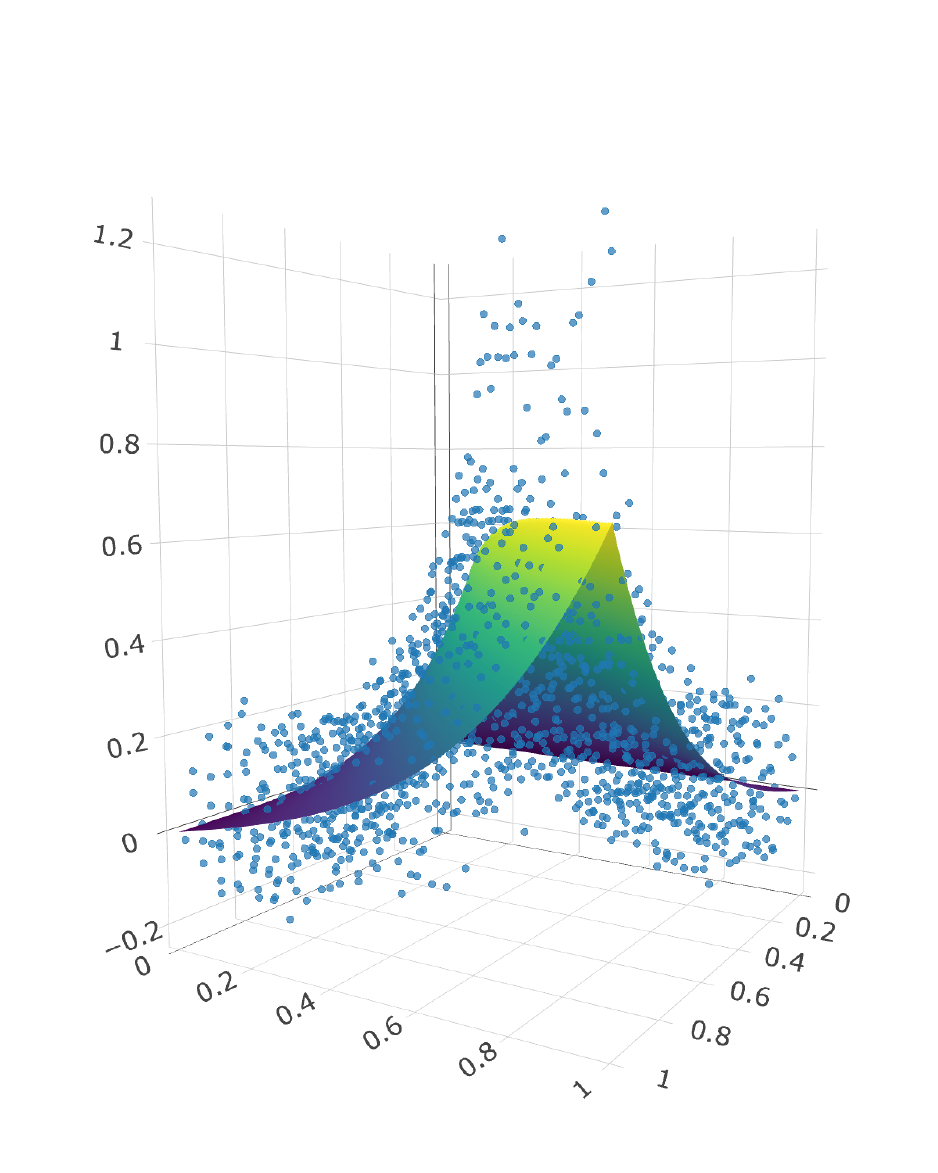}
		\label{fig:OU_observations_b}
        \caption{\small Back view.}
    \end{subfigure}%
    \caption{Covariance kernel $\Gamma_{\mathrm{OU}}$ in \eqref{eq:covkerOU} with parameters $\theta = 3$ and $\sigma = 2$. The scatter plot shows $(x_j,x_k,z_{j,k})$ with $p = 40$  and $n = 100$, and normally distributed errors with standard deviation $0.75$. The influence of the additional error variance is clearly visible on the diagonal of the covariance kernel.}
    \label{fig:OU_observations}
\end{figure}

For most of the simulations we consider the Ornstein-Uhlenbeck process
\begin{align*}
    Z_t & = \sigma \,\int_0^t \exp(- \theta\,(t - s)) \,\dx B_s\,
\end{align*}
with parameters $\theta = 3$ and $\sigma = 2$, and where $(B_s)_{s\geq 0}$ is a standard Brownian motion. It has covariance kernel given by 
\begin{align}\label{eq:covkerOU}
    \Gamma_{\mathrm{OU}}(s,t) & = \frac{\sigma^2}{2\,\theta}\,\big(\exp(-\theta\,\abs{t - s}) - \exp(-\theta\,(s + t))\big)\,,
\end{align}
which has a kink on the diagonal. The mean function $\mu$ in model \eqref{eq:model} is set to $0$ since it is ancillary in forming the estimator, and the errors $\epsilon_{i,j}$ are centered, normally distributed with standard deviation $\sigma_\epsilon = 0.75$. The grid points are equidistant at $x_j = (j-1/2)/p$, $j=1, \ldots, p$.





Figure \ref{fig:OU_observations} contains two plots of the covariance kernel in \eqref{eq:covkerOU} together with a scatter plot $(x_j,x_k,z_{j,k})$ of the empirical covariances  $z_{j,k} = (n-1)^{-1} \sum_{ i = 1}^n (Y_{i,j}Y_{i,k} - \bar Y_{j}\bar Y_k), 1 \leq j, k \leq p$ of the observations, for a particular sample with $p=40$ and $n=100$. One observes that the covariance kernel \eqref{eq:covkerOU} is not smooth on the diagonal, and that the empirical covariances deviate strongly from the underlying covariance kernel on the diagonal due to the additional variance from the observation errors $\epsilon_{i,j}$. For this sample, our variant of the local polynomial estimator \eqref{eqn:minimisationProblemLocPol}  of order $m=1$ (local linear) with bandwidth $h=0.3$ is displayed in Figure \ref{fig:plot_comp_OU_a} together with the true underlying covariance kernel.  For comparison in Figure \ref{fig:plot_comp_OU_b} we display the result for the estimator 
\begin{align}
    \hat \Gamma^{\neq}_n (x,y;h) \defeq \frac1{n-1}\sum_{i=1}^n \sum_{j\neq k}^p
    \wjk xyh \big(Y_{i,j} Y_{i,k} - \bar Y_{n, j} \bar Y_{n, k}\big)\,,\label{eqn:estimatorCovariance_full}
\end{align}
with local polynomial weights, which only excludes empirical variances but otherwise smooths over the diagonal. We again use $m=1$ (local linear) and $h=0.2$. This estimator displays a substantial bias due to the kink of the true covariance kernel along the diagonal. 



\begin{figure}[t!]
    \begin{subfigure}{0.49\linewidth}
        \includegraphics[width=\linewidth]{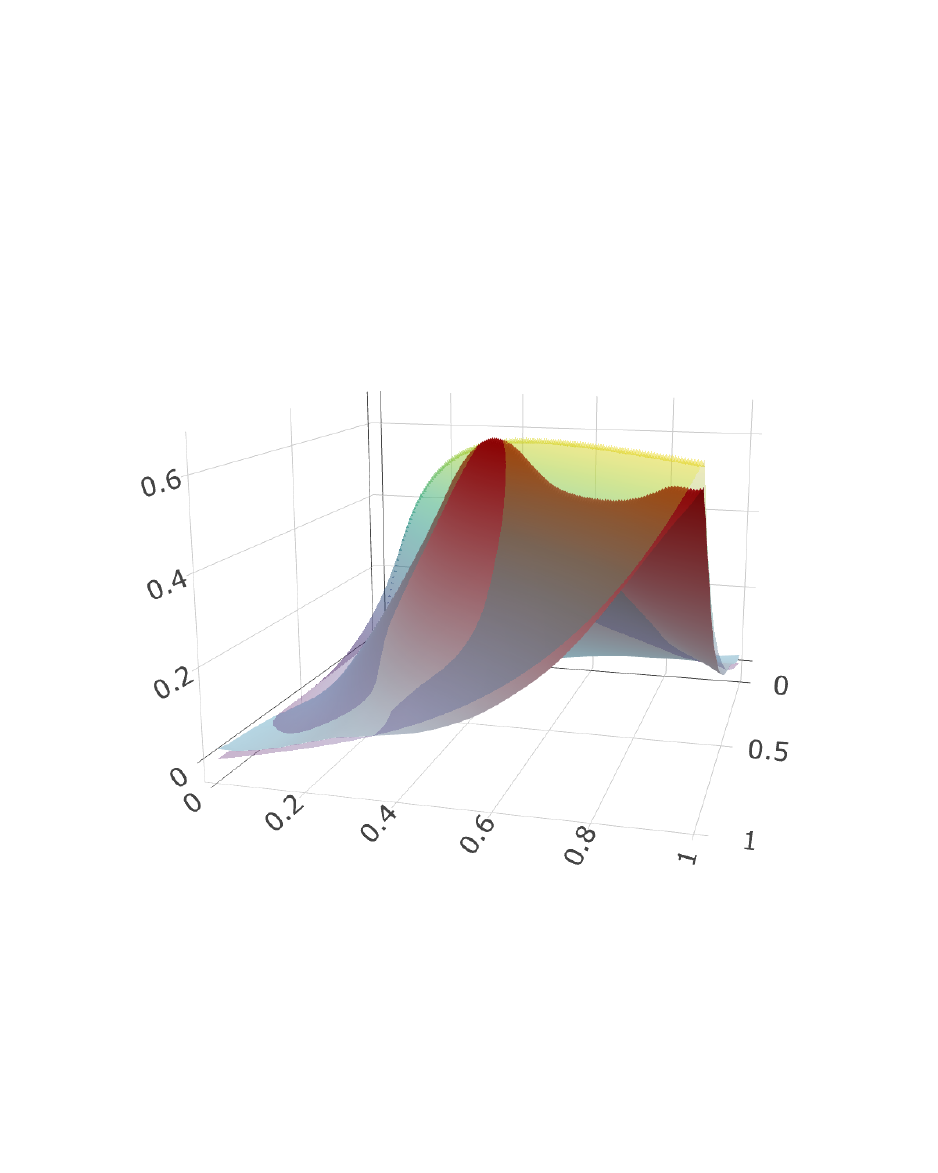}
        \caption{\small \href{https://mbrgr.github.io/from-dense-to-sparse-fda/figure32a.html}{Estimation} with $\hat \Gamma_{100, 40}(\cdot; 0.3)$.}
		\label{fig:plot_comp_OU_a}
    \end{subfigure}
\hfill
    \begin{subfigure}{0.49\linewidth}
        \includegraphics[width=\linewidth]{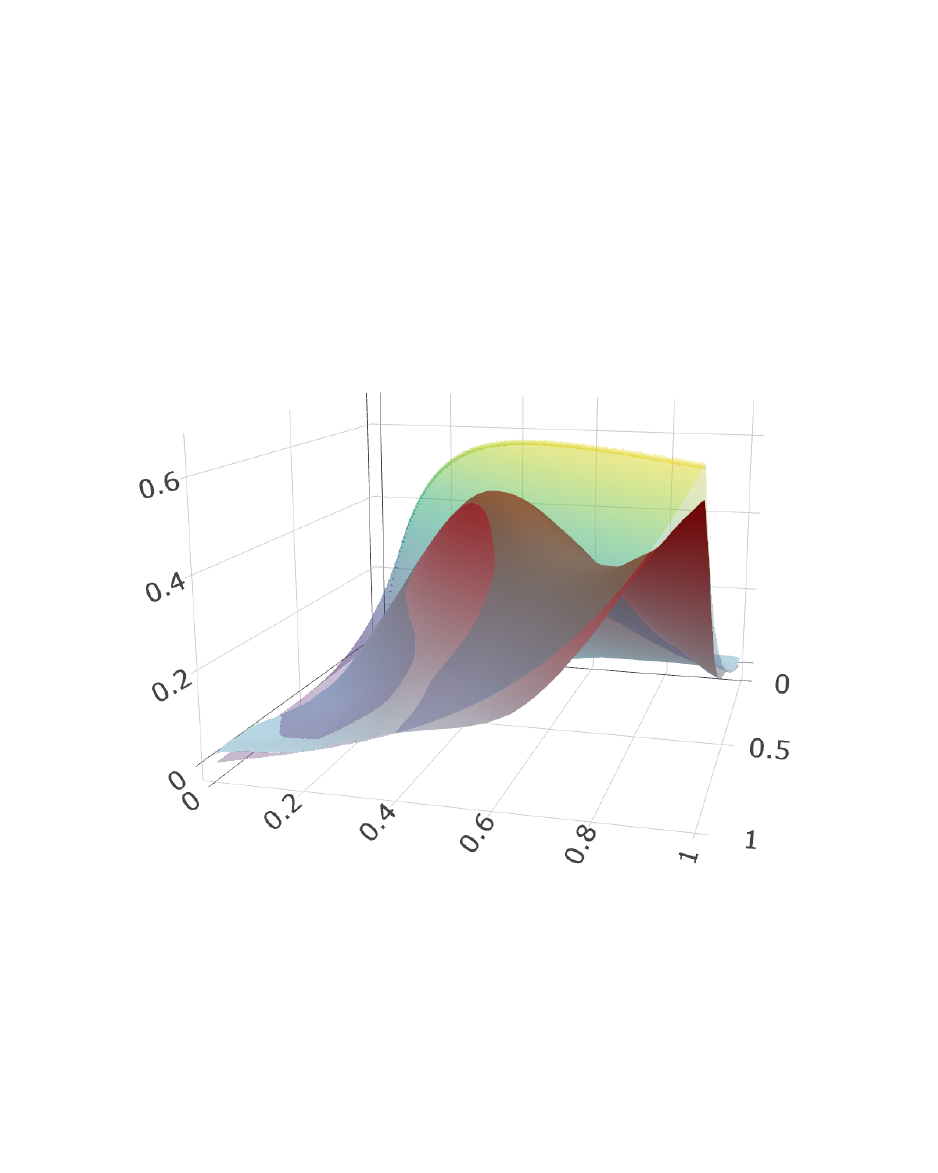}
        \caption{\small \href{https://mbrgr.github.io/from-dense-to-sparse-fda/figure32a.html}{Estimation} with $\hat \Gamma_{100, 40}^{\neq}(\cdot; 0.2)$.}
        \label{fig:plot_comp_OU_b}
    \end{subfigure}%
    \caption{\small Comparison of the estimators (blue to red plane) $\hat \Gamma_{100, 40}(\cdot; 0.3)$ and $\hat \Gamma_{100, 40}^{\neq}(\cdot; 0.2)$, both with local linear weights, of the covariance kernel $\Gamma_{\mathrm{OU}}$ (green) based on the observations of Figure \ref{fig:OU_observations}.}
    \label{fig:plot_comp}
\end{figure}

\subsubsection*{Bandwidth selection}

First we investigate the effect of bandwidth selection on the performance of our variant \eqref{eqn:minimisationProblemLocPol} of the local polynomial estimator, where we restrict ourselves to order $m=1$, for sample size $n=400$ and 
$p \in \{15, 25, 50, 75, 100\}$. For each value of $p$ we use $N = 1000$ repetitions and calculate the supremum norm error of $\hat \Gamma_{n,p}$ for $h$ varying over a grid of bandwidths up to $1$. The results are displayed as curves in $h$ in Figure \ref{fig:bw_comp_OU}. Similar results for $n=50, n = 100$ and $n=200$ are given in the supplementary material \cite[Figure 2]{berger2024covsupplementary}.  
One observes that smoothing is an important part of estimation but should not be overdone. In our setting, for $n=400$ the bandwidths that lead to the smallest overall error become slightly smaller with increasing $p$ and are always between $0.2$ and $0.4$. 

\begin{figure}[t!]
	\begin{minipage}{.48\linewidth}
		\centering
		\includegraphics[width=\linewidth]{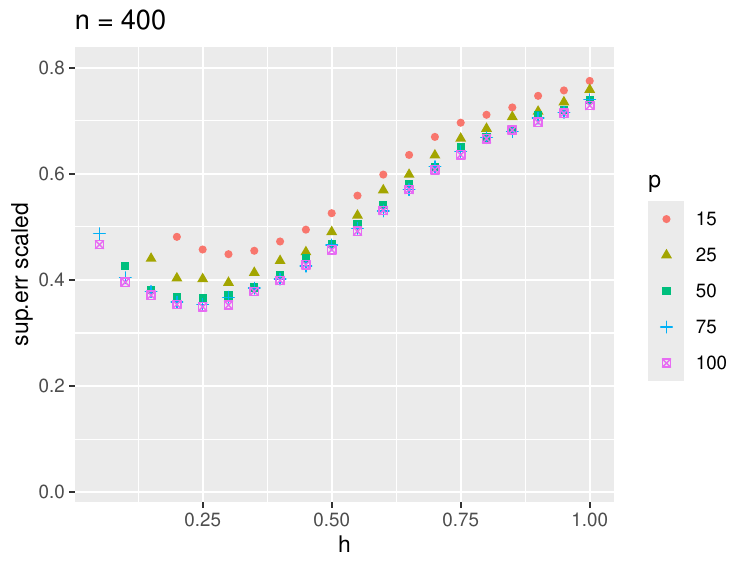}
		\caption{\small {\color{black} Scaled supremum norm error $\norm{\hat\Gamma_{n,p} - \Gamma_{\mathrm{OU}}}_\infty/\norm{\Gamma^{\mathrm{OU}}}_\infty$} of our estimator with $m=1$ for estimating the covariance kernel of the OU-process \eqref{eq:covkerOU} with different bandwidths $h$. }
		\label{fig:bw_comp_OU}
	\end{minipage}
\hspace{0.2cm}
	\begin{minipage}{.48\linewidth}
		\centering
		\includegraphics[width=\linewidth]{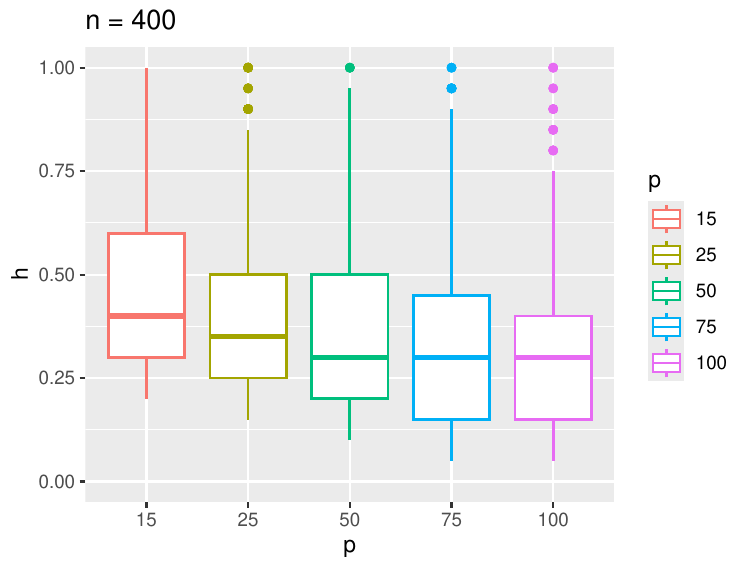}
		\caption{\small \color{black}Boxplots of the selected bandwidths with the five-fold cross validation procedure in the  simulation setup of Figure \ref{fig:bw_comp_OU}.\\\phantom{h}}
		\label{fig:5fold_boxplot}
	\end{minipage}
\end{figure}

Next we investigate a $K$-fold cross validation procedure for selecting $h$.  We proceed as follows: The $n$ observed curves are split randomly into $K$ groups of approximately the same size. One of the groups is used as test data from which  we calculate the empirical covariance matrix $(Z_{jk}^{\mathrm{test}, r})_{j,k = 1, \ldots, p}$. The empirical covariance matrix $(Z^{\mathrm{train}, -r}_{jk})_{j,k = 1,\ldots, p}, r = 1,\ldots, K,$ based on the remaining $K-1$ groups will be used as input  data to our estimator. The procedure requires a grid of bandwidths $h_l$, $l = 1, \ldots, m$ and for every bandwidth $h_l$ we evaluate the local polynomial estimator in \eqref{eqn:minimisationProblemLocPol} $K$-times with each group once as the test set. We take the mean of the $K$ sup norm errors for each bandwidth $h_l$,
\begin{align*}
    \mathrm{CV}(h_l) & = \frac1K \sum_{r = 1}^K \max_{ 1 \leq j < k \leq  p} \absb{\hat \Gamma_{n,p}(x_j, x_k; h_l, Z_{j,k}^{\mathrm{train}, -r}) - Z_{j,k}^{\mathrm{test}, r}}\,.
\end{align*}
Finally we choose the bandwidth with the minimal average sup-norm error, 
\begin{align*}
    h^{\mathrm{cv}} = \argmin\big\{\mathrm{CV} (h_l)\mid h_l, l = 1,\ldots,m\big\}\,.
\end{align*}
We repeat this procedure $N = 1000$ times for $n = 400$ and each $p \in \{15, 25, 50, 75, 100\}$. The results are displayed in Figure \ref{fig:5fold_boxplot}. 
Overall, when compared to the optimal bandwidths visible in Figure \ref{fig:bw_comp_OU}, the cross-validation procedure chooses reasonable but somewhat large bandwidths, which may reflect the additional variability from the estimate in the test data.  


\subsubsection*{Contribution of terms in the error decomposition}

Next we empirically examine the order that the various terms in the error decomposition \eqref{eq:decomp:sketch} have in the supremum norm. Again for each combination of $n \in \{100, 200, 400\}$ and $p \in \{15, 25 ,50,100\}$ we simulate $N = 1000$ repetitions, where we use optimal bandwidths $h_{n,p}$ for the overall supremum-norm error determined by a grid search. 
The results are displayed in Figure \ref{fig:error_decomposition}. Overall, the main contribution to the sup-norm error is from the term involving the processes $Z_i$, which of course decays in $n$ but is not sensitive to $p$. The other terms also decrease somewhat with increasing $p$. Note that individual errors do not add up to but exceed the overall error. {\color{black} This is plausible since they correspond to individual terms in an upper bound for the overall sup-norm error that arises from \eqref{eq:decomp:sketch} by applying the triangle inequality.} 


\begin{figure}
    \centering
    \includegraphics[width = \linewidth]{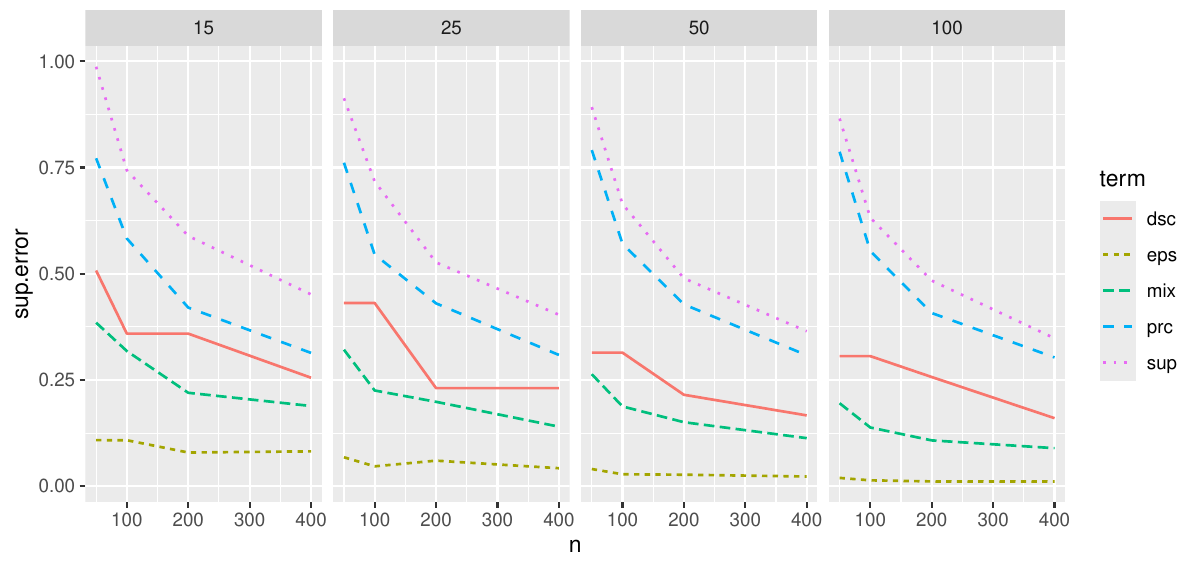}
    \caption{\small Error Decomposition of {\color{black}$\norm{\hat \Gamma_n - \Gamma^{\mathrm{OU}}}_\infty/\norm{\Gamma^{\mathrm{OU}}}_\infty$} where  
    \texttt{dsc} is \eqref{eq:decomp:bias}, \texttt{eps} is \eqref{eq:decomp:eps2}, \texttt{mix} is \eqref{eq:decomp:epsZ}, \texttt{prc} is \eqref{eq:decomp:stochasticError} and \texttt{sup} is the overall supremum error of the estimator. The terms in \eqref{eq:decomp:independendTerms} are not included.}   \label{fig:error_decomposition}  
\end{figure}


\subsubsection*{Comparison of $\hat \Gamma_n$ and   $\hat \Gamma^{\neq}_n$ in \eqref{eqn:estimatorCovariance_full}}

\begin{figure}[t!]
    \begin{subfigure}{0.45\linewidth}
        \includegraphics[width = \linewidth]{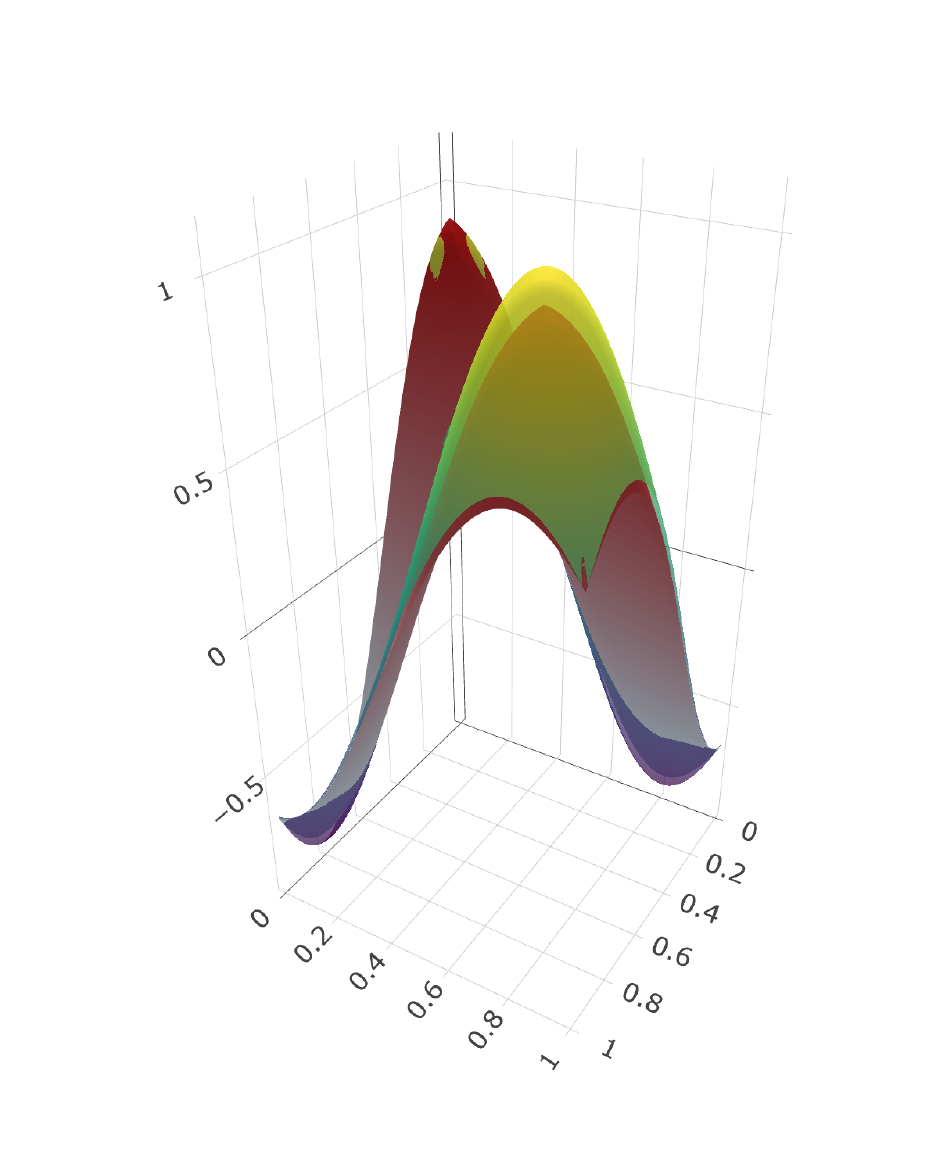}
		\caption{\small\href{https://mbrgr.github.io/from-dense-to-sparse-fda/figure37a.html}{Estimation} with $\hat \Gamma_{100, 40}(\cdot; 0.2)$.}
    \label{fig:plot_comp_2rv_a}
    \end{subfigure}
\hfill
    \begin{subfigure}{0.45\linewidth}
        \includegraphics[width = \linewidth]{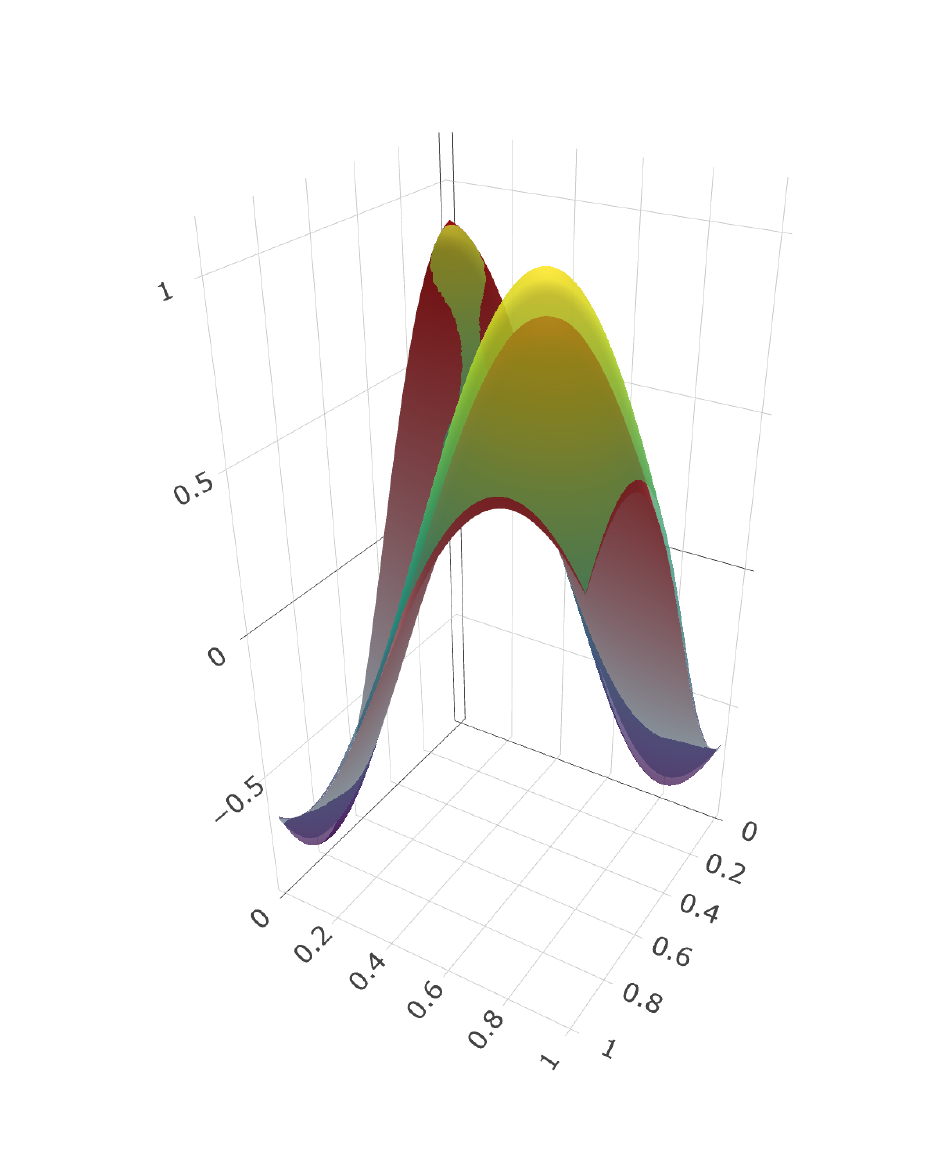}
		\caption{\small \href{https://mbrgr.github.io/from-dense-to-sparse-fda/figure37b.html}{Estimation} with $\hat \Gamma_{100,40}^{\not =}(\cdot; 0.2)$.}
    \label{fig:plot_comp_2rv_b}
    \end{subfigure}%
    \caption{\small {\color{black} Estimates (red) and the true covariance kernel $\tilde \Gamma$ (green / yellow).} For $\hat \Gamma_{100, 40}(\cdot; 0.2)$ a small artificial kink appears on the diagonal of the estimator. The overall estimation quality is about equal. }
    \label{fig:plot_comp_2rv}
\end{figure}

Finally let us compare our estimator $\hat \Gamma_n$ in more detail with the estimator $\hat \Gamma^{\neq}_n$ in \eqref{eqn:estimatorCovariance_full} which uses smoothing over the diagonal. 
%
In addition to the  Ornstein-Uhlenbeck process which has a kink on the diagonal we also consider simulations from the following process 
\begin{align*}
    \tilde Z(x) & \defeq \frac23 \,N_1 \,\sin(\pi\,x) + \sqrt 2\,\frac{2}{3}\,N_2\,\cos(5\,\pi\,x/5)\,,
\end{align*}
where $N_1$ and $N_2$ are independent standard normal distributed random variables. $\tilde Z$ has a smooth covariance kernel  given by
\begin{align*}
    \tilde \Gamma(x,y) = \frac49 \,\sin(\pi\,x) \,\sin(\pi\,y) + \frac89 \,\cos(4\,\pi\,x/5)\,\cos(4\,\pi\,y/5)\,.
\end{align*}



Figure \ref{fig:plot_comp_2rv} displays estimates of the kernel $\tilde \Gamma$ for a particular sample with $n=100$ and $p=40$ using $\hat \Gamma_{100, 40}(\cdot;0.2)$ (Figure \ref{fig:plot_comp_2rv_a})  as well as using $\hat \Gamma_{100, 40}^{\neq}(\cdot;0.2)$ (Figure \ref{fig:plot_comp_2rv_b}), both with local linear weights. While the overall quality of estimation appears be similar, $\hat \Gamma_{100, 40}(\cdot;0.2)$ has a slight artificial kink on the diagonal.


\begin{figure}[b!]
    \begin{subfigure}{0.49\linewidth}
        \includegraphics[width=\linewidth]{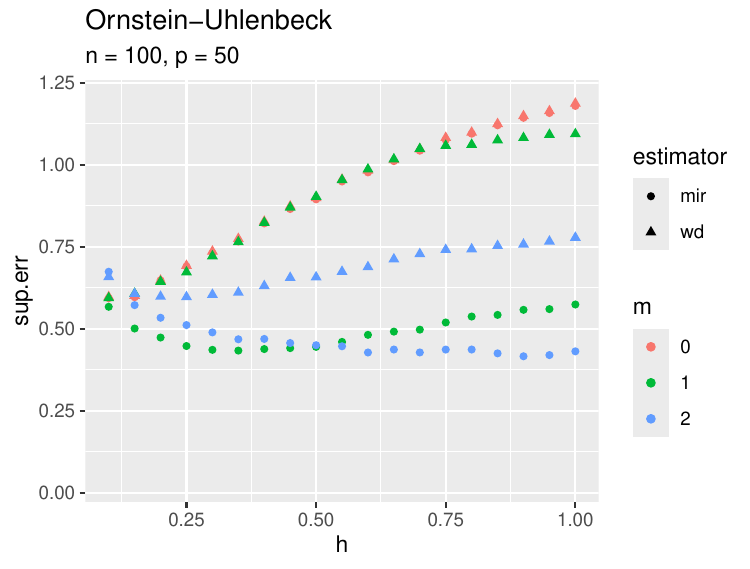}
		\caption{\small Estimation of $\Gamma_{\mathrm{OU}}$. }
    \label{fig:est_comp_2rv_a}
    \end{subfigure}
\hfill
    \begin{subfigure}{0.49\linewidth}
        \includegraphics[width=\linewidth]{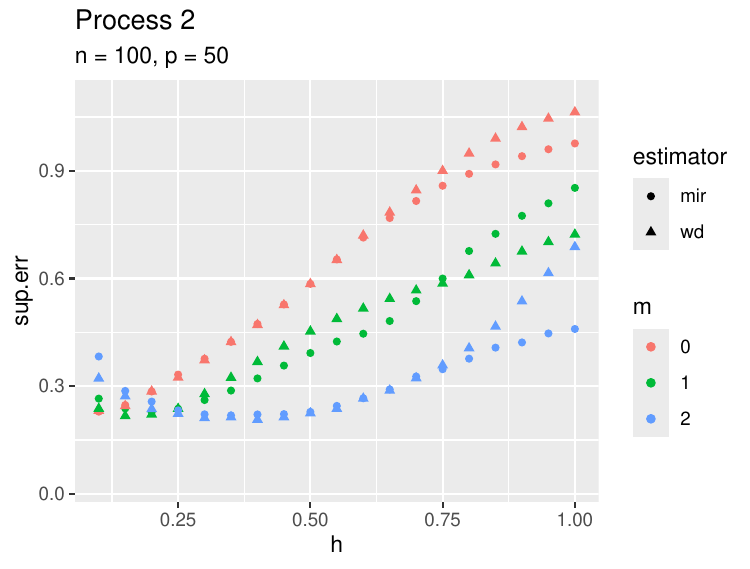}
		\caption{\small Estimation of $\tilde \Gamma$.}
    \label{fig:est_comp_2rv_b}
    \end{subfigure}%
    \caption{\small  Estimation of $\tilde \Gamma$ with $\hat \Gamma_n$ (`mir': mirrored) and with $\hat \Gamma^{\neq}_n$ (`wd': without diagonal)  for different bandwidths, and orders $m$ of the local polynomial weights.}
    \label{fig:est_comp}
\end{figure}

Finally we use $N=1000$ repeated simulations for $n = 100$ and $p = 50$ for estimating $\Gamma_{\mathrm{OU}}$ as well as $\tilde \Gamma$ with both $\hat \Gamma_{n,p}$ and $ \hat \Gamma_{n,p}^{\neq}$ for a grid of bandwidths and orders $m=0,1,2$ of the local polynomial weights. 
The results are display in Figure \ref{fig:est_comp_2rv_a} for estimating $\Gamma_{\mathrm{OU}}$ and in Figure \ref{fig:est_comp_2rv_b} for $\tilde \Gamma$. For the Ornstein-Uhlenbeck process $\Gamma_{\mathrm{OU}}$ the estimators perform similarly for $m= 0$ but there are major differences for the local linear ($m=1$) and local quadratic ($m=2$) case. In particular using the local linear estimator $\Gamma_{n}^{\neq,1}$ is way worse then $\hat \Gamma_n^{1}$. Note that in this case not much smoothing is necessary. In Figure \ref{fig:est_comp_2rv_b} the target kernel is smooth on the whole of $[0,1]^2$ and needs more smoothing to obtain good results. Here both estimators should have the same rates of convergence, which is confirmed in the numerical example. A further advantage of $\hat \Gamma_n$ in comparison to $\hat \Gamma_n^{\neq}$ is the lower computational cost, since only $p\,(p-1)/2$ observations (instead of $p\,(p-1)$) are needed.


\subsection{Weather data in Nuremberg}\label{sec:realdata}

\begin{figure}[b!]
    \begin{subfigure}{0.49\linewidth}
        \includegraphics[width=\linewidth]{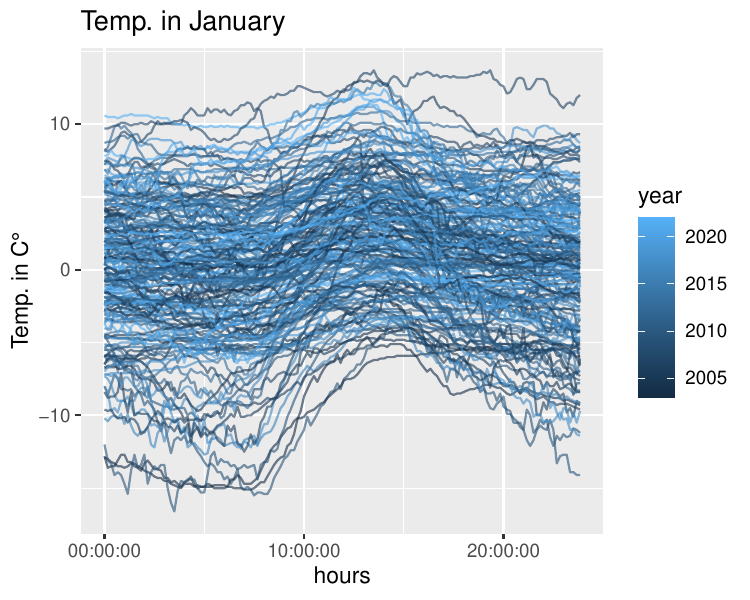}
        \caption{\small January.}
		\label{fig:weather_curves_january}
    \end{subfigure}
\hfill
    \begin{subfigure}{0.49\linewidth}
        \includegraphics[width=\linewidth]{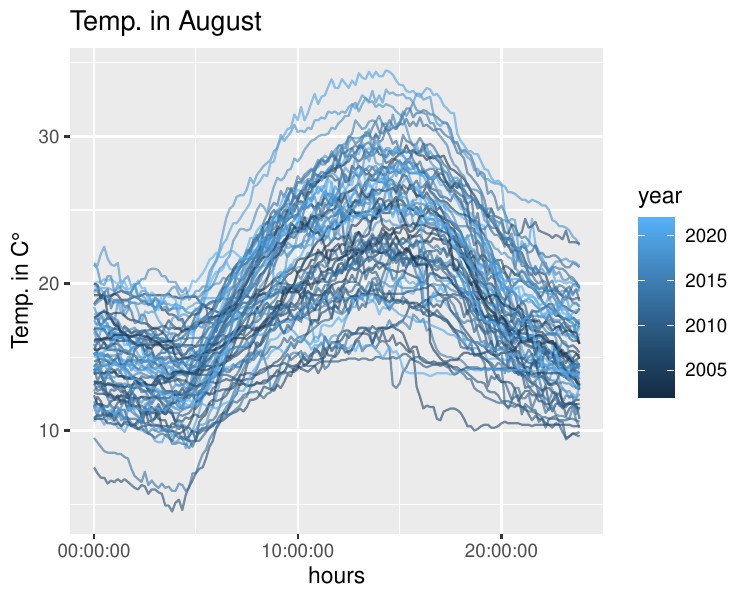}
        \caption{\small August.}
		\label{fig:weather_curves_august}
    \end{subfigure}%
    \caption{\small Temperature curves in January measured in $10$ minute intervals from $2002$ until $2022$. Note that the $y$-axis differ.}
    \label{fig:weather_curves}
\end{figure}

We consider daily temperature curves in each month from the years $2002$ up to $2022$. The data is obtained from the \textit{Deutscher Wetter Dienst (DWD)} at \href{https://opendata.dwd.de/climate_environment/CDC/observations_germany/climate/10_minutes/air_temperature/historical/}{[Link]}. This particular data set was already used in \cite{berger2023dense} to investigate the effect of sparse or dense design on the mean estimation. The observations on each day are taken on a ten minute grid. To reduce the day to day time-series dependency only every third or fourth day, the $1, 4, 8, 12, 15, 18, 22, 25$ and $29^\mathrm{th}$ of every month, was used. This results in around $n = 180$ observations for each month, except for February where $n = 165$. Figure \ref{fig:weather_curves} shows the daily weather curves for January and August. 

We use our estimator in \eqref{eqn:estimatorCovariance} for the covariance kernel and derive estimates {\color{black} $\hat \Gamma^{1/2}(x,x;h)$ of the standard deviation curves $x \mapsto \Gamma^{1/2}(x,x)$ as well as of the correlation surfaces $\Gamma(x,y)/\big(\Gamma(x,x)\, \Gamma(y,y) \big)^{1/2}$.} Results for the standard deviation curves with the bandwidths $144, 288, 720$ for the months January and August can be seen in Figure \ref{fig:std_deviation}. Since a day has $1440$ minutes the bandwidths $144, 288, 720$  are equivalent to $0.1, 0.2, 0.5$ on the unit interval.

\begin{figure}[t!]
    \centering
    \includegraphics[width = 0.75\linewidth]{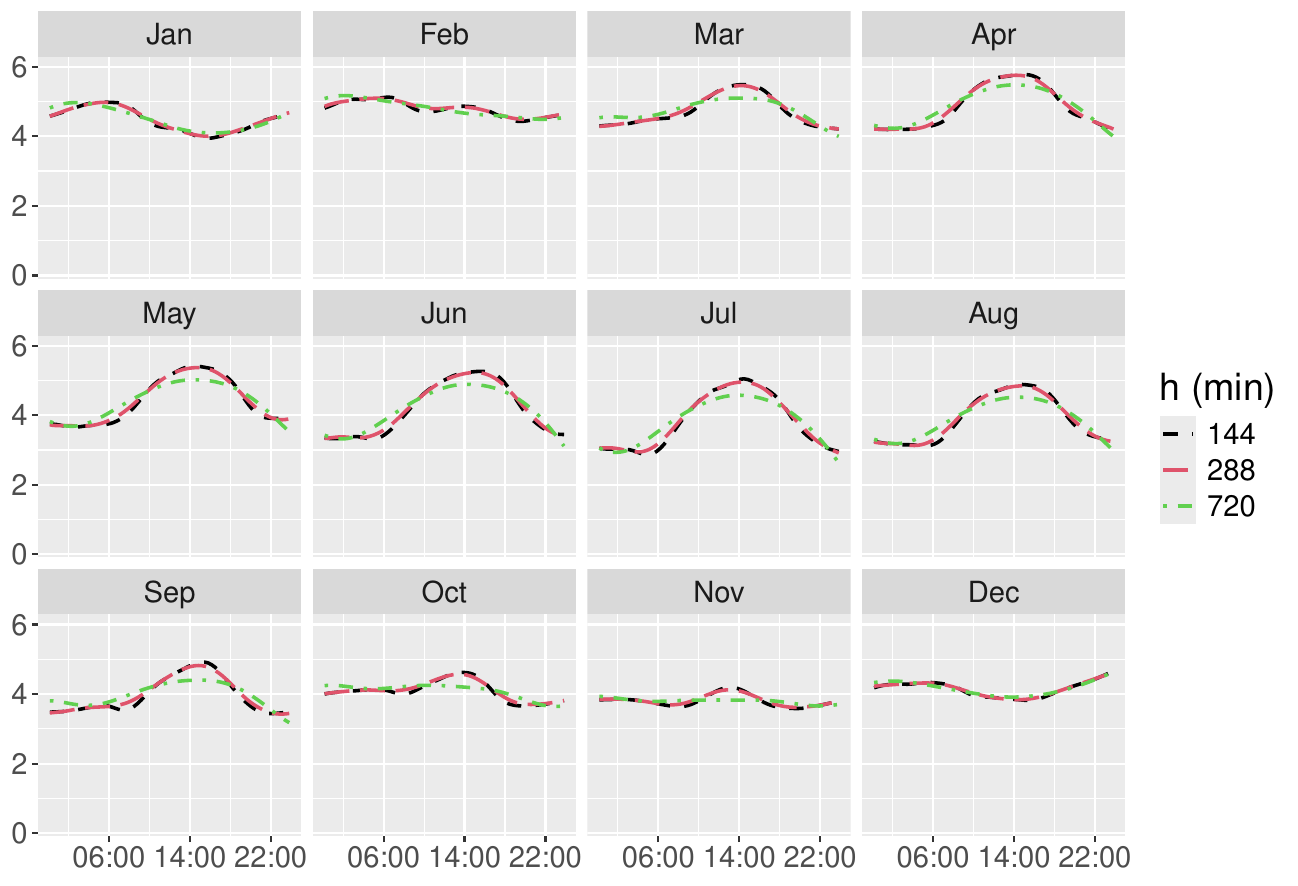}
    \caption{\small Estimation of the standard deviation of the temperatures with three different bandwidths. }
    \label{fig:std_deviation}
\end{figure}

\begin{figure}[b!]
    \begin{subfigure}{0.49\linewidth}
		\includegraphics[width=.8\linewidth]{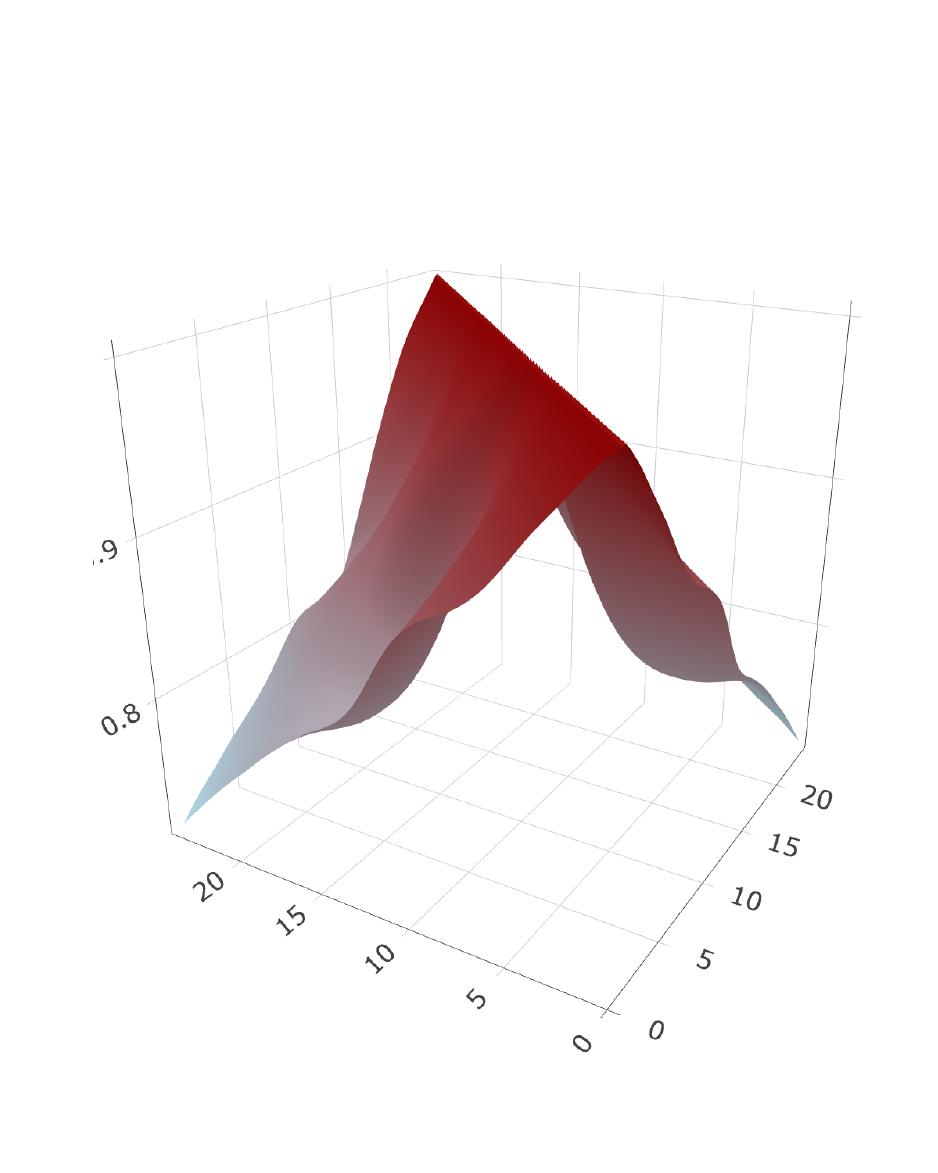}
		\caption{\small January}
		\label{fig:cor_january}
    \end{subfigure}
\hfill
    \begin{subfigure}{0.49\linewidth}
		\includegraphics[width=.8\linewidth]{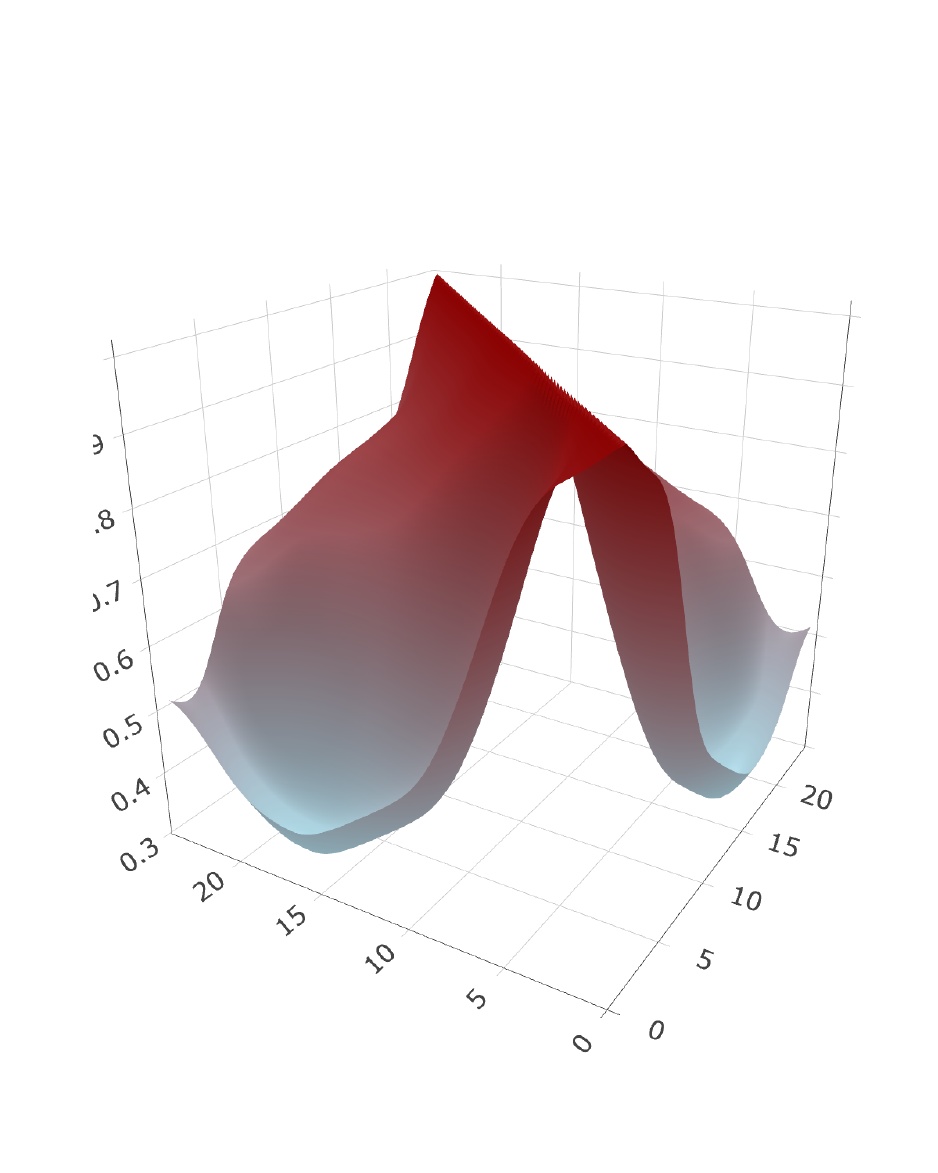}
		\caption{\small August}
		\label{fig:cor_august}
    \end{subfigure}%
    \caption{\small Estimation of the correlation of the temperature during the day. Note that the $z$-axis on the left ranges from $0.6$ to $1$ and on the right from $0.4$ to $1$.}
    \label{fig:cor}
\end{figure}

Next to the notable temperature difference the standard deviation structure seems to be different as well: The winter days have higher standard deviations at night, while the summer days vary more across day time. 
Although the smoothing with $h = 720$ seems a little much, the results are still reasonable. 
Overall the months from February to  April have the highest standard deviations. 
Warm days are the reason for the characteristic bump in the summer months, which are not visible in the cold months from December to February. In December and January the highest standard deviation is in the night due to extremely cold nights, which can be seen in Figure \ref{fig:weather_curves_january}.

In Figure \ref{fig:cor} we display estimates of the correlation function with bandwidth $h = 288$ for January and August. The winter days in January have a higher correlation of the temperature during the night and day compared to the summer month August. However in August the temperature in the morning correlates highly with the day time temperature. Again this is in line with Figure \ref{fig:weather_curves}. When making these estimations it is crucial to take care of the dependency structure of consecutive days, otherwise a high additional time-series correlation in the corners is captured in the estimate.



\section{Conclusions and discussion}\label{sec:conclude}

As the real-data example shows, extensions of our method to time series data would be of quite some interest. Parameter estimation of mean and covariance kernel and related inference methods under the supremum norm have been intensely investigated in recent years \citep{dette2020functional, dette2021functional, dette2022detecting}. Here however it is commonly assumed that full paths of the processes are observed without errors, and results for discrete observations covering the time series setting seem to be missing in the literature. {\color{black} A time-series setting also requires estimation of the cross-covariance kernels $\cov(Z_1(x), Z_{1+h}(y))$, $h \in \N$, and of the long-run covariance kernel. }

\smallskip
{\color{black} 
Estimates of the covariance kernel serve as the basis for estimating the principle component functions. Here most results are in $L_2$ \citep{cai2010nonparametric, hall2006properties, xiao2020asymptotic, zhou2022theory, belhakem2025minimax}. \citet{li2010uniform} and \citet{zhou2022theory} also have rates in the supremum norm, The expansions developed in \citet{hall2009theory} can be helpful for this problem. Rate optimality in $L_2$ has recently been discussed in \citet{belhakem2025minimax}. However, rates in the supremum norm and a CLT in the space of continuous functions for synchronous design still seem to be lacking, and would certainly be of interest when estimating the principle component functions.}

\section{Proofs} \label{sec:proofs}

\subsection{Proof of Theorem \ref{thm:rates_cov_estimation}}\label{sec:proofupperbound}


%
\allowdisplaybreaks

First let us make the error decomposition in Remark \ref{rem:errordecomp} precise. 

\begin{lemma}[Error decomposition]\label{lem:cov_error_decomp}
	For the estimator $\covest \cdot\cdot h$ in \eqref{eqn:estimatorCovariance}, 
	if the weights 
	 satisfy \ref{ass:weights:polynom} for some $\zeta > 0$ then we have the error decomposition 
	\begin{align}
		\covest xyh -\Gamma(x,y) & = \sum_{j<k}^{p} \wjk xyh \bigg[\frac1n \sum_{i=1}^n \epsilon_{i,j}\epsilon_{i,k}\label{eq:decomp:eps2}\\
		&\quad + \Gamma({x_j},{x_k}) - \Gamma(x,y) \label{eq:decomp:bias} \\
		&\quad  +  \frac1n \sum_{i=1}^n\big(Z_i({x_j})Z_i({x_k}) - \Gamma({x_j},{x_k})\big)\label{eq:decomp:stochasticError}\\
		& \quad + \frac1n\sum_{i=1}^n \big(Z_i({x_j})\epsilon_{i,k} + Z_i({x_k})\epsilon_{i,j}\big)\label{eq:decomp:epsZ}\\
		&\quad -  \frac1{n(n-1)}\sum_{i \neq l}^n \big(\epsilon_{i,j} \epsilon_{l,k} + Z_i({x_j})Z_{l}({x_k}) +Z_{i}({x_j})\epsilon_{l,k} + \epsilon_{i,j}Z_l({x_k}) \big)\bigg].\label{eq:decomp:independendTerms}
	\end{align} 
\end{lemma}

The proof of Lemma \ref{lem:cov_error_decomp} is provided in the supplementary appendix in Section \ref{app:technical}. 

\bigskip

Next observe that combining \ref{ass:weights:sup} and Assumption \ref{ass:design:localization} yields for $(x,y) \in  T$ that
\begin{enumerate}[label=\normalfont{(W5)}, leftmargin=9.9mm]
	\item For a constant $\Csum> 0$ the sum of the absolute values of the weights $ \displaystyle \sum_{j<k}^p \big| \wjk xyh \big| \leq \Csum$. \label{ass:weights:sum}
\end{enumerate}

\begin{lemma}[Upper Bounds for the rates of convergence]\label{lem:rates:convergence}
	%
	Given  $0 < \beta_0 \leq 1$ and $\gamma, L, C_Z>0$, for the estimator $\covest \cdot\cdot h$ in \eqref{eqn:estimatorCovariance} with $n$ and $p$ large enough the following rates of convergence hold, where we abbreviate $\Zclass = \mc P (\gamma;L,\beta_0, C_Z)$.
	\begin{enumerate}
		\item[i).] If the weights satisfy \ref{ass:weights:polynom} with $\zeta = \floor{\gamma}$, \ref{ass:weights:vanish} of Assumption \ref{ass:weights} and \ref{ass:weights:sum}, then
		\begin{align*}
		\sup_{h \in (\nicefrac cp, h_0]}\, \sup_{Z \in \Zclass }\,	\sup_{(x,y) \in T} \, h^{-\gamma}\, \absb{\sum_{j<k}^p \wjk xyh \big(\Gamma(x_j, x_k) - \Gamma(x,y)\big)} = \mc O\big( 1\big)\,,
		\end{align*}
	with constants $h_0, c >0 $ according to Assumption \ref{ass:weights}.
		\item[ii).] If the weights satisfy \ref{ass:weights:vanish}, \ref{ass:weights:sup} and \ref{ass:weights:lipschitz}, then under Assumption \ref{ass:design:localization} we have that
		\begin{align*}
			\expec\Big[ \sup_{(x,y) \in T}\absb{ \sum_{j<k}^p \wjk xyh \frac1n \sum_{i=1}^n \epsilon_{i,j} \epsilon_{i,k} } \Big] &= \mc O \bigg(  \sqrt{\frac{\log(n\,p)}{n\,(p\,h)^2}}\bigg)\,,\\
			\expec\Big[\sup_{(x,y) \in T}\absb{ \sum_{j<k}^p \wjk xyh \sum_{i, l=1}^n \frac{\epsilon_{i,j} \epsilon_{l,k}}{n(n-1)} } \Big] &= \mc O \bigg( \frac{ \log(n\,p)}{n\,p\,h}\bigg)\,,
		\end{align*}
        where the constants in the $\mc O$ terms can chosen uniformly for $h \in (c/p, h_0]$.
		\item[iii).] If the weights satisfy \ref{ass:weights:sum}, then
		\begin{align*}
			\sup_{h \in (\nicefrac cp, h_0]} \sup_{Z \in \Zclass}  & \expec\Big[ \sup_{(x,y)\in T} \absb{ \frac1{n} \sum_{j<k}^p \wjk xyh \sum_{i=1}^n \big(Z_i(x_j)Z_i(x_k) - \Gamma(x_j, x_k)\big) } \Big] =\mc O\big( n^{-\frac12}\big),\\
			\sup_{h \in (\nicefrac cp, h_0]}\sup_{Z \in \Zclass}  & \expec\Big[ \sup_{(x,y) \in T} \absb{\frac1n\sum_{j<k}^p \wjk xyh \sum_{i,l=1}^n  \frac{Z_i(x_j)Z_l(x_k)}{n-1}} \Big] =\mc O\big( n^{-1}\big)\,.
		\end{align*}
		\item[iv).] If the weights satisfy \ref{ass:weights:vanish}, \ref{ass:weights:sup} and \ref{ass:weights:lipschitz} and Assumption \ref{ass:design:localization}, then
		\begin{align*}
			 \sup_{Z \in \Zclass}\expec\Big[ \sup_{(x,y) \in T}\absb{ \sum_{j<k}^p \wjk xyh \frac1n \sum_{i=1}^n Z_i(x_j) \epsilon_{i,k} } \Big] &= \mc O \bigg(\sqrt{\frac{\log(h^{-1})}{n\,p\,h}}\bigg)\,,\\
			 \sup_{Z \in \Zclass}\expec\Big[\sup_{(x,y) \in T}\absb{ \sum_{j<k}^p \wjk xyh \sum_{i \neq l}^n \frac{Z_i(x_j) \epsilon_{l,k}}{n(n-1)} } \Big] &= \mc O \bigg(\sqrt{\frac{\log(h^{-1})}{n\,p\,h}}\bigg)\,,
		\end{align*}
	where the constants in the $\mc O$ terms can chosen uniformly for $h \in (c/p, h_0]$.
	\end{enumerate}
\end{lemma}


\begin{proof}[Proof of Lemma \ref{lem:rates:convergence}]
i). The rate of bias term \eqref{eq:decomp:bias} is obtained by standard arguments. For convenience these are detailed in the supplementary appendix, Section \ref{app:technical}.\\

ii). 
We show the bound for the first term in ii), the second is dealt with in the supplementary appendix, Section \ref{app:technical}.
To deal with the quadratic form 
\begin{equation}
		E_{n,p,h}(v) = \sum_{j<k}^{p} w_{j,k}(v;h) \frac1n\sum_{i=1}^n
  \epsilon_{i,j}\epsilon_{i,k}\,, \quad v\in T\,, \label{eqn:def:Enph}
\end{equation}
the main ingredient is the higher dimensional Hanson-Wright inequality \cite[p. 142]{vershynin2018high} which is stated in Lemma \ref{lem:hanson:wright} at the end of this section. This is combined with a discretisation technique 
and the Lipschitz-continuity of the weights, property \ref{ass:weights:lipschitz}.   

Given $\delta>0$ choose a $\delta$-cover $I_\delta = \{\tau_1,\ldots, \tau_{\floor{\sqrt{C_I}/\delta}}\}^2$ of $T$ of cardinality $\nicefrac{C_T}{\delta^2}$ for an appropriate constant $C_T>0$. We estimate
	\begin{align}
		\norm{E_{n,p,h}}_\infty & = \sup_{v^\prime \in I_\delta} \sup_{\substack{v \in T,\\\norm{v-v^\prime}_\infty \leq \delta}} \big|E_{n,p,h}(v)\big| \leq \sup_{v^\prime \in I_\delta} \sup_{\substack{v \in T,\\\norm{v-v^\prime}_\infty \leq \delta}} \Big(\big|E_{n,p,h}(v) - E_{n,p,h}(v^\prime)\big| + \big|E_{n,p,h}(v^\prime)\big|\Big)\nonumber\\
		& \leq  \sup_{\substack{v, v^\prime \in T,\\\norm{v-v^\prime}_\infty \leq \delta}} \big|E_{n,p,h}(v) - E_{n,p,h}(v^\prime)\big| + \sup_{v^\prime \in I_\delta}\big|E_{n,p,h}(v^\prime)\big|.\label{eq:errordiscrete}
 	\end{align}
For the first term, setting $\bs \epsilon_j \defeq (\epsilon_{1,j}, \ldots, \epsilon_{n,j})^\top$, by \ref{ass:weights:lipschitz} and Assumption \ref{ass:design:localization} we have
	\begin{align}
		 \expec\Bigg[\sup_{\substack{v, v^\prime \in T,\\\norm{v-v^\prime}_\infty \leq \delta}}\big|E_{n,p,h}(v) &- E_{n,p,h}(v^\prime)\big|\Bigg]   = \frac1n\,\expec\Bigg[\sup_{\substack{v, v^\prime \in T,\\\norm{v-v^\prime}_\infty \leq \delta}}\bigg|\sum_{j<k}^{p} \big(w_{j,k}(v;h) - w_{j,k}(v^\prime;h)\big) \skpb{\bs \epsilon_{j}}{\bs \epsilon_{k}}\bigg|\Bigg]\nonumber \\
		& \leq 2\,\Clip\,\Ccard\,\delta\,h^{-1}\, \expec\big[\big|n^{-1}\,\max_{j < k }\skpb{\bs \epsilon_{j}}{\bs \epsilon_{k}}\big|\big]\phantom{\frac12}\tag{by \ref{ass:weights:lipschitz} and Assumption \ref{ass:design:localization}}\\
        & = \mc O \bigg( \sqrt{\frac{\log(p)}{n\,(p\,h)^2}}\bigg)\,, \label{eq:rate_sup_over_delta}
	\end{align}
{\color{black} 
where for the last inequality we take $\delta \defeq 1/(np)$ if $ n \gtrsim \log(p)$ and $\delta = 1/(\sqrt{ \log(p) n}\,p)$ if $\log(p) \gtrsim n$ and used the result from Lemma \ref{lem:rate_skp_subExp} in the supplementary Appendix which bounds the expected value.} 

As for the second term in \eqref{eq:errordiscrete}, in order to apply the Hanson-Wright inequality, Lemma \ref{lem:hanson:wright}, we need to find upper bounds for the Frobenius and operator norm of the matrix $A =  (\wjk xyh\, \ind_{j<k})_{j,k = 1,\ldots, p}$ $ \in \R^{p\times p}.$
%
%
%
	 From the properties \ref{ass:weights:sup} and \ref{ass:weights:sum} of the weights the Frobenius-norm satisfies
	\begin{align}
		\norm{A}^2_F & = \sum_{j < k}^{p}\wjk xyh^2 \leq \frac{\Cmax\,\Csum}{({p\,h})^2}\,, \label{eqn:frobenius:A}
	\end{align}
	For the operator-norm $\norm{A}_{\mathrm{op}} = \max_{j=1,\ldots,p} s_j$, where $s_j, j = 1, \ldots,p,$ are the singular values of $A$, 
 we have that
	\begin{align}
		\norm{A}_{\mathrm{op}} & \leq\big(\norm{A}_1\,\norm{A}_\infty\big)^{\nicefrac12} \leq \frac{\Cmax\,\Ccard}{{p\,h}} \label{eqn:operator:A}
	\end{align}
	by the bounds
 \begin{align*}
 \norm{A}_1 & = \max_{k} \sum_{j} \abs{w_{j,k}(v;h)\, \ind_{j<k}} \leq \frac{\Cmax\,\Ccard}{p\,h}, \quad
 \norm{A}_\infty  = \max_{j} \sum_{k} \abs{w_{j,k}(v;h)\, \ind_{j<k}} \leq \frac{\Cmax\,\Ccard}{p\,h}     
 \end{align*}
 by Assumption \ref{ass:design:localization} and  \ref{ass:weights:sum}. See \citep{turkmen2007some} for a proof of the first inequality in \eqref{eqn:operator:A}.  
	Then the statement of the Hanson-Wright inequality (Lemma \ref{lem:hanson:wright}) yields for every $(x,y)\in T$
	\begin{align*}
		\prob\Big(\big|E_{n,p,h}(x,y)\big| \geq t\Big) 
        &\leq 2\, \exp\Bigg(-\tilde c\,\min\bigg(\frac{t^2\,n\,({p\,h})^2}{K^4\,\Cmax\,\Csum},\; \frac{t\,n\,{p\,h}}{K^2\,\Cmax\,\Ccard}\bigg)\Bigg),
	\end{align*}
	and therefore
    \begin{align}
		\prob\bigg(\sqrt{\frac{n\,({p\,h})^2}{\log(n\,p)}}\big|E_{n,p,h}(x,y)\big| \geq t\bigg) 
		& \leq 2\,\exp\bigg( - \frac tC\, \min\Big(\log(n\,p), \sqrt{n\,\log(n\,p)}\,\Big)\bigg) \label{eqn:HW:application}
	\end{align}
	with $C \defeq K^4\,\Cmax\max(\Csum, \Ccard)/\tilde c$ and $t\geq 1$. First we consider the case $\log(n\,p) \lesssim n$. Taking $\eta \geq 1$ yields
	\begin{align}
		& \expec\Bigg[\sup_{(x,y) \in I_\delta} \bigg|\sqrt{\frac{n\,({p\,h})^2}{\log(n\,p)}}E_{n,p,h}(x,y)\bigg|\Bigg] 
		 =  \int_0^\infty \prob\bigg(\sup_{(x,y) \in I_\delta}\sqrt{\frac{n\,({p\,h})^2}{\log(n\,p)}}\big|E_{n,p,h}(x,y)\big| \geq t\bigg) \dx t\nonumber\\
    %
    %
		& \leq \eta + \sum_{(x,y) \in I_\delta}\int_\eta^\infty \prob\bigg(\sqrt{\frac{n\,({p\,h})^2}{\log(n\,p)}}\big|E_{n,p,h}(x,y)\big| \geq t\bigg) \dx t
		\leq \eta + 2\,C_I\,(np)^2\int_{\eta}^\infty\exp\bigg(- \frac{t\, \log (n\,p)}{C}\bigg) \dx t \nonumber\\
    %
    %
        & = \eta + 2\, C_I\, C\,\frac{(np)^{2-\nicefrac \eta C}}{\log(n\,p)}= \mc O\bigg(1+\frac1{\log\big(n\,p\big)}\bigg)\,, \label{eqn:disc:HW:estimate}
	\end{align}
	by choosing $\eta = 2\,C$. 
    Together with \eqref{eq:rate_sup_over_delta} we conclude that
	\begin{equation} \label{eqn:final_rate_error_term}
		\expec\Big[\normb{E_{n,p,h}}_\infty\Big] = \mathcal{O}\bigg(\sqrt{\frac{\log(n\,p)}{n\,({p\,h})^2}}\bigg)\,.
	\end{equation}
 In the case where $p$ grows exponentially in $n$, meaning $\log(n\,p) \simeq n$, the same calculations with $\delta = 1/(\sqrt{\log(p)\,n}\,p)$ lead to
 \begin{align*}
     \expec\Bigg[\sup_{(x,y) \in I_\delta} \bigg|\sqrt{\frac{n\,({p\,h})^2}{\log(n\,p)}}E_{n,p,h}(x,y)\bigg|\Bigg] &  = \mc O\bigg( 1 + \frac1n\bigg)
 \end{align*}
By choosing $\eta$ large enough such that $p^2 \lesssim \exp( n\,\eta/c)$ we get \eqref{eqn:final_rate_error_term}. If $p$ is of even higher order limiting the amount of observation points concludes the proof of the first statement of \textit{ii)}.\\

iii). Again we focus on the first bound, the second can be obtained similarly, details are provided in Section \ref{app:technical} in the supplementary appendix.
%
	First note that by \ref{ass:weights:sum}, we may bound 
	\begin{align}
	 \expec\bigg[\sup_{(x,y) \in T}\Big|\frac1{\sqrt n} \sum_{j<k}^{p} \wjk xyh &\sum_{i=1}^n \big(Z_i({x_j})Z_i(x_k) - \Gamma(x_j, x_k)\big)\Big|\bigg]\notag \\
     \leq  & \Csum \, \expec\bigg[\sup_{(x,y) \in T}\Big|\frac1{\sqrt n} \sum_{i = 1}^n (Z_i(x)Z_i(y) - \Gamma(x,y))\Big|\bigg].\label{eq:uppunifpollardfirst}
\end{align}
 To bound \eqref{eq:uppunifpollardfirst} we shall apply the maximal inequality \citet[Section 7, display (7.10)]{pollard1990empirical}, see also \citet[Lemma 4]{berger2023dense}.
    %
%
By definition of the class of processes $Z_i \in \mc P (\gamma;L,\beta_0, C_Z)$, we have $\Gamma \in \mc H_{T}(\gamma, L)$ so that $\Gamma$ is uniformly upper bounded by $L$, and $|Z_i(x)| \leq |Z_i(0)| + M_i$. Therefore
a square-intergable envelope $\Phi_{n,i}$ of $(Z_i(x)Z_i(y) - \Gamma(x,y))/\sqrt n$ is given by
	\begin{align*}
		\frac1{\sqrt n}\abs{Z_i(x)Z_i(y) - \Gamma(x,y)} &\leq \frac1{\sqrt n}\big(2\,Z_i^2(0) + 2\, M_i^2+ L  \big)\defeql \Phi_{n,i}.
	\end{align*}

 To check the assumption of manageability of $(Z_i(x)Z_i(y) - \Gamma(x,y))/\sqrt n$ in the sense of \citep[see Definition 7.9, p. 38]{pollard1990empirical}, see also \citet[Section 8.3]{berger2023dense}, we note that  from the Hölder continuity of $\Gamma$ as well as of the paths of $Z_i$ in \eqref{eq:hoeldercontpathsZ} the triangle inequality gives
	\begin{align*}
		\frac1{\sqrt n}\Abs{Z_i(x)Z_i(y) - \Gamma(x,y) - \big(Z_i(x^\prime)Z_i(y^\prime) - \Gamma(x^\prime, y^\prime)\big)} & \leq \, 2 \,\Phi_{n,i}\, \max\big(\abs{x-x^\prime},\abs{y - y^\prime}\big)^{\min(\beta, \gamma)}.
	\end{align*}
Then \citet[Lemma 3]{berger2023dense} implies manageability. 

Now \citet[Section 7, display (7.10)]{pollard1990empirical}, see also \citet[Lemma 4]{berger2023dense}, implies that \eqref{eq:uppunifpollardfirst} is upper bounded by 
$\frac1n \sum_{i=1}^n \expec[\Phi_{i,n}] \leq \text{const.}$
   uniformly over the function class $\mc P (\gamma;L,\beta_0, C_Z)$. 
	

\smallskip

iv). Again we focus on the first bound, for the second details are provided in Section \ref{app:technical} in the appendix.
Given $\bs z = (z_{i,j})$, $i=1, \ldots, n$, $j=1, \ldots, p$ let
%
%
$$S_{\mid \bs z}(x,y) \defeq \sum_{j<k}^p \wjk xyh \frac1n \sum_{i=1}^nz_{i,j}\epsilon_{i,k}.$$
By conditioning on $\bs Z = (Z_i(x_j))$ we obtain 
$$
\expec\big[\expec[\sup\abs{S_{\mid \bs Z}(x,y)}|\bs Z]\big] = g(\bs Z),\quad g(\bs z) = \expec \big[  \sup_{(x,y)\in T}\abs{S_{\mid \bs z}(x,y)} \big].$$
%

To apply Dudley's entropy bound \cite[p.~100]{van1996weak}  we show that for given $\bs z$ the process $S_{\mid \bs z}(x,y) $ is sub-Gaussian with respect to the semi norm 
$$d_{\mid \bs z}\big((x,y), (x^\prime, y^\prime)\big) = \zeta\, \expec[(S_{\mid \bs z}(x,y) - S_{\mid \bs z}(x^\prime, y^\prime))^2]^{1/2},\quad (x,y),\, (x^\prime, y^\prime) \in T,$$
where $\zeta>0$ is as in Assumption \ref{ass:distribution}. From  Assumption \ref{ass:distribution}, for $\lambda \in \R$
\begin{align*}
    \expec\Big[ \exp\big(\lambda\, \big(S_{\mid \bs z}(x,y) &- S_{\mid \bs z}(x^\prime, y^\prime)\big)\Big]\\ & = \prod_{i= 1}^n \prod_{k= 2}^p\expec \Bigg[  \exp \bigg(\lambda\, \Big( \sum_{j= 1}^{k-1} \big( \wjk xyh - \wjk{x^\prime}{y^\prime}{h}\big) \,\frac{z_{i,j}}n \Big) \epsilon_{i,k}\bigg)\Bigg] \\
     &\leq  \exp\bigg( \frac{\lambda^2}{2} \sigma^2 \zeta^2 \sum_{i=1}^n \sum_{k=2}^{p} \Big( \sum_{j = 1}^{k-1} \big( \wjk xyh - \wjk {x^\prime}{y^\prime}h\big) \frac{z_{i,j}}n \Big)^2 \bigg)\\
    &=  \exp\bigg( \frac{\lambda^2}{2}\, \, d_{\mid \bs z}^2\big((x,y), (x^\prime, y^\prime)\big) \bigg)\,.
\end{align*}
To upper bound $d_{\mid \bs z}^2 \big((x,y), (x^\prime, y^\prime)\big)$, note that by Assumption \ref{ass:design:localization} and \ref{ass:weights:vanish} for at most $2\Ccard ph$ indices $j$ and $k$ respectively the increment of the weights can be non-zero. 
Using the property \ref{ass:weights:lipschitz} for these differences we obtain the bound
%
\begin{align}
    d_{\mid \bs z}^2 \big((x,y), (x^\prime, y^\prime)\big) & \leq \zeta^2\,\sigma^2\,  \sum_{i = 1}^n  \,2\Ccard ph  \Big(2\Ccard ph \,\frac{1}{(p\,h)^2} \Clip\, \Big(\frac{\max(\abs{x-x^\prime}, \abs{y-y^\prime})}h\wedge 1\Big)\, \frac{m_i}{n}\, \Big)^2\nonumber\\
    & \leq 8\, \Clip^2\,\Ccard^3 \,\frac{\zeta^2\,\sigma^2}{n^2\,p\,h}\,  \sum_{i = 1}^n m_i^2 \,  \bigg(\frac{\max(\abs{x-x^\prime}, \abs{y-y^\prime})}h\wedge 1\bigg)^2.\label{eq:boundsquareddist}
\end{align}
The diameter of $[0,1]^2$ under $d_{\mid \bs z}$ is then upper bounded by
\begin{align*}
    \diam_{d_{\mid \bs z}}([0,1]^2)^2 = \sup_{x,x^\prime, y, y^\prime \in [0,1]} d_{\mid \bs z}^2\big((x,y), (x^\prime, y^\prime)\big) \leq 8\, \Clip^2\,\Ccard^3 \,\frac{\zeta^2\,\sigma^2}{n^2\,p\,h}\,  \sum_{i = 1}^n m_i^2 \, \defeq \Delta_{\bs m}\,, 
\end{align*}
where $\bs m = (m_1, \ldots, m_n)^\top$, and the packing number is upper bounded by
 	$D\big( [0,1]^2,\delta;d_{\mid \bs z} \big)  
    \leq  \Delta_{\bs m} (\delta\, h)^{-2}.$

Dudley's entropy bound for sub-Gaussian processes \citep[p.~100]{van1996weak} yields for $(x_0, y_0) \in T$ that
 	\begin{align} \label{eq:dudley:bound}
 		g(\bs z ) = \expec\Big[\sup_{(x,y) \in T}\absb{S_{\mid \bs z}(x,y)}\Big] \leq \expec\Big[ \big| S_{\mid \bs z}(x_0, y_0) \big| \Big]+ K\, \int_0^{\diam_{d_{\mid \bs z}}([0,1]^2)} \sqrt{\log \Big( D\big( [0,1]^2,\delta \big) \Big)} \,\dx \delta\,.
 	\end{align}
Using $\int_0^a \sqrt{\log (x^{-1})}\dx x = a\sqrt{-\log (a)} + a/(2\sqrt{-\log (a)})$
 for $0<a<1$ and the above bound on the packing number we get that
\begin{align}
    \int_0^{\diam_{d_{\mid \bs z}}([0,1]^2)} \sqrt{\log \Big( D\big( [0,1]^2,\delta\big) \Big)} \,\dx \delta\, &  \leq \int_0^{\Delta_{\bs m}^{1/2}} \sqrt{\log\big(\Delta_{\bs m}\cdot{(\delta {h})^{-2}} \big)}\dx \delta \nonumber\\
    %
    &= \sqrt{2\,\Delta_{\bs m}}\bigg(\sqrt{- \log(h)} + \frac 1{2\sqrt{-\log(h)}}\bigg). \label{eq:dudley:second}
\end{align}
A computation to that leading to \eqref{eq:boundsquareddist} gives the bound
\begin{align}
	\expec\big[ \abs{S_{\mid \bs z}(x_0, y_0)}\big]  & \leq  
	\big(\expec\big[ \abs{S_{\mid \bs z}(x_0, y_0)} ^2\big]\big)^{1/2} 
 \leq \Delta_{\bs m}^{1/2}\,, \label{eq:dudley:first}
\end{align}
Inserting this bound and \eqref{eq:dudley:second} into \eqref{eq:dudley:bound} yields
\begin{align*}
    g(\bs z) & \leq \sqrt{\Delta_{\bs m}} +  K\cdot \sqrt{2\,\Delta_{\bs m}}  \cdot\bigg( \sqrt{-\log(h)} - \frac1{2\sqrt{\log(h)}} \bigg)
\end{align*}
Now from \eqref{eq:hoeldercontpathsZ} it follows that $Z_i(x_j) \leq \abs{Z_i(0)} + M_i$ a.s.. Replacing the deterministic $m_i$ by $\abs{Z_i(0)} + M_i$, using  
%
%
%
 $\expec[(\abs{Z_i(0)} + M_i)^2] \leq 2\expec[Z_i(0)^2 + M_i^2] \leq 2\,C_Z <\infty$ and  Jensen's inequality gives 
\begin{align*}
	\expec\big[ \sqrt{\Delta_{\abs{\bs Z(0)} + \bs M}}\,\big] & \leq 
    \bigg(\Clip^2\,\Ccard^3 \sum_{i=1}^n \frac{\expec[(\abs{Z_i(0)} + M_i)^2]}{n^2}\,\frac{\sigma^2}{p\,h}\bigg)^{1/2}\\
    &\leq \bigg(\frac{2\,C_Z\,\Clip^2\,\Ccard^3\,\sigma^2}{n\,p\,h}\bigg)^{1/2} = \mc O\big((n\,p\,h)^{-1/2}\big)\,,
\end{align*}
so that overall
\begin{align*}
    \expec\Big[\expec\big[\sup_{x,y} \absb{S(x,y)}\mid \bs Z\big]\Big] & \leq \expec\Big[ \sqrt{\Delta_{\abs{\bs Z(0)} + \bs M}}\,\Big] + \expec\Big[ \sqrt{2\Delta_{\abs{\bs Z(0)} + \bs M}}\,\Big]\bigg( \sqrt{-\log(h)} - \frac1{2\sqrt{\log(h)}} \bigg)  \\
    &= \mc O\bigg( \frac1{\sqrt{n\,p\,h}} +  \sqrt {\frac{\log(h^{-1})}{n\,p\,h}} \bigg).
\end{align*}
\end{proof}

\begin{proof}[Proof of Theorem \ref{thm:rates_cov_estimation}]
	Follows by the upper bounds for the rates of convergence from Lemma \ref{lem:rates:convergence} with
%
	$h= \mc O(p^{-1})$ for any feasible sequence $h$.
\end{proof}

To conclude this section we state the Hanson-Wright-Inequality used in the above proof, and give an upper bound for the maximal singular value of a matrix $A \in \R^{p \times p}$ in terms of matrix norms.   

\begin{lemma}[Higher dimensional Hanson-Wright inequality]\label{lem:hanson:wright}
	Let $X_1,\ldots, X_p$ be independent, mean-zero, sub-Gaussian random vectors in $\R^n$. Further let $A = (a_{j,k})_{j,k} \in \R^{p\times p}$ be a matrix. For any $t \geq 0 $ we have
	\begin{align}
		\prob\bigg(\Big|\sum_{j,k = 1}^p {a_{j,k}} \skpb{X_j}{X_k} &- \sum_{j = 1}^p a_{j,j} \expec\skpb{X_j}{X_j}\Big| \geq t\bigg) \nonumber  \\ & \leq 2\, \exp\Bigg(-\tilde c\,\min\bigg(\frac{t^2}{n\,K^4\, \norm{A}_F^2}, \frac{t}{K^2\,\norm{A}_{\mathrm{op}}}\bigg)\Bigg),
	\end{align}
	where the Orlicz-norm $\max_i \norm{X_i}_{\phi_2} = K < \infty$, $\norm{A}_F^2 = \sum_{j,k}^p a_{j,k = 1}^2$ and $\norm{A}_{\mathrm{op}} = \max_i s_i$, where $s_i, i = 1\ldots,p,$ are the singular values of $A$ and $\tilde c > 0 $ is a constant.
\end{lemma}

\begin{proof}
	Follow the proof of \citep[Theorem 6.2.1, p. 139ff]{vershynin2018high} and replace $X^\top A X$ and $\lambda$ by $X_k^\top AX_k$ and $\lambda/d$ respectively. Make use of independence then.
\end{proof}



\subsection{Proof of Theorem \ref{theorem:optimality}}\label{ssec:proof:optimality}

\begin{proof}[Proof of Theorem \ref{theorem:optimality}.]
	
	In the proof we rely on the reduction to hypothesis testing as presented e.g.~in \citet[section 2]{tsybakov2008introduction}.
	In all hypothesis models we set $\mu=0$. 
	
	For the lower bound $  p^{-\gamma}$, using the method of two sequences of hypotheses functions we set 
	$Z_{i;0} = 0$ and 
	construct $Z_{i;1,   p}$ such that its covariance kernel, $\Gamma_{1,   p}$, satisfies $\norm{\Gamma_{1,   p}}_\infty \geq c\,   p^{-\gamma}$ for some constant $c>0$ and that $Z_{i;1,   p}(x_j)=0$ at all design points $  x_{  j}$, so that the distribution of the observations for $Z_{i;0}$ and $Z_{i;1,p}$ coincide. {\color{black} Set $x_0 = 0, x_{p+1}=1$.  Then there is an $l \in \{0,1,\ldots, p\}$ for which $x_{l+1}- x_{l}\geq (p+1)^{-1}$, choose such an $l$. 
   For a constant $\tilde L>0$ to be specified we then set
	$$ f(x) = \tilde L\, p^{-\gamma/2}\, \tilde f\big( 2\, (p+1)\, (x -(x_{l}+ x_{l+1})\,/\,2)\big),\quad \text{where}\quad \tilde f(  x)=  
	\exp\big(-\frac1{1-x^2}\big)\one_{\{|x| < 1\} }.$$
    Note that $f$ has support in $[x_l, x_{l+1}]$ and vanishes at $x_l$ and $x_{l+1}$. 
	Then let
	$ Z_{i;1,p}(  x)= W_i\, f(x), $
	where $W_i \sim \mathcal{N}(0,1)$, which are taken independent over $i$, and note that $Z_{i;1,   p}(  x_{  j})=0$ at all design points. 
	Furthermore, the covariance kernel of $Z_{i;1,p}$ is 
	$$ \Gamma_{1,p}(x,y) = f(x)\, f(y) = \tilde L^2\, p^{- \gamma} 
	\tilde f\big( 2\, (p+1)\, (x -(x_{l}+ x_{l+1})\,/\,2)\big)\, \tilde f\big( 2\, (p+1)\, (y -(x_{l}+ x_{l+1})\,/\,2)\big).$$
	At $x = y = (x_{l}+ x_{l+1})\,/\,2$ this results in a value of 
	$\Gamma_{1,p}\big((x_{l}+ x_{l+1})\,/\,2,(x_{l}+ x_{l+1})\,/\,2\big) = \tilde L^2\, p^{-\gamma}\, e^{-2}, $
	so that $\norm{\Gamma_0 - \Gamma_{1,   p}}_\infty \geq c\, p^{-\gamma}$ holds true. Finally, using the chain rule and the fact that all derivatives of the bump function in the definition of $\tilde f$ are uniformly bounded, $\Gamma_{1,p}$ is $\gamma$-Hölder smooth with constant proportional to $\tilde L^2$, which can be adjusted to yield the Hölder norm $L$. This concludes the proof for the lower bound $  p^{-\gamma}$.  
	}
	
	\vspace{0.3cm}

    {\color{black} For the remaining two settings we take $\sigma_0^2 =1$. 
    
	For the lower bound of order $n^{-1/2}$, consider  $Z_{i,0;n}(x) = W_i$ with $W_i \sim \mathcal N(0,1)$, $i=1, \ldots, n$, and $Z_{i,1;n}(x) = \sigma_{1;n}\, W_i$, $\sigma_{1;n} = 1+ n^{-1/2}$. These have constant covariance kernels $=1$ and $=\sigma_{1;n}^2$ with distance  $n^{-1/2} \, (2+ n^{-1/2})$ of order $n^{-1/2}$.
    Now consider the $p$-dimensional normal distributions $\mathcal N_p(0, \Sigma_{j})$ of  $(Z_{1,j;n}(x_1) + \epsilon_{1,1}, \ldots, Z_{1,j;n}(x_p)+ \epsilon_{1,p})$, $j=0,1$, with $\Sigma_{0} = \ind_p \ind_p^\top + \mathrm{I}_p$ and $\Sigma_{1} =   \sigma_{1;n}^2\, \ind_p \ind_p^\top + \mathrm{I}_p$, where $\ind_p =(1, \ldots, 1)^\top$. Then $\Sigma_1^{-1} = \mathrm{I}_p - \frac{\sigma_{1;n}^2}{1+p\, \sigma_{1;n}^2}\, \ind_p\, \ind_p^\top$, and for the Kullback-Leiber divergence we have that
    \begin{align}
    \KL\big(\mathcal N_p(0, \Sigma_{0})||\mathcal N_p(0, \Sigma_{1})\big) & = \frac12\, \Big(\tr\big(\Sigma_1^{-1}\, \Sigma_0 \big) - p - \log\Big(\frac{\det(\Sigma_0)}{\det(\Sigma_1)} \Big) \Big)\label{eq:multinormkl}\\
       & = \frac12\, \Big(\tr\Big(\mathrm{I}_p + \big(1- \sigma_{1;n}^2\, \frac{1+p}{1 + \sigma_{1;n}^2\, p}\big)\, \ind_p \ind_p^\top \Big) - p - \log\Big(\frac{1+p}{1 + \sigma_{1;n}^2\, p} \Big) \Big)\nonumber\\
       & = \frac12\, \big(\theta_n - \log(1+\theta_n) \tag{$ \theta_n = \frac{p}{1 + \sigma_{1;n}^2\, p}\, (1 - \sigma_{1;n}^2)$}\\ 
       & \leq c\, \theta_n^2 \lesssim n^{-1}, \tag{$n \geq 2$, for some $c>0$}
    \end{align}
as required. 
	\vspace{0.3cm}

	Finally, let us turn to the lower bound of order $(\log(n\, {p})/(n{  p}))^{\gamma/(2\gamma+1)}$. 	Since we have already shown a lower bound of order $  p^{-\gamma}$ we may assume that 
	\begin{equation}\label{eq:regimelowerbound}
		p^{- \gamma} \lesssim (\log(n\, {p})/(n{  p }))^{\gamma/(2\gamma+1)}. 
	\end{equation}

 We will use \cite[Theorem 2.7]{tsybakov2008introduction}. To specify the hypotheses functions, 
	%
 %
	for sufficiently small $c_i$, $i=0,1$ we take $N_{n,  p}= \ceil{c_0\,(np/(\log(n {p})))^{1/(2\gamma+1)}},\, h_{n,  p} = N_{n,  p}^{-1}$ and $s_{n,p}\defeq  c_1\, h_{n,  p}^{\gamma}$ (the desired lower bound for the rate) and for $\tilde L$ to be specified we let
	$$\tilde g(  x) = \tilde L (h_{n,  p}\,/\,2)^{\gamma} \exp\Big(-\big(1-x^2\big)^{-1}\Big)\one_{\{|x|< 1\} }$$
	%

    
	Then for $  l  \in \{1, \ldots, N_{n,  p} \}$ we set $Z_{i;l}(  x) = W_i\,\big( 1+ \tilde g(2\, (  x -  z_l)/h_{n,  p})\big)$ and $z_l = (l-1/2)/N_{n,  p},$ where $W_i$ are independent, standard normally distributed random variables, and $Z_{i;0}(  x) = W_i$. By setting 
    \begin{equation}\label{eq:lowerboudntechhyp}
     g_l(x)\defeq \tilde g(2\, (  x -  z_l)/h_{n,  p})\,,
  \end{equation}
	the covariance kernel of $Z_{i;l}$ can be written as $\Gamma_l(x,y) = \big(1+ g_l(x)\big)\big( 1 + g_l(y)\big)\,.$
	%
 %
	Using the chain rule one checks $\gamma$-Hölder smoothness of each $\Gamma_{l}(x,y)$ for suitable choice of $\tilde L$ (depending on $L$ and $c_0$). Further, by construction it holds that $\supp(\Gamma_l) = [l/N_{n,p}, (l+1)/N_{n,p}]^2$ with $\Gamma_l$ vanishing at the boundary points of its support. Therefore $\norm{\Gamma_{  l} - \Gamma_{   r}}_\infty\geq 2\,s_{n,  p}$ for all $  l \neq   r$, and $c_1$ sufficiently small (depending on $c_0$). 
	
	To show that $s_{n,  p}$ is a lower bound for the rate of convergence, let $\mathcal N_p(0, \Sigma_l)$ denote the joint normal distribution of the $(Z_{1;l}(x_1) + \epsilon_{1,1}, \ldots, Z_{1,l}(x_p) + \epsilon_{1,p})$, so  $\Sigma_0 = \ind_{p \times p} + \mathrm{I}_p$, and $\Sigma_l = v_l\,v_l^\top  + \mathrm{I}_p$, $  l  \in \{1, \ldots, N_{n,  p} \}$, with $ (v_l)_j  = 1 + g_l(x_j)$, $j = 1,\ldots, p$, where $g_l$ is defined in \eqref{eq:lowerboudntechhyp}. Then, similar to \eqref{eq:multinormkl}, 
    \begin{align*}
      &  \KL\big(\mathcal N_p(0, \Sigma_{l})||\mathcal N_p(0, \Sigma_{0})\big) \\
      &  = \frac12\, \Big(\tr\Big(\mathrm{I}_p - (1+p)^{-1}\, \ind_p \ind_p^\top +  v_l\,v_l^\top - (1+p)^{-1}\, (\ind_p^\top \, v_l)\,\ind_p v_l^\top\Big) - p - \log\Big(\frac{1+\norm{v}_2^2}{1 + p} \Big) \Big)\nonumber\\
     & = \frac12\, \Big(\norm {v_l}_2^2 - \frac p{p+1} - \frac{ (\ind_p^\top \, v_l)^2}{p+1} - \log\Big(\frac{1+\norm{v_l}_2^2}{1 + p} \Big) \Big)\\
     & =  \frac12\, \Big(\frac{p\norm {v_l}_2^2 - (\ind_p^\top \, v_l)^2}{p+1} + \frac{\norm {v_l}_2^2 - p}{p+1} - \log\Big(1 + \frac{\norm {v_l}_2^2 - p}{p+1} \Big) \Big).
    \end{align*}
    Now 
    	\begin{align}
		\frac{p\norm {v_l}_2^2 - (\ind_p^\top \, v_l)^2}{p+1} & = \frac p{p+1} \sum_{k = 1}^p g_l^2(x_k) - \frac1{p+1}\bigg( \sum_{k = 1}^p g_l(x_k) \bigg)^2 \leq \sum_{k = 1}^p g_l^2(x_k).\label{eq:rate1}
	\end{align}
    Further, 
    \begin{align}
		\frac{\norm v_2^2 - p}{p+1} = \frac1{p+1} \sum_{k = 1}^p g_l^2(x_k) + \frac 2{p+1} \sum_{k = 1}^p g_l(x_k), \label{eq:rate2}
	\end{align}
    which is small for $n,p$ large enough, so that
    	\begin{align}
		\frac{\norm {v_l}_2^2 - p}{p+1} - \log\Big(1 + \frac{\norm {v_l}_2^2 - p}{p+1} \Big) \Big)& \leq c\, \Big( \frac{\norm {v_l}_2^2 - p}{p+1}\Big)^2 \leq \tilde c\, p^{-1}\,\sum_{k = 1}^p g_l^2(x_k) \label{eq:secondkl}
                    %
	\end{align}
    by bounding the square of the first term on the right of \eqref{eq:rate2} by $\big(\sum_{k = 1}^p g_l^2(x_k)\big)^2 \leq 2\, p\, \sum_{k = 1}^p g_l^2(x_k)$, and by using the Schwartz-inquality on the square of the second term. 

    Let $\prob_l^{\otimes n}$ be the joint normal distribution of the observations $(Z_{1;l}(x_1) + \epsilon_{1,1}, \ldots, Z_{1,l}(x_p) + \epsilon_{1,p}), \ldots, $  $(Z_{n;l}(x_1) + \epsilon_{n,1}, \ldots, Z_{n;l}(x_p) + \epsilon_{n,p})$. We obtain
    	\begin{align*}
		\frac{1}{N_{n, p}} \sum_{ l=1}^{N_{n, p}}  \KL(\prob_l^{\otimes n} || \prob_0^{\otimes n}) 
		& = \frac{1}{N_{n, p}} \sum_{ l=1}^{N_{n, p}} n\,\KL\big(\mathcal N_p(0, \Sigma_{l})||\mathcal N_p(0, \Sigma_{0})\big) \\
		& \leq \frac{n}{N_{n, p}} \, c_1\, \sum_{k = 1}^p \, \sum_{ l=1}^{N_{n, p}} \, g_l^2(x_k) \tag{\eqref{eq:rate1} and \eqref{eq:secondkl}}\\
		&\leq \, n\, c_2\, h_{n, p}^{2\gamma+1}\sum_{k = 1}^p \, \sum_{ l=1}^{N_{n, p}} \,\,\one_{\{2\abs{{x_k}-z_l} < h_{n, p}\}} \tag{constr.~of $g_l$}\\
		&\leq \, n\, c_2\, h_{n, p}^{2\gamma+1} \, p \leq c_2 \log (np) \lesssim  \log(N_{n, p}) \\
	\end{align*}
     as required in \citet[Theorem 2.7]{tsybakov2008introduction}.
}

\end{proof} 

\appendix

\section{Proofs for Lemma \ref{lem:cov_error_decomp} and \ref{lem:rates:convergence}.}\label{app:technical}

\begin{proof}[Proof of Lemma \ref{lem:cov_error_decomp}]
Plugging the representation of the observations $Y_{i,j}$ and $\bar Y_{n,j}$ into the estimator $\covest{x}{y}{h}$ in \eqref{eqn:estimatorCovariance} we get
	\begin{align}
		\covest xyh & = \frac1{n-1}\bigg[\sum_{i=1}^n \sum_{j<k}^{p} \wjk xyh \big(Y_{i,j}Y_{i,k} - \bar Y_{n,j}\bar Y_{n,k}\big)\bigg].\nonumber\\
		& = \frac1{n-1}\sum_{i = 1}^n \sum_{j<k}^{p}\wjk xyh\bigg[ \Big(Z_{i,j}Z_{i,k} + \epsilon_{i,j}\epsilon_{i,k} + Z_{i,j}\epsilon_{i,k} + \epsilon_{i,j}Z_{i,k}\nonumber\\
		&\QQuad \Quad+ \mu(x_j)\big(Z_{i,k} + \epsilon_{i,k}\big) + \big(Z_{i,j} + \epsilon_{i,j}\big)\mu(x_k)\Big)\nonumber\\
		&\quad - \Big( \bar Z_{n,j}\bar Z_{n,k} + \bar \epsilon_{n,j}\bar \epsilon_{n,k} + \bar Z_{n,j}\bar \epsilon_{n,k} + \bar \epsilon_{n,j}\bar Z_{n,k}\nonumber\\
		&\QQuad \Quad+ \mu(x_j)\big(\bar Z_{n,k} + \bar \epsilon_{n,k}\big) + \big(\bar Z_{n,j} + \bar \epsilon_{n,j}\big)\mu(x_k)\Big)\bigg]\nonumber\\
		& = \frac1{n-1}\sum_{i = 1}^n \sum_{j<k}^{p} \wjk xyh\Big[\big(\epsilon_{i,j}\epsilon_{i,k} - \bar \epsilon_{n,j}\bar \epsilon_{n,k} \big)\nonumber\\
		& \qquad + \big(Z_{i,j}Z_{i,k} - \bar Z_{n,j}\bar Z_{n,k} \big)\nonumber\\
		& \qquad + \big(Z_{i,j}\epsilon_{i,k} -\bar Z_{n,j}\bar \epsilon_{n,k}\big) + \big(\epsilon_{i,j}Z_{i,k} - \bar \epsilon_{n,j}\bar Z_{n,k}\big)\nonumber\\
		& \qquad + \mu(x_j)\big(Z_{i,k} - \bar Z_{n,k} + \epsilon_{i,k} - \bar \epsilon_{n,k} \big) + \big(Z_{i,j} - \bar Z_{n,j} + \epsilon_{i,j} - \bar\epsilon_{n,j}\big)\mu(x_k)\Big].\nonumber
	\end{align}
	Let $U_{i,j}$ be a placeholder for $\epsilon_{i,j}$ or $Z_{i,j}$. In order to summarize the last display we use
	\[ \frac1{n-1}\sum_{i=1}^nU_{i,k}- \frac n{n-1}\bar U_{n,k} = 0 \]
	and therefore the last row vanishes. For the first three rows we use
	\[ \frac1{n-1} \sum_{i=1}^nU_{i,j}U_{i,k} - \frac1{n(n-1)} \sum_{l,r=1}^nU_{l,j}U_{r,k} = \frac1n\sum_{i=1}^n U_{i,j}U_{i,k} - \frac1{n(n-1)} \sum_{l\neq r}^n U_{l,j} U_{r,k}, \]
	for $Z_{i,j}Z_{i,k}, \epsilon_{i,j}\epsilon_{i,k}$ and $Z_{i,j}\epsilon_{i,k} $ respectively. By adding $\pm \Gamma(x_j,x_k)$ this yields the decomposition
	\begin{align*}
		\covest{x}{y}{h}  & = \sum_{j<k}^{p} \wjk xyh\Bigg[\bigg(\frac1n\sum_{i=1}^n \epsilon_{i,j}\epsilon_{i,k} - \frac1{n(n-1)} \sum_{l\neq r}^n \epsilon_{l,j} \epsilon_{r,k}\bigg)\\
		& + \frac1n\sum_{i=1}^n\big(Z_{i,j}Z_{i,k} - \Gamma({x_j},x_k) \big) - \frac1{n(n-1)}\sum_{i \neq l}^n Z_{i,j}Z_{l,k}+ \Gamma(x_j,x_k)\\
		& + \bigg(\frac1n\sum_{i=1}^n (Z_{i,j}\epsilon_{i,k} + Z_{i,k}\epsilon_{i,j})  - \frac1{n(n-1)} \sum_{l\neq r}^n (Z_{l,j}\epsilon_{r,k} + Z_{l,k}\epsilon_{r,j} )\bigg)\Bigg].
	\end{align*}
	Subtracting $\Gamma(x,y)$ yields the claim together with \ref{ass:weights:polynom}. 
\end{proof}

\begin{proof}[Proof of Lemma \ref{lem:rates:convergence}, i).]
For the bias we  use a Taylor expansion and the fact that the weights of the estimator reproduce polynomials of the certain degree $\zeta = \left\lfloor \gamma \right \rfloor$. We get for certain $\theta_{1,k}, \theta_{2,k} \in [0,1]$ s.t. $\tau_j^{(1)} \defeq x + \theta_{1,k}(x_j - x) \in [0,1]$ and $\tau_k^{(2)} \defeq y + \theta_{2,k}(x_k - y) \in [0,1]$ 
\begin{align*}\allowdisplaybreaks
	\sum_{j<k}^{ p} &\wjk xyh \Gamma(x_j, x_k) - \Gamma(x,y)	 = \sum_{j<k}^{ p} \wjk xyh\big( \Gamma(x_j, x_k) - \Gamma(x,y)\big) \tag{by \ref{ass:weights:polynom}}\\
	&=\sum_{j<k}^{ p} \wjk xyh \bigg(\sum_{\substack{| r|= 1,\\ r =(r_1,r_2) }}^{\zeta-1} \frac{\partial^{\abs{r}} \Gamma(x,y) }{(\partial x)^{r_1} (\partial y)^{r_2} }\frac{( x_{ j}- x)^{ r_1} (x_k - y)^{r_2}}{ r_1 ! r_2!}\\
	& \qquad + \sum_{\substack{| r|= \zeta\\ r =(r_1,r_2) }}\frac{\partial^{\abs{r}} \Gamma(\tau_j^{(1)}, \tau_k^{(2)})}{(\partial x)^{r_1} (\partial y)^{r_2} }  \frac{( x_{ j}- x)^{ r_1} (x_k - y)^{r_2}}{ r_1 ! r_2!} \bigg)\\
	&=\sum_{j<k}^{ p} \wjk xyh \Bigg( \sum_{|r|= \zeta} \bigg(\frac{\partial^{\abs{r}} \Gamma(x,y)}{(\partial x)^{r_1} (\partial y)^{r_2}} - \frac{\partial^{\abs{r}} \Gamma(\tau_j^{(1)}, \tau_k^{(2)})}{(\partial x)^{r_1} (\partial y)^{r_2} } \bigg) 
	\frac{( x_{ j}- x)^{ r_1} (x_k - y)^{r_2}}{ r_1 ! r_2!}\Bigg) \tag{by \ref{ass:weights:polynom}}\\
	& \leq \sum_{j<k}^{ p} \wjk xyh  \sum_{|r|= \zeta}\max(x_j-x, x_k-y)^{\gamma - \zeta}\frac{( x_{ j}- x)^{ r_1} (x_k - y)^{r_2}}{ r_1 ! r_2!}\\
	&\leq  \sum_{j<k}^{ p} \wjk xyh \max(x_j-x, x_k-y)^{\gamma} \sum_{| r|= \zeta}\frac{1}{ r !}\\
	&\leq h^\gamma \sum_{| r|= \zeta}\frac{C_1}{ r !} = \mc O\big(h^\gamma)\,. \tag{by \ref{ass:weights:sum}}
\end{align*}
\end{proof}

\begin{proof}[Proof of Lemma \ref{lem:rates:convergence} ii), second bound.]
    We rewrite the term 
\begin{align*}
	\sum_{j<k}^{p} \wjk xyh\frac1{n(n-1)}\sum_{l\neq r}^n \epsilon_{l,j} \epsilon_{r,k}
	& = \frac1{n(n-1)} \sum_{j<k}^{p}\wjk xyh \bigg[\sum_{i, l = 1}^n \epsilon_{i,j} \epsilon_{l,k}-  \sum_{i=1}^n\epsilon_{i,j}\epsilon_{i,k}\bigg]\\
	& = \frac1{n-1}\sum_{j<k}^{p}\wjk xyh \tilde X_{j,n} \tilde X_{k,n} - \frac1{n-1}E_{n,p,h}(x,y),
\end{align*}
where $E_{n,p,h}(x,y)$ is the process defined in \eqref{eqn:def:Enph} and $\tilde X_{j,n} = \sqrt n \bar X_{j,n} = n^{-1/2} \sum_{i=1}^n \epsilon_{i,j}$. By the first statement of ii) we obtain
\begin{equation}
	\expec\bigg[\sup_{(x,y)^\top \in D} \frac1{n-1} \, \absb{E_{n,p,h}(x,y)}\bigg] = \cO \Bigg(\frac{\sqrt{\log(n\,p)}}{n^{3/2}\,{p\,h}}\Bigg).
\end{equation}
Simple calculations show that $\tilde X_{j,n}\sim \subG{(\sigma^2)}$ and its Orlicz-Norm is given by $\norm{\tilde X_{j,n}}_{\psi_2} = K$, where $K>0$ is constant not depending on $n$. Defining $A\defeq \wjk xyh \ind_{j<k}$ as before, using the estimates \eqref{eqn:frobenius:A} and \eqref{eqn:operator:A}  and the Hanson-Wright inequality (Lemma \ref{lem:hanson:wright}) we obtain, analogously to \eqref{eqn:HW:application},
\begin{align*}
	\prob\bigg(\frac{(n-1)\,p\,h}{\log(n\,p) }\Big|\sum_{j<k}^{p}\wjk xyh \frac{\tilde X_{j,n}\tilde X_{k,n}}{n-1} \Big| \geq t\bigg)	& \leq 2 \exp\bigg(- \tilde c \min\Big(\frac{t^2 \log^2(n\,p)}{K^4\,\Cmax\,\Csum},\; \frac{t \log(n\,p)}{K^2 \Cmax\,\Ccard}\Big)\bigg)\\
    & \leq 2\,\exp\big( -t\log(n\,p) /C \big)\, 
\end{align*}
for $t\geq 1$ and $C \defeq K^4\Cmax\max(\Ccard,\Csum)$. The same calculations as in \eqref{eqn:disc:HW:estimate} then lead to 
\begin{equation}
	\expec\bigg[\sup_{x,y\in[0,1]}\frac1{n(n-1)} \sum_{j<k}^{p}\wjk xyh \sum_{i, l = 1}^n \epsilon_{i,j} \epsilon_{l,k} \bigg] = \cO\bigg(\frac{\log(n\,p)}{n\,{p\,h}}\bigg).
\end{equation}
\end{proof}

\begin{proof}[Proof of Lemma \ref{lem:rates:convergence} iii), second bound.]
Set
\begin{equation*}
	R_{n,p,h}(x,y) \defeq \frac1{n(n-1)} \sum_{j<k}^{p} \wjk xyh \sum_{i\neq l}^n Z_i(x_j)Z_l(x_k)
\end{equation*}
and consider the envelope $ \bs \Phi_{n} \in \R^{n(n-1)}$ with entries $\phi_{n,(i,l)}$, $i,l = 1, \ldots, n,$ with $i \neq l$  given by 
\begin{equation*}
	\frac1{\sqrt {n(n-1)}}\absA{Z_i(x)Z_j(y)} \leq \frac1{\sqrt {n(n-1)}}\big( M_i + \abs{Z_i(0)}\big)  \big( M_l + \abs{Z_l(0)}\big) . 
\end{equation*}
This leads to
\begin{align*}
	\expec\big[\norm{ \sqrt{n(n-1)}\, R_{n,p,h}}_\infty\big]  & \leq \Csum \, \expec\bigg[\sup_{x,y \in[0,1]}\Big| \sum_{i \neq l}^n \frac{Z_i(x)Z_l(y)}{\sqrt{n(n-1)}}\Big|\bigg] \\
	& \leq \Csum \, \expec\bigg[\sup_{x,y \in[0,1]}\Big| \sum_{i \neq l}^n \frac{Z_i(x)Z_l(y)}{\sqrt{n(n-1)}}\Big|^2\bigg]^{1/2} \\
	& \leq 2\,\Csum \, K_2 \, \Lambda(1)^2 \expec\big[\big( M_1 + \abs{Z_1(0)}\big)  \big( M_2 + \abs{Z_2(0)}\big) ^2\big]^{1/2} \\
	& \leq 4\,\Csum \, K_2 \, \Lambda(1)^2 \Big(2\expec\big[M_1^2\big] + 2\expec\big[Z_1(0)^2\big] \Big) < \infty.
\end{align*}
\end{proof}

\begin{proof}[Proof of Lemma \ref{lem:rates:convergence} iv), second bound.] For $ \tilde S(x,y) \defeq \sum_{j<k}^p \wjk xyh \frac1n \sum_{i\neq l}^n \frac{Z_i(x_j) \epsilon_{l,k}}{n-1} $ we shall proceed analogously as for the first term. Let 
$$\tilde S_{\mid \bs z}(x,y) \defeq \sum_{j<k}^p \wjk xyh \frac1{n(n-1)} \sum_{i\neq l}^n z_{i,j}\epsilon_{l,k} .$$ 
Again we can show that $\tilde S_{\mid \bs z}$ is a sub-Gaussian process with respect to the semi-norm $\expec[(\tilde S_{\mid \bs z}(x,y) - \tilde S_{\mid \bs z}(x^\prime, y^\prime))^2]^{1/2}$. This semi-norm is given by
\begin{align*}
	d_{\tilde S}\big((x,y)^\top, (x^\prime, y^\prime)^\top\big)^2 & = \expec[(\tilde S_{\mid \bs z}(x,y) - \tilde S_{\mid \bs z}(x^\prime, y^\prime))^2]\\
	& = \expec \bigg[ \Big( \sum_{j < k}^p \big( \wjk xyh - \wjk{x^\prime}{y^\prime}h\big) \sum_{i < l}^n \frac{z_{i,j}\,\epsilon_{l,k} }{n(n-1)} \Big)^2\bigg] \\
	& = \frac{\sigma^2}{n-1} \sum_{k=2}^{p} \bigg( \sum_{j = 1}^{k-1} \big(\wjk xyh - \wjk{x^\prime}{y^\prime}h\big) \sum_{i = 1}^n \frac{z_{i,j}}n\bigg)^2 \,,
\end{align*}
and further we have that
\begin{align*}
    \expec\big[\exp(\tilde S_{\mid \bs z}(x,y)  &- \tilde S_{\mid\bs z}(x^\prime, y^\prime))\big] \\ & = \prod_{k = 2}^p \expec \bigg[ \exp\Big( \sum_{j = 1}^{k-1} \big( \wjk xyh - \wjk{x^\prime}{y^\prime}h\big) \sum_{i = 1}^n\frac{z_{i,j}}{n}\sum_{l = 1, l \neq i}^n \frac{\epsilon_{l,k}}{n-1}\Big)\bigg] \\
    & \leq \exp\bigg( \frac{\sigma^2}{2(n-1)}\sum_{k = 2}^p\Big(\sum_{j = 1}^{k-1} \big(\wjk xyh - \wjk{x^\prime}{y^\prime}h\big) \sum_{i = 1}^n \frac{z_{i,j}}{n}\Big)^2 \bigg) \,.
\end{align*}
Therefore $\tilde S_{\mid \bs z}$ is a sub-Gaussian process with respect to the semi-norm $d_{\tilde S}$. With the same arguments as for the first part we can bound the diameter of the semi-norm by
\begin{align*}
	\diam_{\tilde S}([0,1]^2) & \leq \frac{\Clip^2\,\Ccard^3}{(n-1)\,p\,h} \bigg( \sum_{i=1}^n \frac{m_i}n\bigg)^2 \sigma^2 \defeql \tilde \Delta = \mc O \bigg( \frac{1}{n\,p\,h}\bigg)\,.
\end{align*}
From here the result follows analogously to \eqref{eq:dudley:bound}, \eqref{eq:dudley:second} and \eqref{eq:dudley:first}.
\end{proof}

\begin{lemma}\label{lem:rate_skp_subExp}
    For the independent, mean zero, sub-Gaussian errors $\epsilon_{i,j}, i = 1, \ldots, n,\, j = 1, \ldots, p$ it holds 
    \begin{align}
        \expec\big[\max_{1 \leq j < k \leq p} \frac1n\sum_{i = 1}^n \abs{\epsilon_{i,j} \epsilon_{i,k}} \big] & = \mc O \bigg( \sqrt{\frac{\log(p)}{n}} + \frac{\log(p)}{n}\bigg)\,.
    \end{align}
\end{lemma}

\begin{proof}
    By the independence we have $\epsilon_{i,j}\epsilon_{i,k} \sim \mathrm{subExp}(\sigma^4, \sigma^2)$. Together with the Jensen inequality we get
    \begin{align*}
        \expec\bigg[ \max_{j<k} \frac1n\sum_{i = 1}^n \abs{\epsilon_{i,j}\epsilon_{i,k}} \bigg] & \leq \frac1\lambda \expec \bigg[ \log\bigg( \sum_{j<k} \exp\Big(\frac\lambda n\, \sum_{i = 1}^n \abs{\epsilon_{i,j}\epsilon_{i,k}}\Big)\bigg)\bigg]\\
        & \leq \frac1\lambda \log\bigg( \expec\bigg[ \sum_{j<k}\prod_{i = 1}^n \exp\big(\frac\lambda n \abs{\epsilon_{i,j}\epsilon_{i,k}}\big)\bigg]\bigg) \\
        & = \frac1\lambda \log\bigg( \sum_{j<k}\prod_{i = 1}^n \expec\big[ \exp\big(\frac\lambda n \abs{\epsilon_{i,j}\epsilon_{i,k}}\big)\big]\bigg) \\
        & \leq \frac1\lambda \log\bigg( \sum_{j<k}\exp\Big(\frac{\sigma^4\,\lambda^2}{2\,n}\Big)\bigg) \\ 
        & = \frac1\lambda \,\log\bigg(\frac{p\,(p-1)}{2}\bigg) + \frac{\sigma^4\,\lambda}{2\,n}\,,
    \end{align*}
    for all $\abs{\lambda}/n \leq \sigma^{-2}$. If $\log(p)\lesssim n$ choosing $\lambda \simeq \sqrt{n\,\log(p)}$ gives an upper bound of order $\sqrt{\log(p)/n}$ and if $\log(p) \gtrsim$ choosing $\lambda \simeq n$ gives the bound $\log(p)/n$.
\end{proof}

\section{Proof for the asymptotic normality.}\label{ssec:proof_asymp_norm}

\begin{proof}[Proof of Theorem \ref{thm:asymp_norm}]
We make use of the functional central limit theorem \citet[Theorem 10.6]{pollard1990empirical}.
%
 For the proof let $$X_{n,i}(x,y) \defeq \frac1{\sqrt n} \sum_{j<k}^p \wjk xyh \Zi jk,\qquad \Zi jk \defeq Z_i(x_j)Z_i(x_k) - \Gamma(x_j,x_k),$$
  and 
 \[ S_n(x,y) \defeq \sum_{i=1}^n X_{n,i}(x,y), \Quad \rho_n(x,y,s,t) \defeq \bigg( \sum_{i=1}^n \expec\Big[ \abs {X_{n,i}(x,y) - X_{n,i}(s,t)}^2 \Big]\bigg)^{1/2}. \]
To apply \citet[Theorem 10.6]{pollard1990empirical} we need to check the following. 

 i). $X_{n,i}$ is manageable \citep[Definition 7.9]{pollard1990empirical} with respect to the envelope $\Phi_{n}\defeq (\phi_{n,1},\ldots, \phi_{n,n})$ with 
 \begin{align}
   \phi_{n,i} \defeq \frac{\Csum}{\sqrt n} \big( 2\,Z_i^2(0) + 2\,  M_i^2 + L\big).  \label{eq:envelope}
 \end{align}
 ii). $R(x,y,s,t) \defeq \displaystyle\lim_{n \to \infty} \expec\big[S_n(x,y)S_n(s,t)\big]$, $x,y,s,t \in [0,1]$.\\
iii). $\displaystyle \limsup_{n \to \infty} \sum_{i=1}^n \expec[\phi_{n,i}^2] < \infty.$\\
iv). $\displaystyle \lim_{n \to \infty} \sum_{i=1}^n \expec\big[ \phi_{n,i}^2 \,\ind_{\phi_{n,i} > \epsilon}\big] = 0$,  $\epsilon > 0$.\\
v). The limit 
$$\rho(x,y,s,t) \defeq \lim_{n \to \infty} \rho_n(x,y,s,t), \qquad x,y,s,t \in [0,1]\,,$$
is well-defined and for all deterministic sequences $(x_n, y_n)_{n \in \N}, (s_n, t_n)_{n\in \N} \in [0,1]^2 $ with\\ $\rho(x_n, y_n,s_n, t_n) \to 0$ it also holds $\rho_n(x_n, y_n, s_n, t_n) \to 0$.\\

\smallskip

Ad i): As in the proof of Lemma \ref{lem:rates:convergence}, iii), the random vector $ \Phi_{n}$ is an envelope of $X_n \defeq (X_{n,1}, \ldots, X_{n,n})$ since
\begin{align*}
    \absb{X_{n,i}(x,y)} & \leq \frac1{\sqrt n} \sum_{j<k}^p \abs{ \wjk xyh }\big( 2\,Z_i^2(0) + 2\,  M_i^2 + L\big) \leq \frac{\Csum}{\sqrt{n}} \big( 2\,Z_i^2(0) + 2\,  M_i^2 + L\big)\,,
\end{align*}
by using \ref{ass:weights:sum}, where we may assume $\Csum \geq 1$. We make use of \citet[Lemma 3]{berger2023dense}. We need to show that there exist constants $K_1, K_2$ and $b \in \R^+$ such that for all $(x,y), (x^\prime, y^\prime) \in [0,1]^2$ it holds
\begin{align}\label{eq:condmanagable2}
    \normm{ \vecTwo xy - \vecTwo{x^\prime}{y^\prime}}_\infty  \leq K_1 \epsilon^b & \Rightarrow \absb{X_{n,i}(x,y) - X_{n,i}(x^\prime, y^\prime)} \leq K_2\,\epsilon\, \phi_{n,i}, \quad \forall \; i = 1,\ldots, n\,.
\end{align}

Since $X_{n,i}(x,y) = X_{n,i}(y,x)$, we can assume that $(x,y), (x^\prime, y^\prime) \in T$. 


We distinguish the cases $\epsilon \geq h^\gamma$ and $\epsilon < h^\gamma$. For the first case $\epsilon \geq h^\gamma$ consider
\begin{align*}
    \frac1{\sqrt n}\abss{ \sum_{j<k}^p \big( \wjk xyh &- \wjk{x^\prime}{y^\prime}h \big) \Zi jk }  \leq \frac1{\sqrt n}\abss{ \sum_{j<k}^p \wjk xyh \Zi jk - Z_i^{\otimes2}(x,y)} \\
    & + \frac1{\sqrt n}\absb{Z_i^{\otimes2}(x,y) - Z_i^{\otimes2}(x^\prime,y^\prime)} + \frac1{\sqrt n}\abss{Z_i^{\otimes2}(x^\prime, y^\prime) - \sum_{j<k}^p \wjk{x^\prime}{y^\prime}h \Zi jk},
\end{align*}
we shall check \eqref{eq:condmanagable2} for $K_1 = 1, K_2 = 3$ and $ b =1/\gamma$. Take $(x,y)^\top, (x^\prime, y^\prime)^\top$ such that $\max\{\abs{x-x^\prime}, \abs{y - y^\prime}\} \leq \epsilon ^{1/\gamma}$. By Assumption \ref{ass:distribution} the second term can be estimated by
\begin{align*}
    \frac1{\sqrt n}\absb{Z_i^{\otimes2}(x,y) - Z_i^{\otimes2}(x^\prime,y^\prime)} & \leq \frac{\big( 2\,Z_i^2(0) + 2\,  M_i^2 + L\big)}{\sqrt n}\normm{\vecTwo xy - \vecTwo{x^\prime}{y^\prime} }_\infty^\gamma \leq \phi_{n,i}\epsilon. 
\end{align*}
For the first and similarly the third term one gets that  
\begin{align*}
    \frac1{\sqrt n}\abss{ \sum_{j<k}^p \wjk xyh \Zi jk - Z_i^{\otimes2}(x,y)} & \leq \frac1{\sqrt n}\sum_{j<k}^p \absb{ \wjk xyh } \absb{ \Zi jk - Z_i^{\otimes2}(x,y)}\\
    & \leq \frac{\Csum\, \big(\big( 2\,Z_i^2(0) + 2\,  M_i^2 + L\big)}{\sqrt n}\, h^\gamma  \leq \phi_{n,i}\, \epsilon,
\end{align*}
since the weights vanish for $\max\{\abs{x_j -x}, \abs{x_k -y}\} \geq h$. Adding the three terms concludes \eqref{eq:condmanagable2} for the first case $\epsilon \geq h^\gamma$ . 

For the second case $\epsilon < h^\gamma$,  again take $\max\{\abs{x-x^\prime}, \abs{y - y^\prime}\} \leq \epsilon ^{1/\gamma}$. By \ref{ass:weights:lipschitz} we get
\begin{align*}
    \frac1{\sqrt n}\abss{ \sum_{j<k}^p \big( \wjk xyh - \wjk{x^\prime}{y^\prime}h \big) \Zi jk }  & \leq \sum_{j<k}^p \absb{\wjk xyh - \wjk{x^\prime}{y^\prime}h}\,\phi_{n,i}  \\
    & \leq \frac{2\Clip\, \Ccard}{h}\normm{\vecTwo xy -\vecTwo{x^\prime}{y^\prime} }_\infty\,\phi_{n,i}  \\
    & \leq \frac{2\Clip\,\Ccard}h\, \epsilon^{1-1/\gamma} \,\phi_{n,i}\,,
\end{align*}
which yields \eqref{eq:condmanagable2} in the second case. \citet[Lemma 3]{berger2023dense} then yields manageability. 
%

\smallskip

Ad ii). 
\begin{align*}
    \expec[ S_n(x,y) S_n(x^\prime, y^\prime)] & = \frac1n \sum_{i,l = 1}^n\sum_{j<k}^p \sum_{r<s}^p \wjk xyh w_{r,s}(x^\prime, y^\prime;h) \expec\big[\Zi jk Z_{l;r,s}^{\otimes 2}\big]\\
    & = \sum_{j<k}^p \sum_{r<s}^p \wjk xyh w_{r,s}(x^\prime, y^\prime;h) R(x_j,x_k, x_r, x_s).
\end{align*}

By Assumption \ref{ass:asymp_norm}, 
$R\in \mc H_{D^2}(\zeta, \tilde L)$ and therefore we get by the same calculations as for the bias in the supplementary appendix \ref{app:technical} 
\begin{align*}
    \sup_{x,y,x^\prime,y^\prime}\abss{\sum_{j<k}^p \sum_{r<s}^p \wjk xyh w_{r,s}(x^\prime, y^\prime;h) R(x_j,x_k, x_r, x_s) - R(x,y,x^\prime, y^\prime)} = \mc O(h^{\zeta})\,.
\end{align*}


\smallskip

Ad iii). Follows since $Z$ and $M$ have finite fourth moments. 

\smallskip

Ad iv). Follows by by the finite fourth moments of $Z$ and $M$ and the dominated convergence theorem. 

\smallskip

Ad v). 
 \begin{align*}
     \rho_n^2(x,y,x^\prime, y^\prime) & = \sum_{j<k}^p \sum_{r<s}^p \big(\wjk xyh - \wjk{x^\prime}{y^\prime}h\big)\big(w_{r,s}(x,y,h)-w_{r,s}(x^\prime, y^\prime;h)\big)\expec \Big[ Z_{1;j,k}^{\otimes 2}Z_{1;r,s}^{\otimes2} \Big]\\
     & =  \sum_{j<k}^p \sum_{r<s}^p \Big(  \wjk xyh w_{r,s}(x,y;h)  - \wjk xyh w_{r,s}(x^\prime,y^\prime;h) \\
     & \QQuad - \wjk{x^\prime}{y^\prime}h w_{r,s}(x,y;h) + \wjk{x^\prime}{y^\prime}h w_{r,s}(x^\prime, y^\prime;h)\Big) R(x_j,x_k,x_r,x_s) \Big)\\
     & \stackrel{n \to \infty}{\longrightarrow} R(x,y,x,y) - 2R(x,y,x^\prime, y^\prime) + R(x^\prime, y^\prime, x^\prime, y^\prime), 
 \end{align*}
 by the same argument as for ii). Therefore the limit is well defined in $\R$ for all $x,y,x^\prime, y^\prime \in [0,1]$. Given the deterministic sequences $(x_n,y_n)$ and $(x^\prime_n, y^\prime_n)$ such that $\rho(x_n,y_nx^\prime_n,y^\prime_n) \to 0, n \to \infty,$ we get
 \begin{align*}
     0 \leq \rho_n(x_n,y_n,x_n^\prime,y_n^\prime) & \leq \absb{\rho_n(x_n,y_n,x_n^\prime,y_n^\prime) - \rho(x_n,y_n,x_n^\prime,y_n^\prime)} + \absb{ \rho(x_n,y_n,x_n^\prime,y_n^\prime)} \\
     &\leq \sup_{x,y,x^\prime,y^\prime} \absb{\rho_n(x,y,x^\prime,y^\prime) - \rho(x,y,x^\prime,y^\prime)} + \absb{ \rho(x_n,y_n,x_n^\prime,y_n^\prime)} \stackrel{n \to \infty}{\longrightarrow} 0,
 \end{align*}
 since the convergence in ii) holds uniformly and hence similarly also the convergence in the first calculation.

\end{proof}

\section{Proofs related to restricted local polynomial estimation.}
\label{sec:proof:lemma:locpol:weights}

%
From Example \ref{ex:localpolynomial} recall the notation
\begin{align*}
	U_m(u_1,u_2)\defeq \big(1, P_{1}(u_1, u_2), \ldots, P_{m}(u_1,u_2)\big)^\top,\quad u_1,u_2\in[0,1],
\end{align*}
where for $l=1,\ldots, m$ we set 
\begin{equation*}
    P_l(u_1, u_2)\defeq\bigg(\frac{u_1^l}{l!}, \frac{u_1^{l-1}u_2}{(l-1)!}, \frac{u_1^{l-2}u_2^2}{(l-2)!2!},\ldots, \frac{u_2^l}{l!}\bigg), \quad u_1,u_2 \in [0,1].
\end{equation*}
Moreover for a non-negative kernel function $K \colon \R^2 \to [0, \infty)$ and a bandwidth $h >0$ we set $K_{h}(x_1,x_2) \defeq K(x_1/h, x_2/h)$ and $U_{m,h}(u_1, u_2) \defeq U_m(u_1/h, u_2/h)$. 

Given an order $m \in \N$ and observations $((x_j, x_k), z_{j,k}), 1 \leq j<k \leq p$, we consider 
\begin{align} \label{eq:locpol}
	\widehat{\mathrm{LP}}(x, y) \defeq \argmin_{\bs \vartheta \in \R^{N_m}} \sum\limits_{j<k}^{p} \bigg(z_{j,k} - \bs \vartheta^\top\, \Uh{x_j-x}{x_k-y} \bigg)^2 \Kh{x_j - x}{x_k - y} \,, \quad 0 \leq x \leq y \leq 1\,,
\end{align}
and define the local polynomial estimator of the target function with order $m$ at $(x,y)$ as
\begin{align}\label{eq:locpolest}
	\widehat{\mathrm{LP}}_1(x,y) \defeq \left \{\begin{array}{cl} \big(\widehat{\mathrm{LP}}( x,y )\big)^\top U_m(0,0), & x \leq y,\\
    \big(\widehat{\mathrm{LP}}(y,x)\big)^\top U_m(0,0), & \text{otherwise},
    \end{array}\right. \, \quad \text{for}\;(x,y)^\top \in [0,1]^2 \, ,
\end{align}
since $U_m(0,0)$ is the first unit vector in $N_m$ dimensions. 

We require the following more restrictive design assumption. 
\begin{assumption}[Design]\label{ass:designdensity}
	Let the points $x_{1},\ldots,x_{p}$ be given by 
	the equations
	\begin{equation*}
		\int_0^{x_{j}}f(t) \, \dx t=\frac{j-0.5}{p} \,, \qquad j=1,\ldots,p \,,	
	\end{equation*}
	where $f\colon[0,1]\to \R$ is a Lipschitz continuous density that is bounded by $0<f_{\min} \leq f(t) \leq f_{\max} < \infty$ for all $t \in [0,1]$. 
\end{assumption}

In \citet{berger2023dense} it is shown that Assumption \ref{ass:designdensity} implies Assumption \ref{ass:design:localization}. {\color{black}  Note that the assumption only asserts the existence and not the knowledge of such a function $f$. For the equidistant design with $x_j = (j - 0.5)/p, j = 1, \ldots, p,$ choose $f \equiv 1$.}

\begin{lemma}\label{lemma:locpol:weights}
	Let the kernel function $K\colon \R^2 \to [0, \infty)$ have compact support in $[-1,1]^2$ 
	 and satisfy 
	\begin{align}
		K_{\min}\one_{\{- \Delta \leq u_1,u_2 \leq  \Delta \}} \leq K(u_1, u_2) \leq K_{\max} \qquad \text{for all } u_1,u_2 \in \R\,,	\label{eqn:kernel}
	\end{align}
	for a constants $\Delta\in \R_+$, $ K_{\min}, K_{\max} > 0$. Then under Assumption \ref{ass:designdensity} there exist a sufficiently large $p_0 \in \N$ and a sufficiently small  $h_0 >0$ such that for all $p\geq  p_0$ and $h\in(c/ p,h_0]$, where $c >0$ is a sufficiently large constant, the local polynomial estimator in \eqref{eq:locpolest} with any order $m \in \N$ is unique. Moreover, the corresponding weights in \eqref{eq:weights:locpol} satisfy Assumption \ref{ass:weights} with $\gamma = m$.
\end{lemma}

Let us proceed to preparations of the proof of Lemma \ref{lemma:locpol:weights}.  

For $x\leq y$, $x,y \in [0,1]$ 
set
\begin{equation}
	B_{p,h}(x,y) \defeq \frac1{(p\,h)^2}\sum_{j<k}^{p} 
	\Uh{x_j-x}{x_k-y}
	U_h^\top\vecTwo{x_j-x}{x_k-y} \Kh{x_j-x}{x_k-y} \qquad \in\R^{N_{m}\times N_{m}} \label{eq:Bp}
\end{equation}
and
\begin{equation*}
	\bs a_{p,h}(x,y)\defeq \frac1{(p\,h)^2} \sum_{j<k}^{p} \Uh{x_j-x}{x_k-y}\Kh{x_j-x}{x_k-y} \, z_{j,k} \qquad \in \R^{N_{m}} ,
\end{equation*}
then \eqref{eq:locpol} is the solution to the weighted least squares problem
\begin{equation*}
	\widehat{\mathrm{LP}}(x,y) = \argmin_{\bs \vartheta \in \R^{N_m} } \big( - 2\bs \vartheta^\top \bs a_{p,h}(x,y) + \bs \vartheta^\top B_{p,h}(x,y) \bs \vartheta \big).
\end{equation*}
The solution is determined by the normal equations
\begin{equation*}
	\bs a_{p,h}(x,y) = B_{p,h}(x,y)\bs \vartheta.
\end{equation*}
In particular, if $B_{p, h}(x, y)$ is positive definite, the solution in \eqref{eq:locpol} is uniquely determined and we obtain
\begin{equation*}
	\widehat{\mathrm{LP}}_1(x,y) =  \sum_{j<k}^{p} \wjk xyh \, z_{j,k}
\end{equation*}
with (now stating both cases)
\begin{align} \label{eq:weights:locpol}
	\wjk xyh \defeq \frac1{(p\,h)^2} U^\top\vecTwo 00 \cdot\left\{ \begin{array}{cl} B_{p,h}^{-1}(x,y) \, \Uh{x_j-x}{x_k-y}\,  \Kh{x_j-x}{x_k-y},\; \text{for}\; x\leq y,\\
		B_{p,h}^{-1}(y,x) \, \Uh{x_j-y}{x_k-x}\,  \Kh{x_j-y}{x_k-x},\; \text{for}\; x> y,
	\end{array}\right.
\end{align}
so that the local polynomial estimator is a linear estimator. \\

The more challenging part is to prove (LP1) now for a sufficiently large number $p$ of design points and an uniformly choice of the bandwidth $h$. 

\begin{lemma}\label{lemma:proof:(LP1)}
	Suppose that the kernel $K$ suffices \eqref{eqn:kernel} in Lemma \ref{lemma:locpol:weights} and Assumption \ref{ass:designdensity} is satisfied. Then there exist a sufficiently large $p_0 \in \N$ and a sufficiently small $h_0 \in \R_+$ such that for all $p\geq p_0$ and $h\in(c/p, h_0]$, where $c \in \R_+$ is a sufficiently large constant, the smallest eigenvalues $\lambda_{\min} \big(B_{p, h}(x, y)\big)$ of the matrices $B_{p, h}(x,y)$, which are given in \eqref{eq:Bp}, are bounded below by a universal positive constant $\lambda_0 >0$ for any $x, y\in [0,1]$ with $x<y$. 
\end{lemma}

An immediate consequence of Lemma \ref{lemma:proof:(LP1)} is the invertibility of $B_{p, h}(x,y)$ for all $p\geq p_0$, $h\in( c/ p, h_0]$ and $ x<y, x, y \in [0,1]$, and hence also the uniqueness of the local polynomial estimator for these $p$ and $ h$. In \citet[Lemma 1.5]{tsybakov2008introduction} the lower bound for the smallest eigenvalues has only be shown for a fixed sequence $h_p$ (for $d=1$) of bandwidths which satisfies $h_p \to 0$ and $p \, h_p \to \infty$. In contrast, we allow an uniformly choice of $h$ which results, in particular, in the findings of Section \ref{sec:rate_cov_estimation}.

\begin{proof}[Proof of Lemma \ref{lemma:proof:(LP1)}] 
	In the following let $ v \in \R^{N_{m}}$. We show that there exist a sufficiently large $ p_0 \in \N$ and a component wise sufficiently small $ h_0\in\R_+$ such that the estimate 
	\begin{equation}\label{eq:inf:Bp}
		\inf_{ p\geq  p_0} \inf_{ h\in( c/ p,  h_0]} \inf_{ \substack{x,y \in [0,1]\\ x<y}} \inf_{\norm{ v}_2=1}  v^\top B_{ p,  h}( x,y) \,  v \geq \lambda_0 
	\end{equation}
	is satisfied. Then we obtain for these choices of $ p$ and $ h$, and any $ x,y \in [0,1]$ also
	\begin{align*}
		\lambda_{\min}\big(B_{ p,  h}(x,y)\big) &= e_{\min}^\top \big(B_{ p,  h}(x,y)\big) B_{ p,  h}(x,y) \,  e_{\min}\big(B_{ p,  h}(x,y)\big) \geq \inf_{\norm{ v}_2=1}  v^\top B_{ p,  h}(x,y) \, v \geq \lambda_0\,,
	\end{align*} 
	where $ e_{\min}\big(B_{ p,  h}(x,y)\big) \in \R^{m+1}$ is a normalised eigenvector of $\lambda_{\min}\big(B_{ p,  h}(x,y)\big)$.
    Let $E^1\defeq \{(x,y)^\top \in \R^2 \mid x,y\in [0,\Delta), x \leq y\}, E^2 \defeq \{(x,y)^\top \in \R^2 \mid x\in(-\Delta, 0], y\in [0,\Delta) \}$ and $E^3 \defeq \{(x,y)^\top \in \R^2 \mid x,y\in (-\Delta,0], x \leq y\}$, where $\Delta \in \R^+$ is given in \eqref{eqn:kernel}. We set
	\begin{equation*}
		\lambda_i({v}) \defeq f_{\min}^2 \, K_{\min} \int_{E^i} \big({v}^\top U_m(\bs z)\big)^2\dx \bs z \,, \quad \quad \lambda_i \defeq \inf_{\norm{v}_2 =1} \lambda_i({v})
	\end{equation*}
	for all $i=1,2,3$. Applying \citet[Lemma 1]{tsybakov2008introduction} with $K(\bs z) = \one_{E^i}(\bs z), \bs z \in \R^2,$ leads to $\lambda_i({v})\geq \lambda_i >0$.
	Therefore we find a $\lambda_0>0$ such that $\min(\lambda_1, \lambda_2, \lambda_3)>\lambda_0>0$, e.g. $\lambda_0\defeq \min(\lambda_1, \lambda_2, \lambda_3)/2$. Now we want to specify  a partition $S_1\,\cup\, S_2 \,\cup\, S_3= \{ (x,y)^\top \mid 0 \leq x \leq y \leq 1\}$ and functions $A^{(1)}_{p,h}(x,y;v),A^{(2)}_{p,h}(x,y;v)$ and $A^{(3)}_{p,h}(x,y;v)$ such that 
	\begin{align}\label{eq:vBv:A}
	{v}^\top B_{ p,  h}(x,y) \, {v} \geq A^{(i)}_{ p, h}(x,y;v)\,,\quad  (x,y)^\top \in S_i \,,\quad i =1,2,3,
	\end{align} 
	and
	\begin{align}\label{eq:glm:vBv}
		\sup_{(x,y)^\top\in S_i} \sup_{\norm{v}_2 =1} \big|A^{(i)}_{ p, h}( x,y;v) - \lambda_i({v})\big| \leq\frac {c_1}{p\,h}\,,\quad i =1,2,3, 
	\end{align}
	hold true for a positive constant $c_1>0$. Consider $ h\in( c/ p,  h_0]$ with $c\defeq c_1/(\min(\lambda_1, \lambda_2, \lambda_3)-\lambda_0)$ yields $c_1/(p\,h) < c_1/c \leq \lambda_i-\lambda_0$ for all $i =1,2,3$. Hence, it follows by \eqref{eq:vBv:A} and \eqref{eq:glm:vBv} that
		\begin{align*}
		\inf_{(x,y)^\top\in S_i} \inf_{\norm{v}_2 =1} {v}^\top B_{ p,  h}(x,y) \, {v}&=\inf_{(x,y)^\top\in S_i} \inf_{\norm{v}_2 =1}\big(\lambda_i({v})+{v}^\top B_{ p,  h}(x,y) \,{v}-\lambda_i({v})\big)\\
		&\geq\lambda_i+\inf_{(x,y)^\top\in S_i} \inf_{\norm{v}_2 =1} \big(A^{(i)}_{p,h}(x,y;v)-\lambda_i({v})\big)\\
		&=\lambda_i-\sup_{(x,y)^\top \in S_i} \sup_{\norm{v}_2 =1} \big(\lambda_i({v})-A^{(i)}_{p,h}(x,y;v)\big)\\
		&\geq\lambda_i-\sup_{(x,y)^\top \in S_i} \sup_{\norm{v}_2 =1} \big|A^{(i)}_{p,h}(x,y;v)-\lambda_i({v})\big|\\
		&\geq\lambda_i-\frac {c_1}{p\,h} \geq \lambda_0\,,
	\end{align*}
	which leads to \eqref{eq:inf:Bp} because of
	\begin{align*}
		\inf_{{p}\geq {p}_0} \inf_{ h\in(c/{p}, h_0]}& \inf_{0 \leq x \leq y \leq 1} \inf_{\norm{v}_2 =1} {v}^\top B_{ p,  h}(x,y) \, {v} \\
		&= \inf_{{p}\geq {p}_0}\inf_{ h\in(c/{p}, h_0]}\min_{i=1,2,3} \Big(\inf_{(x,y)^\top\in S_i} \inf_{\norm{v}_2 =1} {v}^\top B_{ p,  h}(x,y) \, {v} \Big) \geq \lambda_0 \,.
	\end{align*}

	\medskip
	
	Next we show \eqref{eq:vBv:A} and \eqref{eq:glm:vBv}. Let $I_0 \defeq \big[0,1-\Delta h - 1/(2f_{\min} \,p) \big]$ and $I_1 \defeq \big[1-\Delta h - 1/(2f_{\min} \,p),1 \big]$. We define $S_1\defeq \{ (x,y)^\top \in I_0^2 \mid  x \leq y\},\, S_2 \defeq \{ (x,y)^\top \in I_0 \times I_1\}$ and $ S_3
 \defeq \{ (x,y)^\top \in I_1^2 \mid x \leq y\}$. It is clear that $\bigcup_{i=1}^{3}S_i=\{(x,y)^\top \mid 0 \leq x \leq y \leq 1\}$. Define $\tilde x_j\defeq (x_j-x)/h$, $x \in [0,1]$ and $\tilde y_j \defeq (x_j - y)/h$ for all $1 \leq j \leq p$, where $x_1,\ldots,x_{p}$ are the design points.\\
    We have to differentiate two different cases. At first let $x \in I_0$. Then we get
	\begin{align*}
	\tilde x_1 & \leq \frac{1}{2f_{\min}  \, p\,h} - \frac{x}{h} \leq \frac{1}{2f_{\min}  \, p\,h} \,, \\
	\tilde x_{p}&\geq \frac{1-\frac{1}{2f_{\min}  \, p}-x}{h} \geq \frac{1-\big(1-\Delta h - (2f_{\min} \,p)^{-1}\big)}{h} - \frac 1{2f_{\min}  \, p\,h} = \Delta
	\end{align*} 
	by \eqref{lemma:design:points:auxiliary:1} in Lemma \ref{lemma:design:points:auxiliary}. For appropriate $ p$ and $ h\in ( c/ p,  h_0]$ the quantity $(2f_{\min} \,p\,h)^{-1}$ gets small if $c$ is chosen large enough. Consequently, the points $\tilde x_1,\ldots,\tilde x_p$ form a grid which covers an interval containing at least $\big[1/(2f_{\min} \,p\,h), \Delta \big]$.
	Hence, there exist $1 \leq j_{*} \defeq j_{*}(x) <  j^{*}\defeq j^{*}(x) \leq p$ such that $\tilde{x}_{j_{*}}\geq 0 \;\land \; (j_{*}=1\;\lor\; \tilde{x}_{j_{*}-1}<0)$ and $\tilde{x}_{j^{*}} \leq \Delta \;\land\;\tilde{x}_{j^{*}+1}>\Delta$	are satisfied.  
	Here $\land$ denotes the logical and, and $\lor$ the logical or. For the second case let $x\in I_1$. Then we obtain the estimates 
	\begin{align*}
		\tilde{x}_1& \leq \frac1{2f_{\min} p\,h}- \frac{x}{h} \leq \Delta -\frac1{h}\leq -\Delta\,,&\tilde{x}_{p} & \geq -\frac{1}{2f_{\min} p\,h}\,,
	\end{align*}
	since $h$ has to be sufficiently small and $p$ sufficiently large. Consequently, the points $\tilde{x}_1,\ldots,\tilde{x}_{p}$ form a grid which covers an interval containing $[-\Delta, -1/(2f_{\min} p\,h)]$. Define $\tilde j_*$ and $\tilde j^*$ in the same manner. Of course this also holds for $\tilde y_1, \ldots, \tilde y _p$. \\
	We set $g_v(x,y)\defeq({v}^\top U(x, y) )^2, v\in \R^{N_m}$. For $0 \leq x \leq y \leq 1$ we can further estimate
	\begin{align*}
		{v}^\top B_{ p,  h}(x,y) \, {v}&= \frac1{(p\,h)^2} {v}^\top \bigg(\sum_{j<k}^p U_m\vecTwo{\tilde{x}_j}{\tilde{y}_k} \, U_m^\top\vecTwo{\tilde{x}_j}{\tilde{y}_k} \, K\vecTwo{\tilde{x}_j}{\tilde{y}_k} \bigg) {v} \nonumber\\
		&\geq \frac{f_{\min}^2K_{\min}}{f_{\min}^2 \, (p\,h)^2}\sum_{j<k}^p \Big({v}^\top U_m\vecTwo{\tilde{x}_j}{\tilde{y}_k} \Big)^2 \one_{[- \Delta,  \Delta]^2 } \big( \tilde x_j, \tilde y_{ k} \big) \nonumber\\
		&\geq f_{\min}^2K_{\min}\sum_{j<k}^{p-1} g_v\vecTwo{\tilde{x}_j}{\tilde{y}_k} \big(\tilde{x}_{j+1}-\tilde{x}_j\big)\big(\tilde{y}_{k+1} - \tilde{y}_k\big)\one_{[ 0,  \Delta)^2}( \tilde{x}_j,\tilde{y}_k)\\
		&\geq f_{\min}^2K_{\min} \sum_{j = j_*}^{j^*} \sum_{\substack{ k = k_*\\ j<k}}^{k^*} g_v\vecTwo{\tilde{x}_j}{\tilde{y}_k} \big(\tilde{x}_{j+1}-\tilde{x}_j\big)\big(\tilde y_{k+1} - \tilde{y}_k\big)\\
		&\defeql A^{(i)}_{p,h}(x,y;v), \quad i = 1,2,3,
	\end{align*}
	for $(x,y)^\top \in S_1$ by Assumption \ref{ass:designdensity}, inequality \eqref{lemma:design:points:auxiliary:2} in Lemma \ref{lemma:design:points:auxiliary} for the $\tilde{x}_j, \tilde{y}_k$ and the presentation of $B_{ p,  h}(x,y)$ in \eqref{eq:Bp}.
	We start with the case $ (x,y)^\top \in S_1=I_0\times I_0$. 
	By inserting this function in \eqref{eq:glm:vBv}, dropping the scalar $f_{\min}^2K_{\min}$ and oppressing the sups the object of interest is given by
\begin{align}
		\bigg| \sum_{j = j_*}^{j^*} &\sum_{\substack{ k = k_*\\ j<k}}^{k^*}  g_v\vecTwo{\tilde x_j}{\tilde{y}_k}\big(\tilde{x}_{j+1}-\tilde x_j\big)\big(\tilde{y}_{k+1} - \tilde{y}_k\big) -\int_{E_1} g_v(\bs z)\,\dx \bs z \bigg| \label{eq:integral:approximation}
        %
        %
	\end{align}
	In order to bound these terms of the sum separately we note that $g_v(z)=\big({v}^\top U_m( z) \big)^2$ is a bivariate polynomial function and therefore Lipschitz-continuous on $E_1, E_2$ and $E_3$ respectively. By Lemma \ref{lemma:design:points:auxiliary} it holds $\abs{\tilde x_{j+1} - \tilde x_j} \leq (f_{\min} \,p\,h)^{-1}$. Therefore the difference of every point of $z = (z_1, z_2)^\top \in E_1$ to a design point $(\tilde x_j, \tilde y_k)$ is of order $(p\,h)^{-1}$. Figure \ref{fig:illustration:integral:approximation} illustrates this step. 
	\begin{figure}
		\begin{tikzpicture}[scale=.9]
			\begin{axis}[
				width = 8.5cm, height = 8.5cm,
				axis lines =middle,
				axis line style={thin},
				xmin=0,xmax=1,ymin=0,ymax=1,
				xtick ={.1,.19,.27,.33,.4, .48, .56, .66, .78, .9},
				ytick ={0,.1,.19,.27,.33,.4, .48, .56, .66, .78, .9},
				yticklabels={$-y/h$,$\tilde y_1$,,,,,$\tilde y_{k_\star}$,,,$\tilde y_{k^\star}$,$\tilde y_p$},
				xticklabels={$\tilde x_1$,,$\tilde x_{j_\star}$,,,,$\tilde x_{j^\star}$,,,$\tilde x_p$},
				grid = major,
				grid style={thin,densely dotted,black!20},
				extra x ticks = {0}, extra x tick labels = {$-\nicefrac xh$},
				extra y ticks = {0}, extra y tick labels = {$-\nicefrac yh$}]
				\addplot [mark = none, thin] coordinates {(0,0)(0,1)(0,0)(1,0)};
				\addplot [black, mark = none, ultra thin] coordinates {(0,0)(1,1)};
				\addplot +[->, black, mark=none, thin] coordinates {(.25, .44) (.25, .97)};
				\addplot +[->, black, mark=none, thin] coordinates {(.24, .45) (.77, .45)};
				\addplot +[black, mark=none, ultra thin] coordinates {(.25, .45) (.75, .95)};
				\addplot[black, mark = -, only marks] coordinates {(0.25,.5)(0.25,.545)(0.25,.585)(0.25,.615)(0.25,.65)(0.25,.69)(0.25,.73)(0.25,.78)(0.25,.84)(0.25,.9)(0.25,.95)};
				\addplot[black, mark = +, only marks] coordinates {(.3,.45)(.345,.45)(.385,.45)(.415,.45)(.45,.45)(.49,.45)(.53,.45)(.57,.45)(.63,.45)(.69,.45)(.75,.45)};
				\addplot[black, mark = Mercedes star, only marks] coordinates{(.27,.48)(.27,.56)(.27,.66)(.27,.78)
					(.33,.56)(.33,.66)(.33,.78)(.4,.66)(.4,.78)(.48,.78)(.56,.78)};
				\addplot[red, mark = |] coordinates {(.6,.45)} node [below] {\small $\Delta$};
				\addplot[red, mark = -] coordinates {(.25,.8)} node [left] {\small $\Delta$};
				\addplot[red, mark = none] coordinates {(.25,.45)(.25,.8)(.6,.8)(.25,.45)};
				\addplot[black,no marks] coordinates{(.258,.458)} node[below left] {\small $0$};
				\addplot[black,no marks] coordinates{(.75,.45)} node[below] {\small $1$};
				\addplot[black,no marks] coordinates{(.25,.95)} node[left] {\small $1$};	
				\addplot[black, no marks] coordinates{(0.25,.5)} node [left] {$x_1$};
				\addplot[black, no marks] coordinates{(0.25,.9)} node [left] {$x_p$};
				\addplot[black, no marks] coordinates{(0.3,.45)} node [below] {$x_1$};
				\addplot[black, no marks] coordinates{(0.69,.45)} node [below] {$x_p$};
				\addplot[red, no marks] coordinates {(.45,.6)} node {$E_1$};
		\end{axis}
		\end{tikzpicture}
		\hspace{3mm}
		\begin{tikzpicture}[scale=.9]
			\begin{axis}[
					width = 8.5cm, height = 8.5cm,
					axis lines=middle,
					axis line style={thin},
					xmin=0,xmax=1,ymin=0,ymax=1,
					xtick ={.1,.19,.27,.33,.4, .48, .56, .66, .78, .9},
					ytick ={.1,.19,.27,.33,.4, .48, .56, .66, .78, .9},
					yticklabels={$\tilde y_1$,,,,,$\tilde y_{k_\star}$,,,$\tilde y_{k^\star}$,$\tilde y_p$},
					xticklabels={$\tilde x_1$,,$\tilde x_{j_\star}$,,,,$\tilde x_{j^\star}$,,,$\tilde x_p$},
					grid=major,
					grid style={thin,densely dotted,black!20},
					extra x ticks = {0}, extra x tick labels = {$-\nicefrac xh$},
					extra y ticks = {0}, extra y tick labels = {$-\nicefrac yh$}]
					\addplot [mark = none, thin] coordinates {(0,0)(0,1)(0,0)(1,0)};
					\addplot [black, mark = none, ultra thin] coordinates {(0,0)(1,1)};
					\addplot +[<-, black, mark=none, thin] coordinates {(.6, .28) (.6, .81)};
					\addplot +[<-, black, mark=none, thin] coordinates {(.08, .8) (.61, .8)};
					\addplot +[black, mark=none, ultra thin] coordinates {(.1, .3) (.6, .8)};
					\addplot[black, mark = Mercedes star, only marks] coordinates{(.27,.48)(.27,.56)(.27,.66)(.27,.78)
						(.33,.56)(.33,.66)(.33,.78)(.4,.66)(.4,.78)(.48,.78)(.56,.78)};
					\addplot[red, mark = -] coordinates {(.6,.45)} node [right] {\small $-\Delta$};
					\addplot[red, mark = |] coordinates {(.25,.8)} node [above] {\small $-\Delta$};
					\addplot[red, mark = none] coordinates {(.25,.45)(.25,.8)(.6,.8)(.25,.45)};
					\addplot[black,no marks] coordinates{(.6,.8)} node[above right] {\small $0$};
					\addplot[black,mark = -] coordinates{(.6,.3)} node[right] {\small $-1$};
					\addplot[black,mark = |] coordinates{(.1,.8)} node[above] {\small $-1$};	
					\addplot[red, no marks] coordinates {(.45,.6)} node {$E_3$};
				\end{axis}
		\end{tikzpicture}
		\caption{Illustration to the approximation in \eqref{eq:integral:approximation}. On the left the first and on the right the third case.}
		\label{fig:illustration:integral:approximation}
	\end{figure}
	The amount of points $(\tilde x_j, \tilde y_k)^\top, j_*\leq j \leq j^*, k_* \leq k \leq k^*, j < k,$ is bounded by Lemma \ref{lemma:design:points}
	\begin{align*}
		\sum_{j,k = 1}^p \ind_{E_1}(\tilde x_j, \tilde y_k) &\leq \Big(2f_{\max} \max \big(\Delta\,p\,h, 1 \big)\Big)^2 = \mc O \big((p\,h)^2\big).
	\end{align*}
	Therefore the sum approximates the integral in form of a Riemann-sum of a Riemann-integrable, since it is Lipschitz-continuous, function with the error rate getting small with order $(p\,h)^{-1}$. This concludes the proof via the previous steps.
\end{proof}

Now we can prove Lemma \ref{lemma:locpol:weights}.

\begin{proof}[Proof of Lemma \ref{lemma:locpol:weights}]
	By Lemma \ref{lemma:proof:(LP1)} there exist a sufficiently large $ p_0 \in \N$ and a sufficiently small  $ h_0\in \R_+$ such that for all $ p\geq   p_0$ and $ h\in( c/ p,  h_0]$, where $c>0$ is a sufficiently large constant such that the local polynomial estimator in \eqref{eq:locpolest} with any order $m \in \N$ is unique and a linear estimator with weights given in \eqref{eq:weights:locpol}. We make use of Lemma \ref{lemma:proof:(LP1)} by 
	\begin{align} \label{eq:inverseBp:2norm}
	\norm{B_{ p, h}^{-1}(x,y)}_{M,2} \leq \lambda_0^{-1}
	\end{align}
	for all $ p\geq  p_0$, $ h\in( c/ p,  h_0]$ and $ x \leq y \in [0,1]$, where $\norm{M}_{M,2}$ is the spectral norm of a symmetric matrix $M \in \R^{N_{m} \times N_{m}}$.\\
    \ref{ass:weights:polynom}: Let $Q$ be a bivariate polynomial of degree $\zeta\geq 1$ with $r_1 + r_2 = \zeta$, $r_1, r_2 \in \N$. By the bivariate Taylor expansion we have
    \begin{align*}
        Q(x_j, x_k) & = Q(x,y) + \sum_{i = 1}^{r_1} \sum_{l = 1}^{r_2} \frac{(x_j - x)^i\,(x_k-y)^l}{i!\, l!}\cdot \frac{\partial^{i+l}}{\partial^i x \partial^l y}Q(x_j, x_k) = q(x,y)^\top U_h\vecTwo{x_j-x}{x_k-y},
    \end{align*}
    where $q(x,y) \defeq (q_0(x,y),\ldots,q_\zeta(x,y) )^\top, $ with $$q_l(x,y)\defeq\bigg(\frac{\partial^l}{\partial^l x}\, \frac{q(x,y)}{l!}, \frac{\partial^l}{\partial^{l-1}x\partial y}\, \frac{q(x,y)}{(l-1)!},\ldots,\frac{\partial^l}{\partial^l y}\,\frac{q(x,y)}{l!}\bigg)\,.$$ Setting $z_{j,k} = Q(x_j, x_k)$,  we get for $x \leq y$
    \begin{align*}
        \widehat{\mathrm{LP}}(x,y) & = \argmin_{\vartheta \in \R^{N_\zeta}} \sum_{j < k}^p \bigg( Q(x_j, x_k) - \vartheta ^\top \Uh{x_j-x}{x_k-y}\bigg)^2 \Kh{x_j-x}{x_k-y}\\
        & = \argmin_{\vartheta \in \R^{N_\zeta}} \sum_{j < k}^p \bigg( \big(q(x, y) - \vartheta\big)^\top \Uh{x_j-x}{x_k-y}\bigg)^2 \Kh{x_j-x}{x_k-y}\\
        & = \argmin_{\vartheta \in \R^{N_\zeta}} \big(q(x, y) - \vartheta\big)^\top B_{p,h}(x,y)\big(q(x, y) - \vartheta\big) = q(x,y), 
    \end{align*}
    since $B_{p,h}(x,y)$ is positive definite. \\
    \ref{ass:weights:vanish}: This follows directly by the support of the kernel function $K$ and the display \eqref{eq:weights:locpol}.\\
    \ref{ass:weights:sup}: If design points are dispersed according to Assumption \ref{ass:designdensity}, then by Lemma \ref{lemma:design:points} and the fact $\norm{U_h(0,0)}_2 = 1$ we get
    \begin{align*}
        \abs{\wjk xyh} & \leq \frac1{(p\,h)^2}\normb{U_h(0,0)}_2 \normb{B_{p,h}^{-1}}_{M,2}\normm{\Uh{x_j-x}{x_k-y}\,\Kh{x_j-x}{x_k-y}}_{2}\\
        & \leq \frac{K_{\max}}{(p\,h)^2\,\lambda_0}\normm{\Uh{x_j-x}{x_k-y}}_2 \ind_{\max\{x_j-x,x_k-y\} \leq h} \\
        & \leq \frac{4\,K_{\max}}{(p\,h)^2\,\lambda_0}\,.
    \end{align*}
    \ref{ass:weights:lipschitz}: We divide the proof in three cases with respect to the fact whether the weights vanish or not. In the following let $ 1 \leq  j \leq  p$ and $ x, y \in [0,1]$.
	
	\medskip
	
	Let $\min \big(\max(\abs{x- x_{ j}} ,\abs{y- x_{ k}}), \max(\abs{y^\prime-x_k},\abs{y^\prime-x_k})\big) > h$, then by \ref{ass:weights:vanish} both weights $\wjk xyh$ and $\wjk{x^\prime}{y^\prime}h$ vanish, and hence \ref{ass:weights:lipschitz} is clear.  
	
	\medskip
	
	Let $\max \big(\max(\abs{x- x_{ j}} ,\abs{y- x_{ k}}), \max(\abs{y^\prime-x_k},\abs{y^\prime-x_k})\big) > h$. We assume $\max(x- x_{ j} ,y- x_{ j})>h $ and $\max(y-x_k,y-x_j) \leq h$ without loss of generality. Once again \ref{ass:weights:vanish} leads to $\wjk xyh = 0 $, and hence
	the Cauchy-Schwarz inequality, $\norm{U(0,0)}_2=1$ and \eqref{eq:inverseBp:2norm} imply
	\begin{align*}
	\Big| \wjk xyh - \wjk {x^\prime}{y^\prime}h \Big| &= \frac{1}{p\,h} \Big| U^\top(0,0) \, B_{ p, h}^{-1}(x,y) \, \Uh{x-x_j}{y-x_k} \Big| \,   \Abs{\Kh{x-x_j}{y-x_k}} \\
	&\leq \frac1{ p^{ 1}  h^{ 1}} \normb{U(0,0)}_2 \, \normm{B_{p,h}^{-1}(x,y) \, \Uh{x-x_j}{y-x_k}}_2 \,  \Abs{\Kh{x-x_j}{y-x_k}} \\
    &\leq \frac1{p\,h}  \norm{B_{ p, h}^{-1}(x,y)}_{M,2} \, \normm{\Uh{x-x_j}{y-x_k}}_2 \, \Abs{\Kh{x-x_j}{y-x_k}} \\
	&\leq \frac{1}{\lambda_0  \,  p\,h} \Abs{\Kh{x-x_j}{y-x_k}}  \bigg(\sum_{\substack{|r|=0,\\r=(r_1,r_2)\in\N^2}}^m \bigg(\frac{({x_j}-  x)^{r_1}(x_k-y)^{r_2}}{{h^r}  r_1!r_2!}\Big)^2\bigg)^{1/2} \\
	&\leq \frac{c_1}{\lambda_0 \,p\,h} \Abs{\Kh{x-x_j}{y-x_k}} 
	\end{align*}
	for a positive constant $c_1>0$. In the last step we used the fact that the sum can not get arbitrarily large, also for  component wise small $ h$, because the kernel $K$ has compact support in $[-1,1]$. If $\max(x-y,x^\prime-y^\prime)>h$, we use the upper bound of the kernel function in \eqref{eqn:kernel} and obtain
		\begin{align*}
		\Big| \wjk xyh - \wjk {x^\prime}{y^\prime}h \Big|  \leq \frac{c_1 \, K_{\max}}{\lambda_0 \,  p\,h} \,.
	\end{align*}
	Conversely, if $\max(x-y,x^\prime-y^\prime)\leq h$, we add $K_{h}(x_j- x^\prime, x_k-y^\prime) = 0$ and estimate
	\begin{align*}
		\Big| \wjk xyh - \wjk {x^\prime}{y^\prime}h \Big| & \leq \frac{c_1}{\lambda_0 \,  p^{ 1}  h^{ 1}} \Abs{\Kh{x-x_j}{y-x_k} -\Kh{x^\prime-x_j}{y^\prime-x_k}}\\
  &\leq  \frac{c_1 \, L_K}{\lambda_0 \,  p^{ 1}  h^{ 1}} \, \frac{\max(\abs{x-y},\abs{x^\prime -y^\prime})}{ h} \, ,
	\end{align*}
	because of the Lipschitz continuity of the kernel. In total we get
	\begin{align} \label{proof:locpol:weights:1}
		\Big| \wjk xyh - \wjk {x^\prime}{y^\prime}h \Big| \leq \frac{c_1 \, \max(K_{\max},L_K)}{\lambda_0 \, p\,h} \bigg(\frac{\max(\abs{x-y},\abs{x^\prime -y^\prime})}{ h} \wedge 1 \bigg) \,.
	\end{align}
	
	\medskip
	
	Let $\max \big(\max(\abs{x- x_{ j}} ,\abs{y- x_{ k}}), \max(\abs{y^\prime-x_k},\abs{y^\prime-x_k})\big) \leq h$, then both weights doesn't vanish and we have to show a proper Lipschitz property for
	\begin{align*} 
		\wjk xyh = \frac1{ p\,h} U^\top(0,0) \, B_{p,h}^{-1}(x,y) \, \Uh{x-x_j}{y-x_k}\,\Kh{x-x_j}{y-x_k}\,.
	\end{align*}
	The kernel $K$ and polynomials on compact intervals are Lipschitz continuous, hence it suffices to show that $B_{ p,  h}^{-1}(x,y)$ has this property as well. Then the weights are products of bounded Lipschitz continuous functions and, thus, also Lipschitz continuous. The entries of the matrix $B_{ p,  h}(x,y)$, which is defined in \eqref{eq:Bp}, considered as functions from $[0,1]$ to $\R$ are Lipschitz continuous. Indeed they are of order one by using Assumption \ref{ass:design:localization} and Lemma \ref{lemma:design:points}. Hence, the row sum norm $\norm{B_{ p,  h}(x,y)-B_{ p,  h}(x^\prime, y^\prime)}_{M,\infty}$ is a sum of these Lipschitz continuous functions and, in consequence, there exists a positive constants $L_\infty >0$ such that
	\begin{align*}
	\norm{B_{ p,  h}(x,y)-B_{ p,  h}(x^\prime, y^\prime)}_{M,\infty} \leq (N_{m}+1) L_\infty \, \frac{\max(\abs{x-y},\abs{x^\prime -y^\prime})}{ h}
	\end{align*}
	is satisfied. Since the matrices are symmetric the column sum norm is equal to the row sum norm and we obtain
	\begin{align*}
		\norm{B_{ p,  h}(x,y)-B_{ p,  h}(x^\prime, y^\prime)}_{M,2} \leq \norm{B_{ p,  h}(x,y)-B_{ p,  h}(x^\prime, y^\prime)}_{M,\infty} \leq (N_{m}+1) L_\infty \,\frac{\max(\abs{x-y},\abs{x^\prime -y^\prime})}{ h}\,.
	\end{align*}
	This leads together with \eqref{eq:inverseBp:2norm} and the submultiplicativity of the spectral norm to
	\begin{align*}
		\normb{ B_{ p,  h}^{-1}(x,y)-B_{ p,  h}^{-1}(x^\prime,y^\prime) }_{M,2} &= \normb{ B_{ p,  h}^{-1}(x^\prime,y^\prime) \big( B_{ p,  h}(x^\prime, y^\prime)-B_{ p,  h}(x,y) \big) B_{ p,  h}^{-1}(x,y)}_{M,2}\\
		&\leq \normb{ B_{ p,  h}^{-1}(x,y)}_{M,2} \, \norm{ B_{ p,  h}(x^\prime, y^\prime)-B_{ p,  h}(x,y)}_{M,2} \, \norm{ B_{ p,  h}^{-1}(x,y)}_{M,2} \\
		&\leq \frac{(N_{m}+1) L_\infty }{\lambda_0^2} \, \frac{\max(\abs{x-y},\abs{x^\prime -y^\prime})}{ h} \,,
	\end{align*}
	which is the Lipschitz continuity of $B_{ p,  h}^{-1}(x,y)$ with respect to the spectral norm. So finally there exists a positive constant $c_2>0$ such that 
	\begin{align*}
	 \Big| \wjk xyh - \wjk {x^\prime}{y^\prime}h \Big|  \leq  \frac{c_2}{2\, p\,h} \,\frac{\max(\abs{x-y},\abs{x^\prime -y^\prime})}{ h}\leq \frac{c_2}{ p\,h}
	\end{align*}
	is satisfied. Here we used in the last step that $\max \big(\max(\abs{x-x_j},\abs{y-x_j}), \max(\abs{y-x_k},\abs{y-x_j})\big) \leq h$ implies $\max(\abs{x-y},\abs{x
 ^\prime-y^\prime}) \leq 2\,h\,$.
	
	\medskip
	
	In the end we choose $C_3 \geq \max\big( c_1 \max(K_{\max},L_K)/\lambda_0, c_2\big)$ and obtain Assumption \ref{ass:weights:lipschitz}. 
\end{proof}

Here we state two auxiliary Lemmas which result as a consequence of Assumption \ref{ass:designdensity} and which were used in the proof of Lemma \ref{lemma:locpol:weights}. For their proofs see \cite[Proofs of Lemma 6 and 7]{berger2023dense}.

\begin{lemma} \label{lemma:design:points:auxiliary}
	Let Assumption \ref{ass:designdensity} be satisfied. Then for all $j = 1,\dotsc,p$ it follows that
	\begin{equation}
		\frac{j-0.5}{f_{\max}  \, p} \leq x_{j} \leq \frac{{j}-0.5}{f_{\min}  \, p} \quad \text{and} \quad 1-\frac{p-{j}+0.5}{f_{\min}  \, p} \leq x_{j} \leq 1-\frac{p-{j}+0.5}{f_{\max}  \, p} \, ,\label{lemma:design:points:auxiliary:1}
	\end{equation}
	and for all $1 \leq j < l \leq p$ that
	\begin{equation}
		\frac{l-j}{f_{\max}  \, p} \leq x_{l}-x_{j} \leq \frac{l-j}{f_{\min}  \, p} \,. \label{lemma:design:points:auxiliary:2}
	\end{equation}
\end{lemma}


\begin{lemma} \label{lemma:design:points}
	Suppose that Assumption \ref{ass:designdensity} is satisfied. Then we obtain for all $p \geq 1$ and any set $S = [a_1, b_1] \times [a_2 \times b_2] \subseteq [0,1]^2$ the estimate
	\[ \sum_{j, k = 1}^p \one_{\{(x_k, x_j)^\top \in S\}}\leq 4\,f_{\max}^2 \max\big(p(b_1-a_1), 1 \big)\,\max\big(p(b_2-a_2), 1 \big) \,. \]
	In particular Assumption \ref{ass:design:localization} is satisfied.
\end{lemma}

{\color{black}
\section{Derivative estimation}
\label{sec:derivativeestimation}

	For the weights of the linear derivative estimator in \eqref{eqn:estimatorCovariancederi} we require the following properties.
	
	\begin{assumption}\label{ass:biv:weights} 
		There is a $c>0$ and a $h_0>0$ such that for  sufficiently large $p$, the following holds for all $h \in (c/p_, h_0]$ for constants $\Cmax, \Clip>0$, which are independent of $n, p,h$ and $(x, y) \in T$. Then the weights $w_{j,k}^{(\bs s)}, 1 \leq j < k \leq p$ for a $\bs s = (s_1, s_2) \in \N_0^2$ satisfy the following assumptions.
		\begin{enumerate}[label=\normalfont{(C\arabic*)},leftmargin=9.9mm]
			\item For polynomials up to order $\zeta$ the weights reproduce the $\bs s^{\mathrm{th}}$ partial derivative. That is
			\begin{align*} 
				\sum_{j < k}^p \frac{(x_j - x)^{r_1}(x_j - y)^{r_2}}{r_1!\,r_2!}\,w_{j,k}^{(\bs s)}( x,y;h) = \delta_{(r_1,r_2), \bs s}, \quad (x,y)\in T,
			\end{align*} \label{ass:biv:weights:polynom}
			for $r_1,r_2 \in \N_0$ s.t. $r_1 + r_2 \leq \zeta$ and $\delta_{(r_1, r_2), \bs s} = 1$ if $(r_1, r_2) = \bs s$ and $0$ otherwise.
			\item We have $w_{j,k}^{(\bs s)}( x,y;h) = 0$ if $\max(\abs{x_j-x}, \abs{x_k - y})> h$ for $(x,y) \in T$. \label{ass:biv:weights:vanish}
			%
			%
			\item For the absolute values of the weights $ \displaystyle \max_{1\leq j<k\leq p} \big| w_{j,k}^{(\bs s)}( x,y;h)\big|  \leq \frac\Cmax{p^2\,h^{2+\abs{\bs s}}}$ for $ (x,y) \in T$.  \label{ass:biv:weights:sup}
			\item For a Lipschitz constant $\Clip > 0$  and $(x,y),\, (x^\prime, y^\prime) \in T$ it holds that
			\begin{align*}
				\absb{  w_{j,k}^{(\bs s)}( x,y;h) -  w_{j,k}^{(\bs s)}( x^\prime,y^\prime;h) }\leq \frac{\Clip}{ p^2\, h^{2+ \abs{\bs s}}} \min \bigg(\frac{\max(\abs{x-x^\prime}, \abs{x-y^\prime})}h, 1\bigg) \,.
			\end{align*} \label{ass:biv:weights:lipschitz}
			%
			
		\end{enumerate}		
	\end{assumption}
	
Given $\gamma > 0, C_Z>0, \,\beta> 0 $ and $L> 0$, we consider the class of processes 
\begin{align}
	\PclassCov &  = \big\{ Z\colon[0,1]\to \R \ \text{centered random process} \mid \text{a.s. differentiable paths} \nonumber \\ & \quad \quad \text{up to order $\leq \floor{\beta}$, } \exists \; \text{random variable $M$ s.th. $\norm{Z}_{\mc H, \beta;[0,1]} \leq M$ a.s.}\nonumber\\
	& \quad \quad \text{with } \expec[M^4] \leq C_Z \text{ and } \Gamma\in \mc H_{T}(\gamma, L). \big\}\label{eq:cov:classprocesses}
\end{align}

If $Z \in \PclassCov$ and  $l \leq \floor{\beta}$ we obtain
\begin{align}\label{eq:hoelder:ZZ}
	\normb{\partial^{l} Z(x)\partial^{l}Z(y) - \partial^{l}Z(x^\prime)\partial^{l}Z(y^\prime)}_\infty  & \leq 4\,M^2 \max\big(\abs{x-x^\prime}, \abs{y-y^\prime}\big)^{\min(1, \beta - l)}.
\end{align}

\begin{theorem}\label{thm:cov:rates}
	Consider model \eqref{eq:model} under Assumption \ref{ass:design:localization} and the distribution Assumption \ref{ass:distribution} for the errors with $Z \in \PclassCov$ with $\gamma, \beta, C_Z, L_\varGamma > 0$. Given $\bs s = (s_1, s_2) \in \N_0^2$ with $\abs{\bs s} \leq \floor{\gamma}$ consider the linear estimators $\hat \Gamma_n^{(\bs s)}$ of the partial derivative $D^{\bs s} \Gamma$ of the covariance kernel $\Gamma $ restricted to $T$ in \eqref{eqn:estimatorCovariancederi} with weights satisfying Assumption \ref{ass:biv:weights} with $\zeta = \floor \gamma$. Then for sufficiently large $p$ and $n$, 
	\begin{align*}
		\sup_{h \in (c/p, h_0]}\, \sup_{Z \in \PclassCov} a_{n,p,h}^{-1} \expec\Big[\normb{\hat \Gamma_{n}^{(\bs s)}(\cdot; h)  - D^{\bs s}\Gamma}_\infty\Big] = \mc O(1)\,,
	\end{align*}
	where 
	\begin{equation}\label{eq:cov:orderbound:deriv}
		a_{n, p, h} = \max\bigg( h^{\gamma - \abs{\bs s}}, \sqrt{\frac{\log(h^{-1})}{n\,p\,h^{1+2\abs{\bs s}}}}, \frac{1}{\sqrt n\, h^{\abs{\bs s}- \eta }}\bigg)\,,  
	\end{equation}
	with $\eta \defeq \min(\abs{\bs s}, \beta_1)$ for any $\beta_1$ with $ 0 < \beta_1 < \beta$. 
\end{theorem}

\smallskip

\begin{corollary}[Rates of convergence] Under the assumptions of Theorem \ref{thm:cov:rates} we have the following. Depending on the smoothness $\beta>0$ of the processes $Z \in \mc P(\gamma,\beta) =  \PclassCov$, 
\begin{enumerate}
	\item (Smooth processes) if $\beta > \abs{\bs s}$, then setting
		\begin{align}
			h^\star \simeq \max \bigg( \frac1{p}, \bigg( \frac{\log(np)}{n\,p}\bigg) ^{ \frac1{2\gamma + 1} } \bigg)\,,\label{eq:cov:deriv:h:star}
		\end{align}
		yields 
		\begin{align}
			\sup_{Z \in \mc P(\gamma,\beta)} \expec  \Big[\normb{\hat \Gamma_{n}^{(\bs s)}(\cdot; h^\star) - D^{\bs s}\Gamma}_\infty\Big] 
			= \mc O \bigg(\max \bigg(p^{-(\gamma -\abs{\bs s})},\bigg( \frac{\log(np)}{n\,p}\bigg) ^{ \frac{\gamma -\abs{\bs s}}{2\gamma + 1}}, n^{ - \frac12}  \bigg)\bigg)\,, \label{eq:cov:smooth:processes}
		\end{align}
	\item (Rough processes) if $\beta \leq \abs{\bs s}$, then for any fixed $0 < \beta_1 < \beta$, setting
		\begin{align}
		h^\star_{\beta_1} \simeq \max \bigg( \frac1p, \bigg( \frac{\log(np)}{n\,p}\bigg) ^{ \frac1{2\gamma + 1} },n^{- \frac{1}{2(\gamma - \beta_1)}}  \bigg)\,,\label{eq:cov:deriv:h:star:kappa}
	\end{align}
	yields 
	\begin{align}
		\sup_{Z \in \mc P(\gamma,\beta)} \expec  \Big[\normb{\hat \Gamma_{n}^{(\bs s)}(\cdot; h^\star_{\beta_1}) - D^{\bs s}\Gamma}_\infty\Big] 
		 = \mc O \bigg(\max \bigg(p^{-(\gamma -\abs{\bs s})},\bigg( \frac{\log(np)}{n\,p}\bigg) ^{ \frac{\gamma -\abs{\bs s}}{2\gamma + 1} },  n^{- \frac{\gamma -\abs{\bs s}}{2(\gamma - \beta_1)}}  \bigg)\bigg)\,. \label{eq:cov:rough:processes}
	\end{align}
\end{enumerate}
\end{corollary}

For the proofs and further discussion see \citet{berger2025diss}. 

\section{Asynchronous Design}\label{sec:asynchronousdesign}

Let the design points $\{x_{i,j}\}_{i = 1, \ldots, n, \,j = 1, \ldots, p}$ be different for each of the $n$ observed curves. For technical simplicity we assume that the amount of design points in each row is equal to $p$. Since it is not possible to take the average at the design points any more it is mandatory to estimate the mean function. For a pre-estimation $\hat \mu$ of the mean function $\mu$ let the estimator of the covariance kernel be given by

\begin{align}
    \hat \Gamma^{\mathrm{a}}_{n,p} (x,y;h) & = \frac{1}{n} \sum_{i = 1}^n \sum_{j <k}^p w_{i;j,k}(x,y;h) \big( Y_{i,j}- \hat \mu(x_{i,j})\big) \big( Y_{i,k}- \hat \mu(x_{i,k})\big)\,. \label{eq:cov_est_asynchron}
\end{align}

For the error decomposition we subtract the target function 
$$\Gamma(x,y) = \frac{1}{n}\sum_{i = 1}^n\sum_{j<k}^pw_{i;j,k}(x,y;h)\,\Gamma(x,y)$$
and with $\pm \frac{1}{n} \sum_{j<k}^pw_{i;j,k}(x,y;h)\Gamma(x_{i,j}, x_{i,k})$ we obtain the error terms
\begin{align}
    \big( Y_{i,j}- \hat \mu(x_{i,j}) \big( Y_{i,k}- &\hat \mu(x_{i,k})\big)  - \Gamma(x,y) \nonumber \\
     = &\;\big(\mu(x_{i,j}) - \hat \mu(x_{i,j})\big)\big(\mu(x_{i,k}) - \hat \mu(x_{i,k})\big) \tag{I} \label{term:I}\\
    &  + \big(\mu(x_{i,j}) - \hat \mu(x_{i,j})\big) Z_{i}(x_{i,k}) +  \big(\mu(x_{i,k}) - \hat \mu(x_{i,k})\big) Z_{i}(x_{i,j}) \tag{II}\label{term:II}\\
    &  + \big(\mu(x_{i,j}) - \hat \mu(x_{i,j})\big) \,\epsilon_{i,k} +  \big(\mu(x_{i,k}) - \hat \mu(x_{i,k})\big) \,\epsilon_{i,j} \tag{III}\label{term:III} \\
    & + Z_i(x_{i,j}) Z_i(x_{i,k}) - \Gamma(x_{i,j}, x_{i,k}) \tag{IV} \label{term:IV}\\
    & + \epsilon_{i,j}\epsilon_{i,k} \label{term:V} \tag{V} \\
    & + Z_i(x_{i,j})\epsilon_{i,k} + Z_i(x_{i,k}) \epsilon_{i,j} \tag{VI} \label{term:VI}\\
    & + \Gamma(x_{i,j}, x_{i,k}) - \Gamma(x,y)\,. \tag{VII} \label{term:VII}
\end{align}


\begin{theorem}
    Consider model \eqref{eq:model} with asynchronous design points under Assumption \ref{ass:distribution} and Assumption \ref{ass:design:localization} and the estimator $\hat \Gamma_{n,p}^a$ in \eqref{eq:cov_est_asynchron} for the covariance kernel $\Gamma$ with a pre-estimation $\hat \mu = \hat \mu^{(i)}$ of the mean function $\mu$ leaving out the $i^\mathrm{th}$ curve. Further we assume that the weights $\wijk xyh$ suffice Assumption \ref{ass:weights} uniformly for all $i = 1, \ldots, n$ with $\zeta = \floor\gamma$ and $h \in (c/p, h_0]$. Then for $0 \leq \beta_0 \leq 1$ and $L, C_Z > 0$ we have the following upper bounds
    \begin{align*}
        \sup_{h \in (c/p, h_0]}\, \sup_{Z \in \mc P (\gamma;L,\beta_0, C_Z)} \,\expec\big[a_{n,p,h}^{-1}\norm{\hat\Gamma_{n,p}^a- \Gamma}_\infty\big] = \mc O (1)\,,
    \end{align*}
    with 
    \begin{align*}
        a_{n,p,h} = \max\bigg(\expec\big[\norm{\hat \mu - \mu}_\infty^2\big], \,\expec\big[\norm{\hat \mu - \mu}_\infty\big]\bigg(\sqrt{\frac{\log(1/h)}{p\,h}} \vee 1 \bigg), h^{\gamma}, \sqrt{\frac{\log(h^{-1})}{n\,p\,h}}, n^{-1/2}\bigg)\,.
    \end{align*}
    If we assume the pre-estimation $\hat \mu$ to be fully independent of the data $Y_{i,j}, i = 1, \ldots, n, j = 1, \ldots, p$ used for the covariance kernel estimation (e.g. by sample splitting), then the rate of convergence is upper bounded by 
    \begin{align*}
        a_{n,p,h} = \max\bigg(\expec\big[\norm{\hat \mu - \mu}_\infty^2\big], h^{\gamma}, \big(1 + \sqrt{\expec\big[\norm{\hat \mu - \mu}_\infty^2\big]}\big)\sqrt{\frac{\log(h^{-1})}{n\,p\,h}}, n^{-1/2}\bigg)\,,
    \end{align*}
\end{theorem}



\begin{proof}
    The rate of \eqref{term:I} is bounded by $\expec\big[\norm{\hat \mu - \mu}_\infty^2\big]$, since
\begin{align*}
    \frac1{n}\sum_{ i = 1}^n\sum_{j < k}^p w_{i;j,k}(x,y;h) \big(\mu(x_{i,j}) - \hat \mu(x_{i,j})\big)\big(\mu(x_{i,k}) - \hat \mu(x_{i,k})\big) & \leq \,\frac{\mathrm{const.}}{n} \sum_{i = 1}^n \, \Csum\, \norm{\hat \mu - \mu}_\infty^2\,,
\end{align*}
by \ref{ass:weights:sum} and \citet[Theorem 1]{berger2023dense}.\\
For the term \eqref{term:II} we estimate
\begin{align*}
    \expec\bigg[\sup_{(x,y) \in T} &\bigg| \frac1{n} \sum_{i = 1}^n \sum_{j< k}^p \wijk xyh \big( \mu(x_{i,j}) - \hat \mu(x_{i,j})\big) \, Z_i(x_{i,k})\bigg|\bigg] \\
    & \leq\expec\bigg[\sup_{(x,y)\in T} \frac1{n} \sum_{i = 1}^n \sum_{j < k}^p \absb{\wijk xyh} \absb{\mu(x_{i,j}) - \hat \mu(x_{i,j})} \, \absb{Z_i(x_{i,k})}\bigg]\\
    & \leq \frac1{n} \sum_{i = 1}^n \expec\Big[\sup_{x \in [0,1]}\absb{\mu(x) - \hat \mu(x)}\Big]  \,\expec\Big[\sup_{x \in [0,1]}  \absb{Z_{i}(x)}\Big] \, \sup_{(x,y) \in T}\sum_{j < k}^p \absb{\wijk xyh} \\
    & \leq \expec\big[\norm{\hat \mu - \mu}_\infty\big]\, \expec\big[M_1 + Z_1(0)\big] \, \Cmax \, \Ccard\,.
\end{align*}
For the term \eqref{term:III} we proceed analogously while taking care of the sub-Gaussian error distribution. 
\begin{align*}
    \expec\bigg[\sup_{(x,y) \in T} &\Big| \frac1{n} \sum_{i = 1}^n \sum_{j< k}^p \wijk xyh \big( \mu(x_{i,j}) - \hat \mu(x_{i,j})\big) \,\epsilon_{i,k}\Big|\bigg] \\
    & \leq\expec\bigg[\sup_{(x,y)\in T} \frac1{n} \sum_{i = 1}^n \sum_{j = 1}^{p-1} \absb{\mu(x_{i,j}) - \hat \mu(x_{i,j})} \, \Abs{\sum_{k = j + 1}^p \wijk xyh \epsilon_{i,k}}\bigg]\\
    & \leq \frac1{n} \sum_{i = 1}^n \sum_{j = 1}^{p-1} \expec\Big[\sup_{x \in [0,1]}\absb{\mu(x) - \hat \mu(x)}\Big] \, \expec\bigg[\sup_{(x,y) \in T}\abss{\sum_{k = j + 1}^p \wijk xyh \epsilon_{i,k}} \bigg]\\
    & = \mc O \bigg( \expec\big[\norm{\hat \mu - \mu}_\infty\big]\, \sqrt{\frac{\log(1/h)}{p\,h}} \bigg)\,,
\end{align*}
where the rate $\expec[\sup_{x,y}\abs{\sum_{k = j + 1}^p \wijk xyh \epsilon_{i,k}}] = \sqrt{\log(1/h)/(p\,h)^3}$ for all $i = 1, \ldots, n,$ $ j = 1, \ldots, p-1$ follows analogously as the rate for the second error term in the mean function estimation in \citet[Theorem 1]{berger2023dense}. \\

The upper bounds for the term \eqref{term:IV} is more involved then the analysis for ii) of Lemma \ref{lem:rates:convergence}, because the maximal inequality from \citet{pollard1990empirical} needs to be applied to the smoothed versions of the processes. However in the proof for the asymptotic normality of the estimator we show that the smoothed version is also manageable. Using the fact that the weights suffice Assumption \ref{ass:weights} uniformly shows that the envelope is also given by \eqref{eq:envelope} and therefore the maximal inequality \citet[Section 7,display (7.10)]{pollard1990empirical} yields the $1/\sqrt n$ rate.\\

Term \eqref{term:V} and \eqref{term:VI} again rely on the fact that the weights suffice Assumption \ref{ass:weights} uniformly for all $i = 1, \ldots, n$ and therefore following the steps of the proofs ii) and iv) of Lemma \ref{lem:rates:convergence} lead to the rate $\sqrt{\log(n\,p)/(n\,p^2\,h^2)}$ and $\sqrt{\log(h^{-1})/(n\,p\,h)}$ respectively.\\

The discretization error in term \eqref{term:VII} becomes small with the rate $h^\gamma$ for each $i = 1, \ldots, n$ and therefore the mean of the errors also decreases with this rate. \\

Now, we assume that the pre-estimation $\hat \mu$ of the mean function $\mu$ is independent of the observations $Y_{i,j}$ used for the covariance estimation with the aim to show lower upper bounds for the rates of the terms \eqref{term:II} and \eqref{term:III}. First we consider term \eqref{term:II} which we denote by $T^{(2)}_{n,p}(x,y)$ such that for
\begin{align*}
    X_{n,i}(x,y) \defeq \frac{1}{\sqrt n} \sum_{j< k}^p \wijk xyh \big( \mu(x_{i,j}) - \hat \mu(x_{i,j})\big) Z_i(x_{i,k})
\end{align*}
we have $T^{(2)}_{n,p}(x,y) = n^{-1/2}\,\sum_{i = 1}^n\,X_{n,i}(x,y)$. For the conditioned version on  $\mu(x_{i,j}) - \hat \mu(x_{i,j}) = \Delta_{i,j}$ we write 
\begin{align*}
    X_{n,i}(x,y)_{\mid \mu(x_{i,j}) - \hat \mu(x_{i,j}) = \Delta_{i,j}} \defeq \frac{1}{\sqrt n} \sum_{j< k}^p \wijk xyh \,\Delta_{i,j}\, Z_i(x_{i,k})
\end{align*}
For given $\mu(x_{i,j}) - \hat \mu(x_{i,j}) = \Delta_{i,j}$ we can show that this process is manageable with respect to the envelope $\Phi_{n,\Delta} = (\phi_{n,1,\Delta}, \ldots, \phi_{n,p,\Delta})$ with
\begin{align*}
    \phi_{n,i,\Delta} = \frac{\Cmax\,\Ccard}{\sqrt n} \max_{j = 1, \ldots, p} \Delta_{i,j} \,\big( \abs{Z_i(0)} + M_i\big)\,,
\end{align*}
by using \ref{ass:weights:sum}. Now, following the steps of the proof of asymptotic normality in Theorem 2 of \citet{berger2023dense} we can use the maximal inequality of \citet[Section 7, display (7.10)]{pollard1990empirical} and Jensen to upper bound the conditional expected value with
\begin{align*}
    \expec\big[\sqrt n \norm{T^{(2)}_{n,p}}_\infty &\,\big|\,\{\mu(x_{i,j}) - \hat \mu(x_{i,j})\mid i = 1, \ldots, n, j = 1, \ldots, p\}\big] \\
    &\leq \tilde C_1 \, \Lambda(1)\, \Cmax\,\Ccard\,\bigg(\frac{1}{n^2} \sum_{i = 1}^n \max_{j = 1, \ldots, p} \Delta_{i,j}^2\,2\,\expec\big[\abs{Z_i(0)}^2 + M_i^2\big] \bigg)^{1/2}
\end{align*}
All in all we obtain with then Jensen inequality
\begin{align*}
    \expec\big[\norm{T_{n,p}^{(2)}}_\infty\big] & = \expec\Big[\expec\big[ \norm{T_{n,p}^{(2)}}_\infty\big|\,\{\mu(x_{i,j}) - \hat \mu(x_{i,j})\mid i = 1, \ldots, n, j = 1, \ldots, p\}\big] \Big] \\
    & \leq \tilde C_1 \,\Lambda(1) \frac{\Cmax\,\Ccard}{\sqrt n} \,\Big( \expec\big[\norm{\mu-\hat \mu}_\infty^2\big]\Big)^{1/2} \,2\,\Big(\expec\big[ \abs{Z_1(0)}^2 + M_1^2\big]\Big)^{1/2}\\
    & = \mc O \Bigg(\sqrt{\frac{\expec\big[\norm{\hat \mu - \mu}_\infty^2\big]}{n}}\Bigg)\,.
\end{align*}

For upper bound of term \eqref{term:III} let 
\begin{align*}
    T_{n,p,h}^{(3)}(x,y) & = \frac{1}{n}\sum_{i = 1}^n \sum_{j <k}^p \wijk xyh \big(\hat \mu(x_{i,j}) - \mu(x_{i,j})\big)\,\epsilon_{i,k}\,.
\end{align*}
We proceed similar to the proof of Lemma \ref{lem:rates:convergence} iv). Instead of conditioning on $Z_i(x_j) = z_{i,j}$ we condition on $\hat\mu(x_{i,j}) - \mu(x_{i,j})  = \Delta_{i,j}$. Setting 
$$ W \defeq \frac{\sigma^2 \,\Clip^2 \,\Ccard^2}{n^2\,p\,h} \sum_{i = 1}^n \max_{j = 1,\ldots, p} \Delta_{i,j}^2 $$
and following the exact same steps as in the referenced proof we obtain the bound
\begin{align*}
    \expec\big[\norm{T_{n,p,h}^{(3)}}_\infty \big] & = \expec\Big[\expec\big[\norm{T_{n,p,h}^{(3)}}_\infty \,\big|\, \{\hat\mu(x_{i,j}) - \mu(x_{i,j})  = \Delta_{i,j}, i = 1,\ldots, n, j = 1, \ldots, p\} \big]\Big] \\
    & \leq \expec\big[\sqrt{W}\big] + \expec\big[\sqrt{W}\big] \bigg( \sqrt{\log(h^{-1})} + \frac{1}{2\,\sqrt{\log(h)}}\bigg)\\
    & = \mc O\Bigg( \sqrt{\frac{\log(h^{-1})\,\expec[\norm{\hat\mu-\mu}_\infty^2]}{n\,p\,h}}\Bigg)
\end{align*}
since $\expec\big[\sqrt{W}\big]  = \mc O \big((\expec[\norm{\hat\mu-\mu}_\infty^2]/(n\,p\,h))^{-1/2}\,\big)$.
\end{proof}



}
\section{Additional numerical results.}
\label{sec:add_num_results}

\begin{figure}[h]
    \centering
    \includegraphics[width = \linewidth]{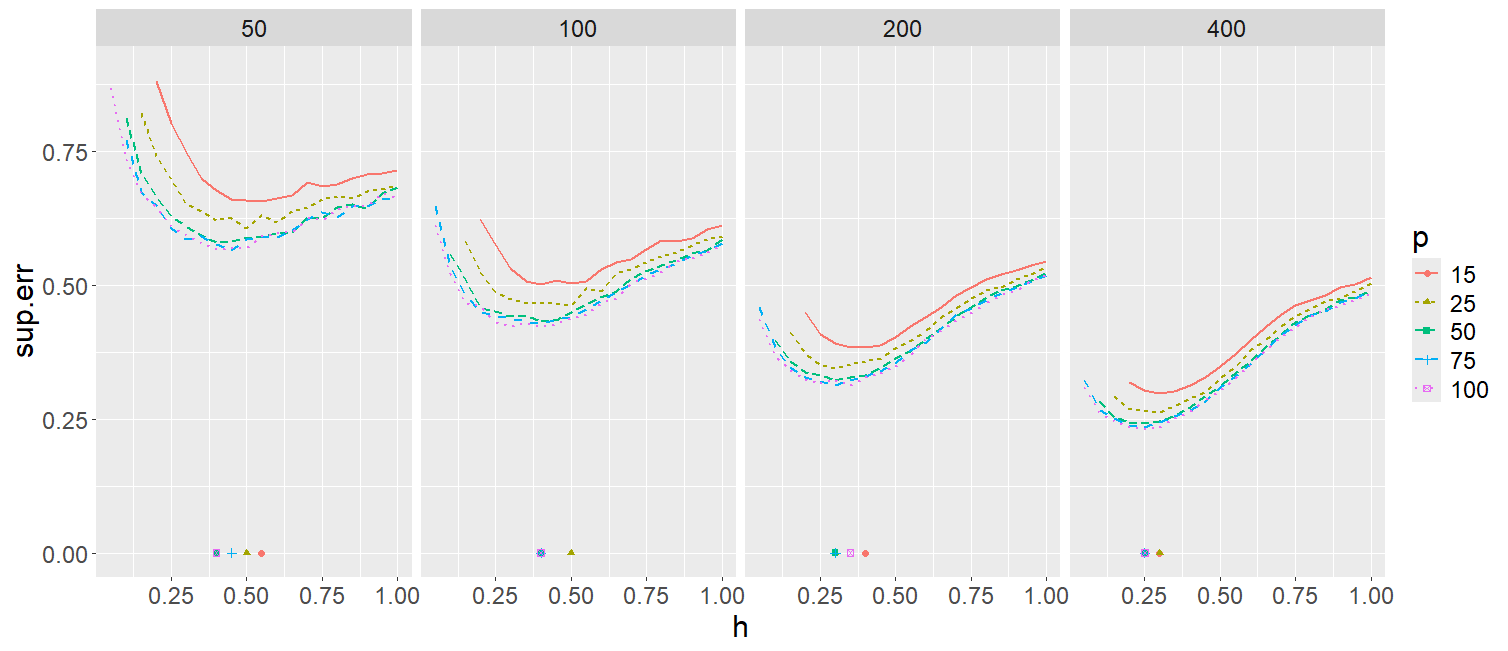}
    \caption{\small Supremum error of the estimation of the covariance kernel of an Ornstein Uhlenbeck process with $\theta = 3$ and $\sigma = 2$. The additional Gaussian noise has standard deviation $\sigma_\epsilon = 0.75$. The graphics show the supremum error for $n = 50, 100, 200, 400$ curves for different $p$. The points on the $x$-axis indicate which bandwidth led to the minimal empirical supremum error. }
    \label{fig:bandwidth_comparison_different_n}
\end{figure}

\begin{figure}[t!]
    \centering
    \includegraphics[width = \linewidth]{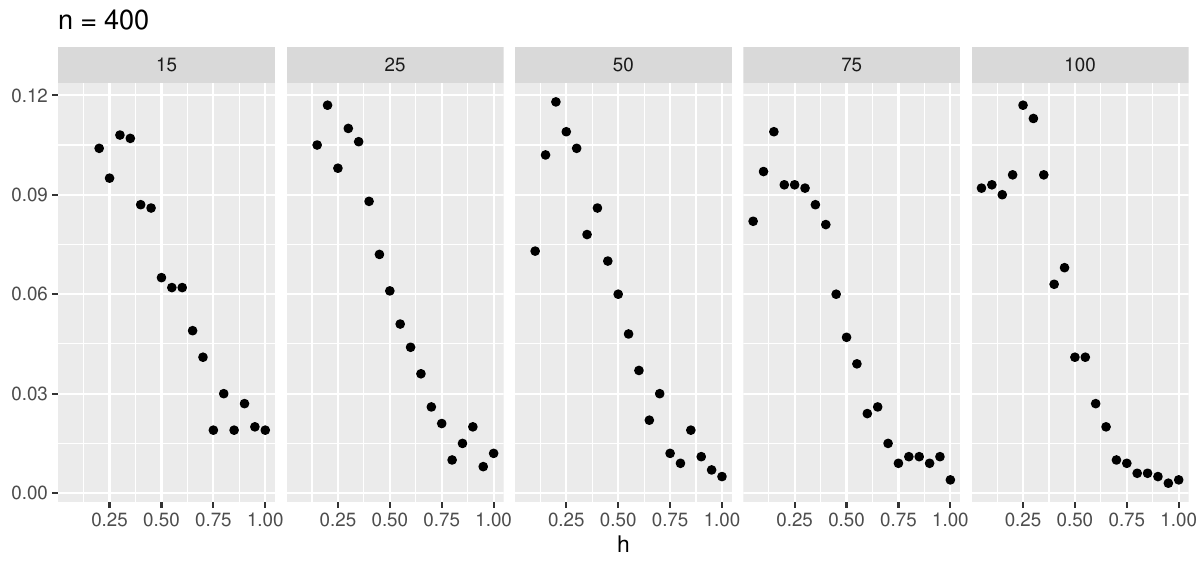}
    \caption{\small \color{black} Y-axis shows the percentage of cases in which the bandwidth on the x-axis was selected by the five fold cross validation as in Figure \ref{fig:5fold_boxplot}.}
    \label{fig:5fold_table}
\end{figure}

\begin{figure}[h]
    \centering
    \includegraphics[width = \linewidth]{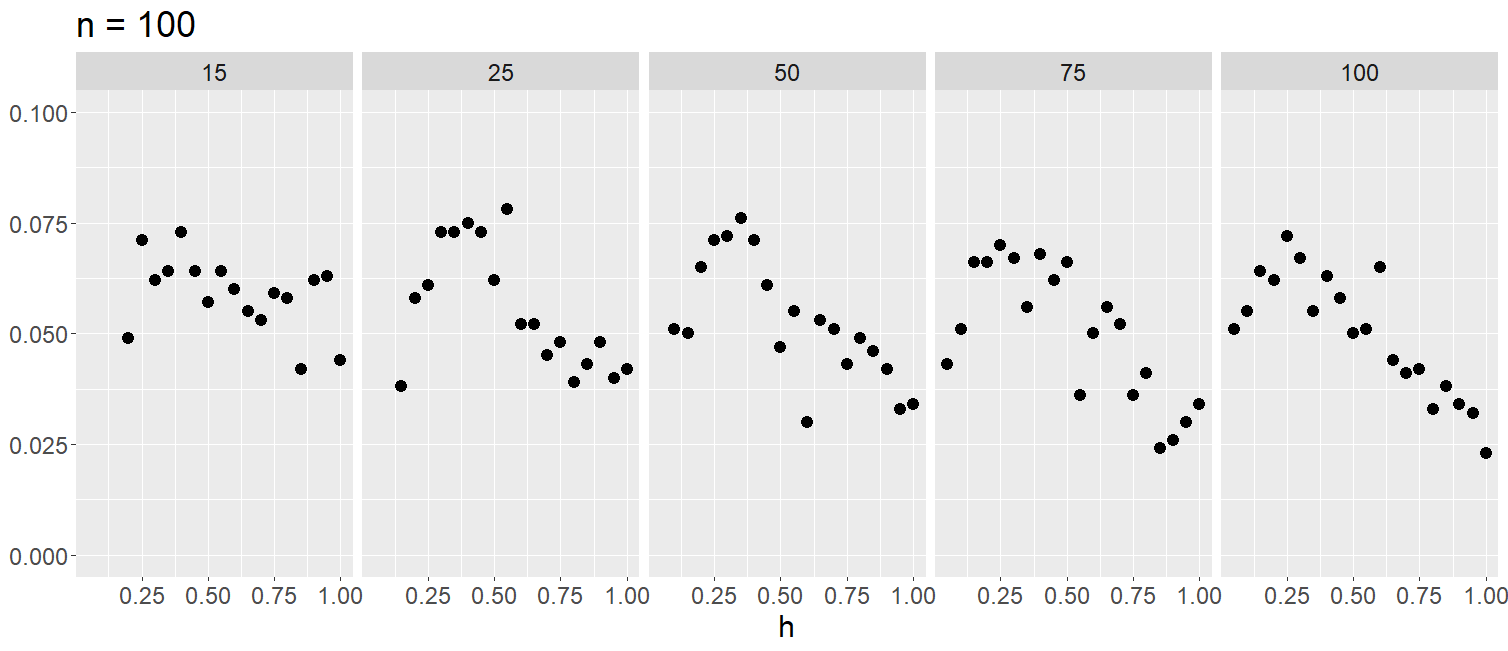}
    \caption{\small Results from $N = 1000$ times using the $5$-fold cross validation for $n = 100$ and various $p$. The procedure does not perform as well as for $n = 400$ in Figure \ref{fig:5fold_table}, which is explainable by Figure \ref{fig:bandwidth_comparison_different_n}.}
    \label{fig:enter-label}
\end{figure}

\end{document}